\documentclass[10pt,a4paper]{amsart}

\usepackage{lipsum}
\makeatletter
\g@addto@macro{\endabstract}{\@setabstract}
\newcommand{\authorfootnotes}{\renewcommand\thefootnote{\@fnsymbol\c@footnote}}%
\makeatother

\usepackage{todonotes} 
\usepackage{rotating}
\usepackage{float}
\usepackage[all,web,arc,poly]{xy}
\UseCrayolaColors
\usepackage{epsfig}
\usepackage{array}
\usepackage{pifont}
\usepackage{multirow}
\usepackage{graphics}
\usepackage{amsthm}
\usepackage{amsxtra}
\usepackage{amssymb}
\usepackage{mathtools}
\usepackage{mathrsfs}
\usepackage{mathptmx}
\usepackage[utf8]{inputenc}
\usepackage{bm}
\usepackage{bbm}
\usepackage{amsmath}
\usepackage{amsfonts}
\makeatletter
\def\amsbb{\use@mathgroup \M@U \symAMSb}
\makeatother

\usepackage{graphicx}
\usepackage{caption}
\usepackage{subcaption}

\usepackage{bbold}

\usepackage{mathptmx}

 \usepackage{dutchcal}

\usepackage{newtxtext,newtxmath}
\usepackage{pgfplots}
\pgfplotsset{compat=1.15}

\usepackage{xcolor}
\definecolor{webgreen}{rgb}{0,.5,0}
\definecolor{webbrown}{rgb}{.6,0,0}
\definecolor{RoyalBlue}{cmyk}{1, 0.50, 0, 0}
\usepackage[colorlinks=true, breaklinks=true, urlcolor=webbrown, linkcolor=RoyalBlue, citecolor=webgreen,backref=page]{hyperref}

\usepackage{tikz}
\usepackage{pict2e}
\usepackage{todonotes}
\usetikzlibrary{decorations.markings}

\DeclareSymbolFont{bbold}{U}{bbold}{m}{n}
\DeclareSymbolFontAlphabet{\mathbbold}{bbold}

\newcommand{\C}{{\mathbb C}}

\newcommand{\Z}{{\mathbb Z}}

\newcommand{\N}{{\mathbb N}}

\newcommand{\T}{{\mathbb T}}

\newcommand{\al}{\alpha}
\newcommand{\be}{\beta}

\newcommand{\ga}{\gamma}
\newcommand{\Ga}{\Gamma}
\newcommand{\La}{\Lambda}

\newcommand{\de}{\delta}

\newcommand{\ze}{\zeta}

\newcommand{\di}{\displaystyle}

\newcommand{\ic}{\textrm{i}}
\newcommand{\dd}{\textrm{d}}

\newcommand{\qandq}{\quad \text{and} \quad}
\newcommand{\tn}[1]{\textnormal{#1}}
\newcommand{\pr}[1]{\left( #1\right)}

\newtheorem{definition}{Definition}
\newtheorem{theorem}{Theorem}

\newtheorem{remark}[theorem]{Remark}
\newtheorem{lemma}[theorem]{Lemma}

\newtheorem{corollary}{Corollary}[theorem]
\numberwithin{equation}{section}
\numberwithin{theorem}{section}
\numberwithin{definition}{section}

\usepackage[mathscr]{euscript}

\newmuskip\pFqmuskip

\newcommand*\pFq[6][8]{%
	\begingroup 
	\pFqmuskip=#1mu\relax
	\mathchardef\normalcomma=\mathcode`,
	\mathcode`\,=\string"8000
	\begingroup\lccode`\~=`\,
	\lowercase{\endgroup\let~}\pFqcomma
	{}_{#2}F_{#3}{\left[\genfrac..{0pt}{}{#4}{#5};#6\right]}%
	\endgroup
}
\newcommand{\pFqcomma}{{\normalcomma}\mskip\pFqmuskip}

\let\origmaketitle\maketitle
\def\maketitle{
	\begingroup
	\def\uppercasenonmath##1{} 
	\let\MakeUppercase\relax 
	\origmaketitle
	\endgroup
}

\usepackage{fancyhdr}
\usepackage{mathtools}

\makeatletter
\DeclareRobustCommand\widecheck[1]{{\mathpalette\@widecheck{#1}}}
\def\@widecheck#1#2{%
	\setbox\z@\hbox{\m@th$#1#2$}%
	\setbox\tw@\hbox{\m@th$#1%
		\widehat{%
			\vrule\@width\z@\@height\ht\z@
			\vrule\@height\z@\@width\wd\z@}$}%
	\dp\tw@-\ht\z@
	\@tempdima\ht\z@ \advance\@tempdima2\ht\tw@ \divide\@tempdima\thr@@
	\setbox\tw@\hbox{%
		\raise\@tempdima\hbox{\scalebox{1}[-1]{\lower\@tempdima\box
				\tw@}}}%
	{\ooalign{\box\tw@ \cr \box\z@}}}
\makeatother

\begin{document}

\title[Modulated Bi-orthogonal Polynomials on the Unit Circle: The $2j-k$ and $j-2k$ Systems]{Modulated Bi-orthogonal Polynomials on the Unit Circle: \\ The $2j-k$ and $j-2k$ Systems}

\maketitle

\begin{center}
\authorfootnotes	
Roozbeh Gharakhloo\footnote{Department of Mathematics, Colorado State University, Fort Collins, CO 80521, USA, E-mail: roozbeh.gharakhloo@colostate.edu},  
Nicholas S. Witte\footnote{School of Mathematics and Statistics, Victoria University of Wellington, Wellington, New Zealand e-mail: n.s.witte@protonmail.com}
\par \bigskip
\end{center}	

\begin{abstract}
	We construct the systems of bi-orthogonal polynomials on the unit circle where the Toeplitz structure of the moment determinants is replaced by
	$ \det(w_{2j-k})_{0\leq j,k \leq N-1} $ and the corresponding Vandermonde modulus squared is replaced by 
	$ \prod_{1 \le j < k \le N}(\zeta_k - \zeta_j)(\zeta^{-2}_k - \zeta^{-2}_j) $.
	This is the simplest case of a general system of $pj-qk$ with $p,q$ co-prime integers.
	We derive analogues of the structures well known in the Toeplitz case:
	third order recurrence relations, determinantal and multiple-integral representations, their reproducing kernel and Christoffel-Darboux sum, and associated (Carath{\'e}odory) functions. We close by giving full explicit details for the system defined by the simple weight $ w(\zeta)=e^{\zeta}$, which is a specialisation of a weight arising from averages of moments of derivatives of characteristic polynomials over $\tn{USp}(2N)$, $\tn{SO}(2N)$ and $\tn{O}^-(2N)$.
\end{abstract}

\begin{itemize}
	\item[] \footnotesize \textit{2020 Mathematics Subject Classification.} 42A80, 42A52, 47B35, 33C45, 39A06, 11M50.
	\item[] \footnotesize \textit{Keywords.} bi-orthogonal polynomials on the unit circle, Toeplitz matrices, linear difference equations, random matrix theory
\end{itemize}



\tableofcontents

\section{Motivation}\label{Sec Intro}

The unitary group $U(N)$ with Haar (uniform) measure possesses the explicit character formula of Weyl \cite{Weyl_1946}
\begin{equation}
	\frac{1}{(2 \pi )^N N!} \prod_{1 \le j < k \le N} (\zeta_k - \zeta_j)(\zeta^{-1}_k - \zeta^{-1}_j),
	\quad \zeta_l \coloneqq e^{i \theta_l} \in \T, \quad \theta_l \in (-\pi,\pi] ,
\label{UN_Haar}
\end{equation}
where $ \{\zeta_1,\dots, \zeta_N \} \in {\rm Spec}(U) $ and $ \T = \{\zeta \in \C: |\zeta|=1 \} $.
This also has an interpretation as the joint probability density function (PDF) (see e.g.~\cite[Chapter 2]{For_2010}) for the eigenphases of the Dyson circular ensemble (CUE).
One fundamental application of this formula is to characterise averages over $U \in U(N)$ of class functions $ c(U) $,
i.e. symmetric functions of the eigenvalues of $ U $ only.
An example of such functions are products $ \prod_{j=1}^{N}w(\zeta_j) $ where $ w(\zeta) $ may be interpreted as a weight function or density.
Introducing the Fourier components $\{w_l\}_{l\in \Z}$ of this weight
\begin{equation}\label{laurent_exp}
	w(\zeta) = \sum_{l=-\infty}^\infty w_l \zeta^l ,
\end{equation} 
due to the well known Heine identity \cite{So_1967}
\begin{equation}
	\mathbb{E} _{U(N)} \big[\prod_{l=1}^N w(\zeta_l)\big] = \det[ w_{j-k} ]_{j,k=0,\dots,N-1},
\label{Haar_avge}
\end{equation}
we recognise that this is equivalent to studying Toeplitz determinants.
Intimately connected with the above problem are systems of bi-orthogonal polynomials
on the unit circle and its relationship to general, non-hermitian (i.e. complex weight) Toeplitz determinants.
The system of bi-orthogonal polynomials $ \{ \varphi_n(z),\bar{\varphi}_n(z) \}^{\infty}_{n=0} $\footnote{The bar notation in  $\{\bar{\varphi}_n(z)\}^{\infty}_{n=0}$ is a standard notation for the polynomials orthogonal to the polynomials $\{\varphi_n(z)\}^{\infty}_{n=0}$   (see e.g. \cite{FW_2006}). If the weight $w$ is complex-valued the bar notation does not represent the complex conjugation, while if the weight $w$ is real-valued the bar notation represents the complex conjugation.} with respect to the 
weight $ w(\zeta) $ on the unit circle may be defined by the orthogonality relation
\begin{equation}
	\int_{\T} \frac{\dd \zeta}{2\pi i\zeta} w(\zeta)\varphi_m(\zeta)\bar{\varphi}_n(\bar{\zeta}) = \delta_{m,n} .
\label{ops_onorm}
\end{equation}

Such averages over the unitary group are ubiquitous in many areas of mathematical physics,  
in particular the gap probabilities and characteristic polynomial averages in the 
circular ensembles of random matrix theory \cite{AvM_2001b},\cite{AvM_2003},\cite{FW_2004},\cite{For_2010},
the spin-spin correlations of the planar Ising model \cite{MW_1973},\cite{JM_1980},
the density matrix of a system of impenetrable bosons on the ring \cite{FFGW_2003a} 
and probability distributions for various classes of non-intersecting lattice path problems \cite{For_2003}.

Our study is motivated by yet another application of random matrix techniques, in particular to analytic number theory.
Random matrix models have been very successful in constructing conjectures for estimating the integral moments of central values in the $U(N)$ 
family of $L$-functions, through the works of \cite{KS_2000}, \cite{CFK+_2003}, \cite{CFK+_2008}.
This program was extended by Al\.i Altu\u{g} et al \cite{ABP+_2014} to the three other families of $L$-functions:
those characterised by the symmetry types $\tn{USp}(2N)$, $\tn{SO}(2N)$ and $\tn{O}^-(2N)$, 
and where the statistic of concern was the $n$-th moment of the $m$-th derivative of the characteristic polynomial $\Lambda_{A}$
with $ A \in \tn{USp}(2N),\tn{SO}(2N),\tn{O}^-(2N) $. They computed 
\begin{equation}\label{key}
M_n(\tn{G}(2N),m) := \int_{\tn{G}(2N)} \pr{\Lambda^{(m)}_A(1)}^n\,\dd A,
\end{equation}
where $\tn{G}$ denotes $\tn{USp}$, $\tn{SO}$, or $\tn{O}^-$, and $\dd A$ is the Haar measure on $\tn{G}$, in the regime $N\to\infty$ and fixed $n,m$.
Employing similar techniques to \cite{CRS_2006} they found that the leading coefficient in the large-$N$ expansion is proportional to the $n$-th derivative of $e^{u}\mathcal{T}_{n,\ell}(u)$ where
\begin{equation}\label{taufunction}
	\mathcal{T}_{n,\ell}(u):=\det\pr{g_{2j-k+\ell}(u)}_{0 \leq j,k \leq n-1},  
\end{equation}
for $n\geq0$, $\ell\in\Z$ is fixed by the symmetry type and $u\in\C$, and where
\begin{equation}
	g_l(u) 	=\frac{1}{2\pi i}\int_{\T}\frac{e^{\zeta+u\zeta^{-2}}}{\zeta^{l+1}} \, \dd \zeta
			=\frac{1}{\Gamma(l+1)}{}_{0}F_{2}\pr{;\tfrac{1}{2}l+1,\tfrac{1}{2}(l+1);\tfrac{1}{4}u}.
\label{OSweight}
\end{equation}

Clearly for the $\tn{USp}(2N)$, $\tn{SO}(2N)$ and $\tn{O}^-(2N)$ types the relevant moment determinant has the structure
\begin{equation}\label{2j-k Det}
	\det(w_{2j-k+r})_{0\leq j,k \leq n-1}, 
\end{equation} 
and the corresponding joint density function has the form
\begin{equation}\label{1m2_JPDF}
	\prod_{1 \le j < k \le n} \left(\zeta_k - \zeta_j\right)  \left(\zeta^{-2}_k - \zeta^{-2}_j\right)  ,
	\quad \zeta_l \coloneqq e^{i \theta_l} \in \T, \quad \theta_l \in (-\pi,\pi] .
\end{equation}
Note that this joint density function is not real and positive, unlike the Toeplitz case \eqref{UN_Haar}, 
and it will become clear that meaning can only be given to more restrictive classes of weights than applies in the Toeplitz case, 
see Theorems \ref{P and Q exist and are unique} and \ref{R and S exist and are unique} for precise conditions on the existence of our systems. 
For example the system with a constant weight or even a terminating Laurent expansion of \eqref{laurent_exp}, i.e. a finite banded moment matrix, will not exist
\footnote{An initial sequence of moment determinants can be non-zero, but the remainder will vanish identically.}.
Of course one could take as the definition of a putative system the modulus of \eqref{1m2_JPDF} however this has some disadvantages and we have chosen to pursue the analytic form here. 
If one takes the modulus of \eqref{1m2_JPDF} it factorises into three factors
\begin{equation}\label{log-gas}
	\prod_{1 \le j < k \le n} \left|\zeta_k - \zeta_j\right|^{2}  \left|\zeta_k + \zeta_j\right| , 
\end{equation}
and one can see that this also applies to the joint density function
\begin{equation}\label{2m1_JPDF}
	\prod_{1 \le j < k \le n} \left(\zeta^{2}_k - \zeta^{2}_j\right)  \left(\zeta^{-1}_k - \zeta^{-1}_j\right)  ,
\end{equation}
and thus any distinction is lost.
More significantly is that the property of complex analyticity would be lost in some aspects of the theory, 
and in particular the differential structures whereby one requires spectral derivatives $ \dd/\dd\zeta $ of the bi-orthogonal polynomials and their associated functions 
- such derivatives form one member of key Lax pairs of the integrable system.
We will finish our discussion of these issues with a final observation on the log-gas interpretation of Eq. \eqref{log-gas}:
it has the repulsion of coincident eigenvalues $ \theta_k=\theta_j, k\neq j $ with strength $\beta=2$ but also a repulsion when $\theta_k=\theta_j\pm \pi$ with half this strength. 
It is our goal to extend the theory of Toeplitz determinants and associated bi-orthogonal polynomial systems on the unit circle to the case where they have a structure $ pj-qk $ for two positive co-prime integers, based upon \eqref{2j-k Det} and \eqref{1m2_JPDF}. 
In the first example of such an extension we consider the $ p=2, q=1 $ and $ p=1, q=2 $ cases exclusively, 
however we will only employ techniques which are capable of generalisation to the other cases.

Analogous systems of bi-orthogonal polynomials on the line with the Hankel structure $ \det(h_{pj+qk})_{0\leq j,k\leq N-1}$ have been studied for some time -
the area was initiated by Preiser \cite{Pre_1962}; 
general properties of these systems were investigated by Konhauser \cite{Kon_1965}, Ilyasov \cite{Il_1983}, Iserles and N{\o}rsett \cite{IN_1988}, \cite{IN_1989}; 
explicit examples of cases related to Laguerre polynomials by Konhauser \cite{Kon_1967}, 
Carlitz \cite{Car_1968}, \cite{Car_1973}, Genin and Calvez \cite{GC_1969}, \cite{GC_1969b}, Prabhakar \cite{Pra_1970}, Srivastava \cite{Sri_1973}, Raizada \cite{Rai_1993};
and to the Jacobi polynomials by Madhekar and Thakare \cite{MT_1982}, \cite{TM_1984}, \cite{MT_1984}, \cite{TM_1986}, \cite{TM_1986a}.
As they stand the results reported in the above works do not translate directly into the ones we seek.
Systems of bi-orthogonal Laurent polynomials on the line would be expected to provide an equivalent framework to the system we treat here, 
however we prefer our approach because of its direct linkage to the random matrix application described earlier.

In the random matrix literature systems of bi-orthogonal polynomials on the line are known as Muttalib-Borodin ensembles following the pioneering work of \cite{Mut_1995} and \cite{Bor_1999a}; 
Claeys and Romano have derived recurrence relations and a scalar Riemann-Hilbert problem for general classes of weights \cite{CR_2014}.
See \cite{FW_2017}, \cite{FI_2018} for some selected recent developments.
Other related lines of investigation are the multiple orthogonal polynomial ensembles studied by Kuijlaars and McLaughlin \cite{KMcL_2005}, and by Kuijlaars \cite{Kui_2010}. In the latter work higher rank (i.e. greater than two) matrix systems of bi-orthogonal multiple polynomials - the Nikishin systems and Angelesco systems - were discussed
and only the case of the Angelesco system (see Eq. (4.8)) with $ p=2 $ species of $ n_1=n_2=n $ particles $ \{x^{(1)}_k\}_{k=1}^{n} $ , $ \{x^{(2)}_k\}_{k=1}^{n}  $ 
linked by $ x^{(2)}_k = -x^{(1)}_k, k = 1,\ldots, n $ and $ w_2(x)=w_1(-x) $ has any correspondence with our example
 - it is actually proportional to the square of \eqref{1m2_JPDF}.
 
 It is also worth mentioning that on the Operator Theory side, the associated operators on $L^2(\T)$ and their restrictions on the Hardy space $H^2(\T)$ have been a subject of research in the last 25 years or so. In a series of works \cite{Hothesis,HoIndianaMathJournal,HoMichiganMathJournal,HoAdjoints1,HoAdjoints2}, Mark C. Ho introduced and made some fundamental studies on the $2j-k$ operators which he referred to as \textit{slant Toeplitz} operators. Later, Subhash C. Arora and Ruchika Batra studied the properties of $pj-k$ operators ($p \in \N_{\geq 2}$) which they referred to as \textit{generalized slant Toeplitz} or \textit{p-th order slant Toeplitz} operators in a collection of papers \cite{AB4,AB3,AB2,AB1}. Particularly regarding the large-size asymptotic analysis of the $pj-qk$ determinants, it is our hope that the interplay between Operator Theory and the Riemann-Hilbert method (which is rooted in the orthogonality structures studied this work and is the subject of a future paper) could be made somewhat tractable. Examples of such interplay for Toeplitz, bordered Toeplitz and Toeplitz+Hankel determinants can be respectively found in \cite{DIKimpetus}, \cite{BEGIL}, and the introduction of \cite{GI}.

\color{black}

\subsection{Outline}
Here we summarise the outline of our study. In \S \ref{Sec Definitions and Notations} we present the definitions of the $2j-k$ and $j-2k$ (master) determinants and the corresponding systems of bi-orthogonal polynomials. In this section we also introduce the two main tools for our analysis: the Dodgson Condensation identity and the multiple integral formulae for the $2j-k$ and $j-2k$ determinants.
The LDU decompositions for the $2j-k$ and $j-2k$ determinants are also derived which will be useful for proving the Christoffel-Darboux identity. In \S \ref{Sec Bordered}  we  provide the bordered determinant representations and prove the existence and uniqueness of the systems of bi-orthogonal polynomials provided that the associated determinants are non-zero. We also discuss the connection of $2j-k$ and $j-2k$ determinants and bi-orthogonal polynomials. Finally in this section we express $2j-k$ and $j-2k$ bi-orthogonal polynomials and reproducing kernels in terms of the corresponding master determinants. In \S \ref{Sec Rec Rel} we derive pure-degree and pure-offset recurrence relations for the $2j-k$ and $j-2k$ bi-orthogonal polynomials. We also discuss equivalent Dodgson condensation identities which result in the same recurrence relations and present several mixed recurrence relations involving both $2j-k$ and $j-2k$ bi-orthogonal polynomials. In subsection \ref{subsec poly tails, rec rels and dets} we derive several relationships between polynomial tails, recurrence coefficients and determinants of the $2j-k$ and $j-2k$ systems. In \S \ref{Sec MultInt} we prove multiple integral formulae for the $2j-k$ and $j-2k$ determinants, bi-orthogonal polynomials, and reproducing kernels. We use some of these formulae to derive representations of $Q$ and $R$-polynomials in terms of $P$ and $S$-polynomials, respectively. The multiple integral formulae are also useful in deriving the Christoffel-Darboux identity in subsection \ref{subsec CD}. In \S \ref{Sec Associated Functions} we introduce the associated functions and derive their multiple integral formulae representations. We also find the corresponding Casorati matrices and the first order recurrence relations they satisfy. In \S \ref{Sec Exp weight} we go back to the weight relevant to the 
$L$-functions of the symmetry types $\tn{USp}(2N)$, $\tn{SO}(2N)$ and $\tn{O}^-(2N)$. To study a first concrete example, in this subsection we specifically study the undeformed weight when $u=0$ where we find explicit formulae for the determinants,  Carath\'{e}odory functions, and the $2j-k$ polynomials. Eventually, in \S \ref{Sec open Qs}, we will lay out a list of important open questions and the prospects of future work. Throughout the paper, to highlight the distinguishing features of these modulated bi-orthogonal systems and for the convenience of the reader we try to make a comparison with the Toeplitz ($j-k$) theory whenever a result about $2j-k$ and $j-2k$ systems is presented, mainly by referring to \cite{FW_2006}.

\section{Definitions and Notations}\label{Sec Definitions and Notations}

In this section we will define the objects studied in the paper and introduce the necessary notations and conventions. Throughout the paper we respectively use $j$ and $k$ as indices referring to the rows and the columns, and we frequently use a boldfaced letter to distinguish the determinant of a matrix with the matrix itself: $\det\boldsymbol{\mathscr{M}} \equiv \mathscr{M}$. Let $\boldsymbol{\mathscr{M}}$ be an $n \times n$ matrix. By
\begin{equation}\label{matrix with removed col-rows}
	\boldsymbol{\mathscr{M}} \left\lbrace \begin{matrix} j_1& j_2& \cdots & j_{\ell} \\  k_1& k_2& \cdots & k_{\ell} \end{matrix} \right\rbrace, \qandq \mathscr{M} \left\lbrace \begin{matrix} j_1& j_2& \cdots & j_{\ell} \\  k_1& k_2& \cdots & k_{\ell} \end{matrix} \right\rbrace,
\end{equation}
we respectively mean the $(n-\ell)\times(n-\ell)$ matrix obtained from $\boldsymbol{\mathscr{M}}$ by removing the rows $j_i$ and the columns $k_i$, $1\leq i \leq \ell$, and its corresponding determinant. Although the order of writing the row and column indices is immaterial for this definition, in this work we prefer to respect the order of indices $j_{\ell_1}<j_{\ell_2}$ and $k_{\ell_1}<k_{\ell_2}$ if $\ell_1<\ell_2$. We now recall the \textit{Dodgson Condensation}  identity\footnote{which is also known as the \textit{Desnanot–Jacobi} identity or the \textit{Sylvester determinant}  identity.} (see \cite{Abeles,Fulmek-Kleber,Bressoud} and references therein) which is an important tool for deriving many important relationships between the objects studied in this work:
\begin{equation}\label{DODGSON}
\mathscr{M} \cdot \mathscr{M}\left\lbrace \begin{matrix} j_1 & j_2 \\  k_1& k_2 \end{matrix} \right\rbrace = \mathscr{M}\left\lbrace \begin{matrix} j_1  \\  k_1 \end{matrix} \right\rbrace \cdot \mathscr{M}\left\lbrace \begin{matrix} j_2  \\  k_2 \end{matrix} \right\rbrace - \mathscr{M}\left\lbrace \begin{matrix} j_1  \\  k_2 \end{matrix} \right\rbrace \cdot \mathscr{M}\left\lbrace \begin{matrix} j_2  \\  k_1 \end{matrix} \right\rbrace.
\end{equation}
 
\begin{definition} \normalfont
For fixed \textit{offset values} $r,s \in \Z$, we define the following $(n+3)\times(n+3)$ \textit{master matrices} of $2j-k$ and $j-2k$ structure:

\noindent
	\begin{minipage}{.5\linewidth}
		\begin{alignat}{2}
		&\boldsymbol{\mathscr{D}}_r(z,\mathcal{z}) &&= \begin{pmatrix}
		w_r & w_{r-1} &  \cdots & w_{r-n-1} & z^{n+1} \\ 
		w_{r+2} & w_{r+1} & \cdots & w_{r-n+1} & z^{n}  \\
		\vdots & \vdots &  \cdots & \vdots & \vdots \\
		w_{r+2n+2} & w_{r+2n+1} & \cdots & w_{r+n+1} & 1 \\
		1 & \mathcal{z} &  \cdots & \mathcal{z}^{n+1} & \star \\ 			 			
		\end{pmatrix}, \label{DDD}
		\end{alignat}	
	\end{minipage}	
	\begin{minipage}{.5\linewidth}
		\begin{alignat}{2}
		&\boldsymbol{\mathscr{E}}_s(z,\mathcal{z}) &&= \begin{pmatrix}
		w_s & w_{s-2} &  \cdots &  w_{s-2n-2} & z^{n+1} \\ 
		w_{s+1} & w_{s-1} &  \cdots  & w_{s-2n-1} & z^{n} \\
		\vdots & \vdots &  \cdots  & \vdots & \vdots \\
		w_{s+n+1} & w_{s+n-1} &  \cdots &  w_{s-n-1} & 1 \\
		1 & \mathcal{z} & \cdots  & \mathcal{z}^{n+1} & \star \\ 			 			
		\end{pmatrix}. \label{EEE} 
		\end{alignat}	
	\end{minipage}
where
	\begin{equation}\label{Fourier Coeff}
	w_k=\int_{\T} \zeta^{-k}w(\zeta)\frac{\dd \zeta}{2\pi \ic \zeta}, 
	\end{equation}
is the $k$-th Fourier coefficient of the symbol $ w(z) = \sum_{\ell \in \Z} w_{l}z^{l} $.
\end{definition}

 In \eqref{DDD} and \eqref{EEE} for simplicity of notation we have suppressed the dependence of $\boldsymbol{\mathscr{D}}_r(z,\mathcal{z})$ and $\boldsymbol{\mathscr{E}}_s(z,\mathcal{z})$ on $n$. Also, throughout the paper, when $z$ and $\mathcal{z}$ are not distinguished as \textit{distinct} independent variables we simply drop the arguments in the notation of the master matrices:
\begin{equation}\label{DDEE}
	\boldsymbol{\mathscr{D}}_r \equiv \boldsymbol{\mathscr{D}}_r(z,z) \qandq \boldsymbol{\mathscr{E}}_s \equiv \boldsymbol{\mathscr{E}}_s(z,z).
\end{equation}
 We use the determinants of master matrices $\boldsymbol{\mathscr{D}}_r$ and $\boldsymbol{\mathscr{E}}_s$ to construct the $2j-k$ and $j-2k$ systems of bi-orthogonal polynomials in \S \ref{Sec Bordered}, but since in all of those constructions either the last row or the last column is removed, the entry $\star$ in \eqref{DDD} and \eqref{EEE} never plays a role for construction of the bi-orthogonal polynomials and thus is arbitrary for the purposes of this work. Let $\boldsymbol{D}^{(r)}_{n}$ and $\boldsymbol{E}^{(s)}_{n}$ respectively denote the $n\times n$ matrices of $2j-k$ and $j-2k$ structure and by $D^{(r)}_{n}$ and $E^{(s)}_{n}$ denote their determinants:
\begin{equation}\label{Det}
D_{n}^{(r)} := \det \begin{pmatrix}
w_{r} & w_{r-1}   & \cdots & w_{r-n+1} \\
w_{r+2}  & w_{r+1}  & \cdots & w_{r-n+3} \\
\vdots & \vdots &  \vdots & \vdots \\
w_{r+2n-2} & w_{r+2n-3} &  \cdots & w_{r+n-1}
\end{pmatrix} \equiv \underset{0\leq j,k \leq n-1}{\det}\left( w_{r+2j-k} \right), 
\end{equation}
\begin{equation}\label{Det E}
E_{n}^{(s)} := \det \begin{pmatrix}
w_{s} & w_{s-2}  &  \cdots & w_{s-2n+2} \\
w_{s+1}  & w_{s-1} &  \cdots & w_{s-2n+3} \\
\vdots & \vdots &  \vdots & \vdots \\
w_{s+n-1} & w_{s+n-3} &  \cdots & w_{s-n+1} 
\end{pmatrix} \equiv \underset{0\leq j,k \leq n-1}{\det} \left( w_{s+j-2k} \right). 
\end{equation}
 $D^{(r)}_{n}$ and $E^{(s)}_{n}$ can obviously be written in terms of the determinant of master matrices $\boldsymbol{\mathscr{D}}_r$ and $\boldsymbol{\mathscr{E}}_s$,  as
 \begin{equation}\label{D DD E EE}
 	D^{(r)}_{n}=\mathscr{D}_r \left\lbrace \begin{matrix} n& n+1& n+2 \\  n& n+1& n+2 \end{matrix} \right\rbrace, \qandq E^{(s)}_{n} = \mathscr{E}_s \left\lbrace \begin{matrix} n& n+1& n+2 \\  n& n+1& n+2 \end{matrix} \right\rbrace.
 \end{equation}
Also notice that

\noindent\begin{minipage}{.5\linewidth}
	\begin{alignat}{2}
	&\boldsymbol{\mathscr{D}}_r \left\lbrace \begin{matrix} n& n+1& n+2 \\  n& n+1& n+2 \end{matrix} \right\rbrace  &&=\boldsymbol{\mathscr{D}}_r \left\lbrace \begin{matrix} 0& n+1& n+2 \\  0& 1& n+2 \end{matrix} \right\rbrace, \label{DeqD}
	\end{alignat}	
\end{minipage}	
\begin{minipage}{.5\linewidth}
	\begin{alignat}{2}
	&\boldsymbol{\mathscr{E}}_s \left\lbrace \begin{matrix} n& n+1& n+2 \\  n& n+1& n+2 \end{matrix} \right\rbrace &&= \boldsymbol{\mathscr{E}}_s \left\lbrace \begin{matrix} 0& 1& n+2 \\  0& n+1& n+2 \end{matrix} \right\rbrace. \label{EeqE} 
	\end{alignat}	
\end{minipage}
\begin{definition} \normalfont
	For an integrable function $f$ on the unit circle, we respectively define the  $2j-k$ and $j-2k$ multiple integrals as 
	\begin{equation}\label{DD}
	\mathcal{D}_n[f(\ze)]:= \frac{1}{n!} \int_{\T} \frac{\dd \ze_1}{2 \pi \ic \ze_1}\int_{\T} \frac{\dd \ze_2}{2 \pi \ic \ze_2} \cdots \int_{\T} \frac{\dd \ze_n}{2 \pi \ic \ze_n} \prod_{j=1}^{n}f(\ze_j) \prod_{1\leq j<k\leq n} (\ze_k-\ze_j)(\ze^{-2}_k-\ze^{-2}_j),
	\end{equation} 
	and
	\begin{equation}\label{EE}
	\mathcal{E}_n[f(\ze)]:= \frac{1}{n!} \int_{\T} \frac{\dd \ze_1}{2 \pi \ic \ze_1}\int_{\T} \frac{\dd \ze_2}{2 \pi \ic \ze_2} \cdots \int_{\T} \frac{\dd \ze_n}{2 \pi \ic \ze_n} \prod_{j=1}^{n}f(\ze_j) \prod_{1\leq j<k\leq n} (\ze^2_k-\ze^2_j)(\ze^{-1}_k-\ze^{-1}_j).
	\end{equation} 
\end{definition}
In \S \ref{Sec MultInt}, in particular, we show the following multiple integral representations for  $D^{(r)}_{n}$ and $E^{(s)}_{n}$:  

\noindent\begin{minipage}{.5\linewidth}
	\begin{alignat}{2}
	&\mathcal{D}_n[w(\ze)\ze^{-r}] && = D^{(r)}_n, \label{Dd}
	\end{alignat}	
\end{minipage}	
\begin{minipage}{.5\linewidth}
	\begin{alignat}{2}
	&\mathcal{E}_n[w(\ze)\ze^{-s}] && = E^{(s)}_n. \label{Ee}
	\end{alignat}	
\end{minipage}
\begin{definition} \normalfont
For each offset value $r \in \Z$, define the $2j-k$ sequences of monic polynomials $\{P_{n}(z;r)\}^{\infty}_{n=0}$ and $\{Q_{n}(z;r)\}^{\infty}_{n=0}$,  $\deg P_n(z;r)=\deg Q_n(z;r) = n$, satisfying the \textit{bi-orthogonality}  condition:
\begin{equation}\label{PQorth}
\int_{\T} P_{m}(\ze;r) Q_{n}(\ze^{-2};r)\ze^{-r}\frac{\dd \mu(\ze)}{2\pi \ic \ze} = h^{(r)}_{n}\delta_{mn}, \qquad m,n \in \N \cup \{0\},
\end{equation}
and for each offset value $s \in \Z$, define the $j-2k$ sequences of monic polynomials $\{R_{n}(z;s)\}^{\infty}_{n=0}$ and $\{S_{n}(z;s)\}^{\infty}_{n=0}$, $\deg R_n(z;s)=\deg S_n(z;s) = n$, satisfying the bi-orthogonality condition:
\begin{equation}\label{RSorth}
\int_{\T} R_{m}(\ze^2;s) S_{n}(\ze^{-1};s) \ze^{-s}\frac{\dd \mu(\ze)}{2\pi \ic \ze} = g^{(s)}_{n}\delta_{mn}, \qquad m,n \in \N \cup \{0\},
\end{equation}
where $h^{(r)}_{n}$ and $g^{(s)}_{n}$ are the \textit{norms} of the polynomials squared and $\dd\mu(\ze)\equiv w(\ze)\dd \ze$ for some weight function $w(z)$.
\end{definition}

We will give representations of $h^{(r)}_{n}$ and $g^{(s)}_{n}$ as ratios of $2j-k$ and $j-2k$ determinants in Theorems \ref{P and Q exist and are unique} and \ref{R and S exist and are unique}. Notice that the bi-orthogonality condition \eqref{PQorth} is equivalent to the orthogonality relations
\begin{equation}\label{OP1}
\int_{\T} P_{n}(\ze;r) \ze^{-2m-r} \frac{\dd\mu(\ze)}{2\pi \ic \ze} = h^{(r)}_{n} \delta_{mn}, \qquad m=0,1,\cdots, n, 
\end{equation}
and
\begin{equation}\label{OP2}
\int_{\T} Q_n(\ze^{-2};r) \ze^{m-r} \frac{\dd\mu(\ze)}{2\pi \ic \ze} = h^{(r)}_{n} \delta_{mn}, \qquad m=0,1,\cdots, n.
\end{equation}
Similarly, the bi-orthogonality condition \eqref{RSorth} is equivalent to the orthogonality relations
\begin{equation}\label{OP1 R}
\int_{\T} R_{n}(\ze^2;s) \ze^{-m-s} \frac{\dd\mu(\ze)}{2\pi \ic \ze} = g^{(s)}_{n} \delta_{mn},\qquad m=0,1,\cdots, n, 
\end{equation}
and
\begin{equation}\label{OP2 S}
\int_{\T} S_n(\ze^{-1};s) \ze^{2m-s} \frac{\dd\mu(\ze)}{2\pi \ic \ze} = g^{(s)}_{n} \delta_{mn},\qquad m=0,1,\cdots, n.
\end{equation}

\subsection{LDU Decomposition of $2j-k$ and $j-2k$ moment matrices}\label{LDU}

The linear space of polynomials $\{P_n(z)\}^{\infty}_{n=0}$,  $\{Q_n(z)\}^{\infty}_{n=0}$, etc. have expansions in an appropriate basis,
in this case the monomial basis $\{z^n\}^{\infty}_{n=-\infty}$ as befits a P{\'a}de approximation problem with two fixed singularities 
$0,  \infty$. We are therefore led to consider the linear transformations from this preferred basis to our orthogonal system.
Let us denote the coefficients of $2j-k$ and $j-2k$ polynomials more precisely as 
\begin{equation}\label{polys}
P_n(z;r) = \sum^{n}_{\ell=0} \mathcal{p}^{(r)}_{n,\ell} z^{\ell}, \qquad Q_n(z;r) =\sum^{n}_{\ell=0} \mathcal{q}^{(r)}_{n,\ell} z^{\ell}, \qquad R_n(z;s) = \sum^{n}_{\ell=0} \mathcal{r}^{(s)}_{n,\ell} z^{\ell}, \qquad S_n(z;s) =\sum^{n}_{\ell=0} \mathcal{s}^{(s)}_{n,\ell} z^{\ell}.
\end{equation}
with $\mathcal{p}^{(r)}_{n,n}=\mathcal{q}^{(r)}_{n,n}=\mathcal{r}^{(s)}_{n,n}=\mathcal{s}^{(s)}_{n,n}=1$. Let us also denote 
\begin{equation}\label{Vectors}
\boldsymbol{Z}_n(z) :=
 \begin{pmatrix}
1 \\ z \\ \vdots \\ z^{n}
\end{pmatrix}  \qandq \boldsymbol{F}_n(z;r) :=
\begin{pmatrix}
F_0(z;r) \\ F_1(z;r) \\ \vdots \\ F_n(z;r)
\end{pmatrix}, \qquad \boldsymbol{F} \in \{\boldsymbol{P},\boldsymbol{Q},\boldsymbol{R},\boldsymbol{S}\}.
\end{equation}
We thus have
\begin{equation}
\boldsymbol{P}_n(z;r) = \boldsymbol{\mathcal{P}}_{n}^{(r)}  \boldsymbol{Z}_n(z), \qquad	\boldsymbol{Q}_n(z;r) = \boldsymbol{\mathcal{Q}}_{n}^{(r)}  \boldsymbol{Z}_n(z), \qquad \boldsymbol{R}_n(z;s) = \boldsymbol{\mathcal{R}}_{n}^{(s)}  \boldsymbol{Z}_n(z), \qquad	\boldsymbol{S}_n(z;s) = \boldsymbol{\mathcal{S}}_{n}^{(s)}  \boldsymbol{Z}_n(z), 
\end{equation}
where $\boldsymbol{\mathcal{P}}_{n}^{(r)}, \boldsymbol{\mathcal{Q}}_{n}^{(r)},  \boldsymbol{\mathcal{R}}_{n}^{(s)}$ and $ \boldsymbol{\mathcal{S}}_{n}^{(s)}$ are the following $(n+1)\times(n+1)$ lower triangular matrices
\begin{equation}\label{A B}
\boldsymbol{\mathcal{P}}_{n}^{(r)} :=
\begin{pmatrix}
1 & 0 & \cdots & 0 \\
\mathcal{p}^{(r)}_{1,0} & 1 & \cdots & 0 \\
\vdots & \vdots & \ddots & \vdots \\
\mathcal{p}^{(r)}_{n,0} & \mathcal{p}^{(r)}_{n,1} & \cdots & 1
\end{pmatrix}, \qquad 	\boldsymbol{\mathcal{Q}}_{n}^{(r)} := 
\begin{pmatrix}
1 & 0 & \cdots & 0 \\
\mathcal{q}^{(r)}_{1,0} & 1 & \cdots & 0 \\
\vdots & \vdots & \ddots & \vdots \\
\mathcal{q}^{(r)}_{n,0} & \mathcal{q}^{(r)}_{n,1} & \cdots & 1
\end{pmatrix}, 
\end{equation}
\begin{equation}\label{C G}
\boldsymbol{\mathcal{R}}_{n}^{(s)} := 
\begin{pmatrix}
1 & 0 & \cdots & 0 \\
\mathcal{\mathcal{r}}^{(s)}_{1,0} & 1 & \cdots & 0 \\
\vdots & \vdots & \ddots & \vdots \\
\mathcal{r}^{(s)}_{n,0} & \mathcal{r}^{(s)}_{n,1} & \cdots & 1
\end{pmatrix}, \qquad 	\boldsymbol{\mathcal{S}}_{n}^{(s)} := 
\begin{pmatrix}
1 & 0 & \cdots & 0 \\
\mathcal{s}^{(s)}_{1,0} & 1 & \cdots & 0 \\
\vdots & \vdots & \ddots & \vdots \\
\mathcal{s}^{(s)}_{n,0} & \mathcal{s}^{(s)}_{n,1} & \cdots & 1
\end{pmatrix}, 
\end{equation}
whose inverses are also lower diagonal with $1$'s on the main diagonal. Also let us denote the diagonal matrices of norms of polynomials by $\boldsymbol{h}^{(r)}_{n}$ and $\boldsymbol{g}^{(r)}_{n}$:
\begin{equation}\label{h&H diag}
\boldsymbol{h}^{(r)}_{n} := 
\begin{pmatrix}
h^{(r)}_0 &  \cdots & 0 \\
\vdots  & \ddots & \vdots \\
0  & \cdots & h^{(r)}_n
\end{pmatrix}, \qandq \boldsymbol{g}^{(s)}_{n} := 
\begin{pmatrix}
g^{(s)}_0 &  \cdots & 0 \\
 \vdots & \ddots & \vdots \\
0  & \cdots & g^{(s)}_n
\end{pmatrix}.
\end{equation}

In the following result we give the LDU decompositions of the moment matrices for the $2j-k$ and $j-2k$ systems which parallels the Toeplitz case closely,
see the unpublished work \cite{Mag_2013}.

\begin{theorem}
The LDU decompositions of $\boldsymbol{D}^{(r)}_{n+1}$ and $\boldsymbol{E}^{(s)}_{n+1}$ are given by 
	
\noindent\begin{minipage}{.5\linewidth}
		\begin{alignat}{2}
		&\boldsymbol{D}^{(r)}_{n+1}  &&= \left[\boldsymbol{\mathcal{Q}}^{(r)}_{n}\right]^{-1} 	\boldsymbol{h}^{(r)}_{n}  \left[\left(\boldsymbol{\mathcal{P}}^{(r)}_{n}\right)^T\right]^{-1}, \label{LDU D}
		\end{alignat}	
	\end{minipage}	
	\begin{minipage}{.5\linewidth}
		\begin{alignat}{2}
		&\boldsymbol{E}^{(s)}_{n+1} &&= \left[\boldsymbol{\mathcal{S}}^{(s)}_{n}\right]^{-1} 	\boldsymbol{g}^{(s)}_{n}  \left[\left(\boldsymbol{\mathcal{R}}^{(s)}_{n}\right)^T\right]^{-1}. \label{LDU E} 
		\end{alignat}	
	\end{minipage}
\end{theorem}
\begin{proof}
	In this proof we drop the dependence of all functions and quantities on the offsets $r$ and $s$ for simplicity of notation. We have
	\begin{equation}
	\begin{split}
	h_{\nu}\delta_{\nu \mu} & = \int_{\T} P_{\nu}(\ze) Q_{\mu}(\ze^{-2}) \ze^{-r} w(\ze) \frac{\dd \ze}{2\pi \ic \ze}  = \sum_{m=0}^{\nu} \sum_{\ell=0}^{\mu} \mathcal{p}_{\nu,m}\mathcal{q}_{\mu,\ell} \int_{\T} \ze^{m-2\ell-r} w(\ze) \frac{\dd \ze}{2\pi \ic \ze}  \\ &
	= \sum_{m=0}^{\nu} \sum_{\ell=0}^{\mu} \mathcal{p}_{\nu,m}\mathcal{q}_{\mu,\ell} w_{2\ell-m+r} = \sum_{m=0}^{\nu} \sum_{\ell=0}^{\mu} 	\left( \boldsymbol{\mathcal{P}}_{n} \right)_{\nu,m} \left( \boldsymbol{\mathcal{Q}}_{n} \right)_{\mu,\ell} \left( \boldsymbol{D}_{n+1} \right)_{\ell,m}  \\ & = \sum_{m=0}^{\nu} \sum_{\ell=0}^{\mu}  \left( \boldsymbol{\mathcal{Q}}_{n} \right)_{\mu,\ell} \left( \boldsymbol{D}_{n+1} \right)_{\ell,m}  \left( \boldsymbol{\mathcal{P}}_{n}^T \right)_{m,\nu} =  \left( \boldsymbol{\mathcal{Q}}_{n} \boldsymbol{D}_{n+1} \boldsymbol{\mathcal{P}}_{n}^T \right)_{\mu,\nu},
	\end{split}
	\end{equation} 
	which is equivalent to \eqref{LDU D}. For the decomposition of $\boldsymbol{E}^{(s)}_{n+1}$ we consider
	\begin{equation}
	\begin{split}
	g_{\nu}\delta_{\nu \mu} & = \int_{\T} R_{\nu}(\ze^2) S_{\mu}(\ze^{-1}) \ze^{-s} w(\ze) \frac{\dd \ze}{2\pi \ic \ze}  = \sum_{m=0}^{\nu} \sum_{\ell=0}^{\mu} \mathcal{r}_{\nu,m}\mathcal{s}_{\mu,\ell} \int_{\T} \ze^{2m-\ell-s} w(\ze) \frac{\dd \ze}{2\pi \ic \ze}  \\ &
	= \sum_{m=0}^{\nu} \sum_{\ell=0}^{\mu} \mathcal{r}_{\nu,m}\mathcal{s}_{\mu,\ell} w_{\ell-2m+s} = \sum_{m=0}^{\nu} \sum_{\ell=0}^{\mu} 	\left( \boldsymbol{\mathcal{R}}_{n} \right)_{\nu,m} \left( \boldsymbol{\mathcal{S}}_{n} \right)_{\mu,\ell} \left( \boldsymbol{E}_{n+1} \right)_{\ell,m}  \\ & = \sum_{m=0}^{\nu} \sum_{\ell=0}^{\mu}  \left( \boldsymbol{\mathcal{S}}_{n} \right)_{\mu,\ell} \left( \boldsymbol{E}_{n+1} \right)_{\ell,m}  \left( \boldsymbol{\mathcal{R}}_{n}^T \right)_{m,\nu} =  \left( \boldsymbol{\mathcal{S}}_{n} \boldsymbol{E}_{n+1} \boldsymbol{\mathcal{R}}_{n}^T \right)_{\mu,\nu},
	\end{split}
	\end{equation} 
	which yields \eqref{LDU E}.
\end{proof}

\section{Bordered determinant representations}\label{Sec Bordered}
In this section we focus on determinantal representations of fundamental elements in the theory where the moment matrix is 
bordered by rows or columns containing basis vectors for the polynomial spaces.
In the following result we find that the bi-orthogonal polynomials of both systems can be represented as bordered moment determinants in almost exactly the same way as for the Toeplitz case, see e.g. Eq.(2.15,16) of \cite{FW_2006} for comparison. As well as revealing the conditions on the existence and uniqueness of these systems in a simple way, this form will be very useful in our subsequent treatment.
\begin{theorem}\label{P and Q exist and are unique}
If $D_{n}^{(r)} \neq 0$, the polynomials $P_n(z;r)$ and $Q_n(z;r)$ exist and are uniquely given by
\begin{equation}\label{OP11}
	P_n(z;r) = \frac{1}{D_{n}^{(r)}} 
	\det \begin{pmatrix}
	w_{r} & w_{r-1}  & \cdots & w_{r-n} \\
	w_{r+2} & w_{r+1} & \cdots & w_{r-n+2} \\
	\vdots & \vdots  & \vdots & \vdots \\
	w_{r+2n-2} & w_{r+2n-3} &  \cdots & w_{r+n-2} \\
	1 & z & \cdots  & z^n
	\end{pmatrix},
\end{equation}
and 
\begin{equation}\label{OP22}
	Q_n(z;r) = \frac{1}{D_{n}^{(r)}} 
	\det \begin{pmatrix}
	w_{r} & w_{r-1}  & \cdots & w_{r-n+1} & 1 \\
	w_{r+2} & w_{r+1} &  \cdots & w_{r-n+3} & z \\
	\vdots & \vdots  & \vdots & \vdots & \vdots \\
	w_{r+2n} & w_{r+2n-1} & \cdots & w_{r+n+1} & z^n
	\end{pmatrix},
\end{equation}
from which one can observe that $	h^{(s)}_{n}$ exists and can be written as
\begin{equation}\label{h}
	   h^{(r)}_{n} = \frac{D_{n+1}^{(r)}}{D_{n}^{(r)}}, \qquad n \in \N \cup \{0\}, \qquad D^{(r)}_0 \equiv 1.
\end{equation}
Therefore if all $h^{(r)}_{\ell}$ exist and are non-zero for $\ell=0, \cdots, n-1$, then
\begin{equation}\label{Dets from norms}
	   D^{(r)}_{n}=\prod_{\ell=0}^{n-1} h^{(r)}_{\ell}.
\end{equation}
\end{theorem}
\begin{proof}
The existence simply follows from the fact that, if $D_{n}^{(r)} \neq 0$, we can explicitly construct the system of monic bi-orthogonal  polynomials $P_n(z;r)$ and $Q_n(z;r)$ as in \eqref{OP11} and \eqref{OP22}, respectively satisfying \eqref{OP1} and \eqref{OP2}. Now, let us discuss the uniqueness if $D_{n}^{(r)} \neq 0$. Assume that $P_n(z;r)=z^n+ \di \sum^{n-1}_{\ell=0} \mathcal{p}^{(r)}_{n,\ell} z^{\ell}$ satisfies the orthogonality conditions \eqref{OP1}. One can write the orthogonality conditions \eqref{OP1} for $m=0,1,\cdots,n-1$ as the following linear system for solving the constants $\mathcal{p}^{(r)}_{n,\ell}$, $0\leq \ell \leq n-1$:
\begin{equation}
	\begin{pmatrix}
	w_{r} & w_{r-1}   & \cdots & w_{r-n+1} \\
	w_{r+2}  & w_{r+1}  & \cdots & w_{r-n+3} \\
	\vdots & \vdots &  \vdots & \vdots \\
	w_{r+2n-2} & w_{r+2n-3} &  \cdots & w_{r+n-1}
	\end{pmatrix} 
	\begin{pmatrix}
	\mathcal{p}^{(r)}_{n,0} \\[4pt] \mathcal{p}^{(r)}_{n,1} \\[4pt] \vdots \\[4pt] \mathcal{p}^{(r)}_{n,n-1}
	\end{pmatrix}=  
	\begin{pmatrix}
	-w_{r-n} \\ -w_{r-n+2} \\ \vdots \\ -w_{r+n-2}
	\end{pmatrix}. 
\end{equation}
So, if $D_{n}^{(r)} \neq 0$, the above linear system has a unique solution, and thus $P_{n}(z;r)$ is uniquely given by \eqref{OP11}. Now, assume that $Q_n(z;r)=z^n+ \di \sum^{n-1}_{\ell=0} \mathcal{q}^{(r)}_{n,\ell} z^{\ell}$ satisfies the orthogonality conditions \eqref{OP2}. We can write the orthogonality conditions \eqref{OP2} for $m=0,1,\cdots,n-1$ as the following linear system for solving the constants $\mathcal{q}^{(r)}_{n,\ell}$, $0\leq \ell \leq n-1$:
\begin{equation}
	\begin{pmatrix}
	w_{r} & w_{r-1}   & \cdots & w_{r-n+1} \\
	w_{r+2}  & w_{r+1}  & \cdots & w_{r-n+3} \\
	\vdots & \vdots &  \vdots & \vdots \\
	w_{r+2n-2} & w_{r+2n-3} &  \cdots & w_{r+n-1}
	\end{pmatrix}^T 
	\begin{pmatrix}
	\mathcal{q}^{(r)}_{n,0} \\ \mathcal{q}^{(r)}_{n,1} \\ \vdots \\ \mathcal{q}^{(r)}_{n,n-1}
	\end{pmatrix}=  
	\begin{pmatrix}
	-w_{2n+r} \\ -w_{2n+r-1} \\ \vdots \\ -w_{n+r+1}
	\end{pmatrix}. 
\end{equation}
So, again, if $D_{n}^{(r)} \neq 0$, the above linear system can be inverted, and thus has a unique solution. Therefore, $Q_{n}(z;r)$ is uniquely given by \eqref{OP22}. Finally, one can directly find \eqref{h} using \eqref{OP11} and \eqref{OP1}, or alternatively, using \eqref{OP22} and \eqref{OP2}.
\end{proof}

In an identical manner, we can prove the following Theorem about existence and uniqueness of $j-2k$ system of bi-orthogonal polynomials.

\begin{theorem}\label{R and S exist and are unique}
If $E_{n}^{(s)} \neq 0$, the polynomials $R_n(z;s)$ and $S_n(z;s)$ exist and are uniquely given by
\begin{equation}\label{OP11 R}
	R_n(z;s) = \frac{1}{E_{n}^{(s)}} 
	\det \begin{pmatrix}
	w_{s} & w_{s-2}  & \cdots & w_{s-2n} \\
	w_{s+1} & w_{s-1} &  \cdots & w_{s-2n+1} \\
	\vdots & \vdots  & \vdots & \vdots \\
	w_{s+n-1} & w_{s+n-3} &  \cdots & w_{s-n-1} \\
	1 & z & \cdots  & z^n
	\end{pmatrix},
\end{equation}
and
\begin{equation}\label{OP22 S}
	S_n(z;s) = \frac{1}{E_{n}^{(s)}} 
	\det \begin{pmatrix}
	w_{s} & w_{s-2}  & \cdots & w_{s-2n+2} & 1 \\
	w_{s+1} & w_{s-1}  & \cdots & w_{s-2n+3} & z \\
	\vdots & \vdots  & \vdots & \vdots & \vdots \\
	w_{s+n} & w_{s+n-2} &  \cdots & w_{s-n+2} & z^n
	\end{pmatrix},
\end{equation}
from which one can observe that $	g^{(s)}_{n}$ exists and can be written as
\begin{equation}\label{H}
	g^{(s)}_{n} = \frac{E_{n+1}^{(s)}}{E_{n}^{(s)}}, \qquad n \in \N \cup \{0\}, \qquad E^{(s)}_0 \equiv 1.
\end{equation} 
Therefore if all $g^{(s)}_{\ell}$ exist and are non-zero for $\ell=0, \cdots, n-1$, then
\begin{equation}\label{Dets from norms 2}
 	E^{(s)}_{n}=\prod_{\ell=0}^{n-1} g^{(s)}_{\ell}.
\end{equation}
\end{theorem}

\subsection{Connection of $2j-k$ and $j-2k$ polynomials}
The $2j-k$ and $j-2k$ systems are intimately related by a duality through exploiting the freedom to choose suitable offsets.
This equivalence is only exhibited in a formal algebraic rather than an analytical sense and reflects a mapping of the interior of the unit circle to the exterior and vice-versa.
In the Toeplitz case one has a system of self-duality.
Consequently the first of such duality relations involves the determinants, which are related to one another if one selects the appropriate offset values. More precisely, we have
\begin{equation}\label{E&D rel}
D^{(r)}_n = E^{(r+n-1)}_{n},
\end{equation}
as these are reflections of each other across the main anti-diagonal. Indeed, for any $n \times n$ matrix $\boldsymbol{\mathscr{M}}$ we have $\det \boldsymbol{\mathscr{M}}^{\perp} = \det \boldsymbol{\mathscr{M}}$, which is due to the identity $\boldsymbol{\mathscr{M}}^{\perp} = \boldsymbol{\mathscr{A}}_{n} \boldsymbol{\mathscr{M}}^{T} \boldsymbol{\mathscr{A}}_n,$
 where $\boldsymbol{\mathscr{A}}_n$ is the $n \times n$ matrix with ones on the anti-diagonal and zeros everywhere else, $\boldsymbol{\mathscr{M}}^T$ is the transpose of the $n \times n$ matrix $\boldsymbol{\mathscr{M}}$, and $\boldsymbol{\mathscr{M}}^{\perp}$ is the reflection of the $n \times n$ matrix $\boldsymbol{\mathscr{M}}$ across its main anti-diagonal. Continuing the development of the duality theme our next result furnishes further details in regard to the polynomials themselves. But first let us recall a standard notation:  
 for any polynomial $p$ of degree $n$, we denote the reciprocal polynomial by $p^*$, that is
 \begin{equation}\label{star}
	 p^*(z):=z^n p(z^{-1}).
 \end{equation}
\begin{theorem}\label{Thm S*-P & Q*-R}
	The following identities hold between $2j-k$ and $j-2k$ polynomials
	
\noindent\begin{minipage}{.5\linewidth}
		\begin{alignat}{2}
		&S^*_n(z;s) &&= \frac{(-1)^nD^{(s-n+2)}_n}{E_{n}^{(s)}}P_n(z;s-n+2), \label{S* and P}
		\end{alignat}	
	\end{minipage}	
	\begin{minipage}{.5\linewidth}
		\begin{alignat}{2}
		&R_n(z;s) &&= \frac{(-1)^nD_{n}^{(s-n-1)}}{E_{n}^{(s)}}Q^*_n(z;s-n-1). \label{Q* and R} 
		\end{alignat}	
	\end{minipage}
\end{theorem}
\begin{proof}
	Computing $S^*_n(z;s)$ from \eqref{OP22 S} in view of \eqref{star} and reflection across the anti-diagonal gives
	\begin{equation}\label{OP22 S*}
	S^*_n(z;s) := \frac{1}{E_{n}^{(s)}} 
	\det \begin{pmatrix}
	1 & z &  \cdots  & z^n \\
	w_{s-n+2} & w_{s-n+1}  & \cdots  & w_{s-2n+2} \\
	\vdots & \vdots & \vdots & \vdots \\
	w_{s+n} & w_{s+n-1} &  \cdots &  w_{s}
	\end{pmatrix}.
	\end{equation}
	Now, we do $n$ consecutive adjacent row-swaps to move the first row to the last which results in \eqref{S* and P} in view of \eqref{OP11}. The relationship \eqref{Q* and R} can be found in an identical manner.
\end{proof}
\subsection{Biorthogonal polynomials in terms of the master determinants} For many choices of triples $\{a,b,c\}\in \Z^3$, one can express $z^{a}f_{n+b}(z;r+c)$, $f \in \{P,Q^*,R,S^*\}$,  in terms of the determinants of master matrices \eqref{DDD} and \eqref{EEE}. For each polynomial $P,Q^*,R$ and $S^*$, among the choices mentioned above, we select five which turn out to be useful in connection to the specific Dodgson condensation identities used in \S \ref{Sec Rec Rel}.

\noindent\begin{minipage}{.5\linewidth}
	\begin{alignat}{2}
	&	P_{n+1}(z;r) &&= \frac{\mathscr{D}_r \left\lbrace \begin{matrix} n+1 \\  n+2 \end{matrix} \right\rbrace}{\mathscr{D}_r \left\lbrace \begin{matrix}  n+1& n+2 \\   n+1& n+2 \end{matrix} \right\rbrace},\label{P_n+1 r}\\
	&P_n(z;r) &&= \frac{\mathscr{D}_r \left\lbrace \begin{matrix} n& n+1 \\  n+1& n+2 \end{matrix} \right\rbrace}{\mathscr{D}_r \left\lbrace \begin{matrix} n& n+1& n+2 \\  n& n+1& n+2 \end{matrix} \right\rbrace},\label{P_n r}\\
	& P_{n}(z;r+2) &&= \frac{\mathscr{D}_{r} \left\lbrace \begin{matrix} 0 & n+1 \\ n+1 & n+2 \end{matrix} \right\rbrace}{\mathscr{D}_{r+2} \left\lbrace \begin{matrix} n & n+1& n+2 \\  n & n+1& n+2 \end{matrix} \right\rbrace}, \label{P_n r+2} \\ & zP_{n}(z;r-1) &&= \frac{\mathscr{D}_{r} \left\lbrace \begin{matrix} n & n+1 \\ 0 & n+2 \end{matrix} \right\rbrace}{\mathscr{D}_{r-1} \left\lbrace \begin{matrix} n & n+1& n+2 \\  n & n+1& n+2 \end{matrix} \right\rbrace}, \label{zP_n r-1}
	\end{alignat}	
\end{minipage}%
\begin{minipage}{.5\linewidth}
	\begin{alignat}{2}
	& R_{n+1}(z;s) &&= \frac{\mathscr{E}_s \left\lbrace \begin{matrix} n+1 \\  n+2 \end{matrix} \right\rbrace}{\mathscr{E}_s \left\lbrace \begin{matrix}  n+1& n+2 \\   n+1& n+2 \end{matrix} \right\rbrace}, \label{R_n+1 s} \\
	& R_n(z;s)  &&= \frac{\mathscr{E}_s \left\lbrace \begin{matrix} n& n+1 \\  n+1& n+2 \end{matrix} \right\rbrace}{\mathscr{E}_s \left\lbrace \begin{matrix} n& n+1& n+2 \\  n& n+1& n+2 \end{matrix} \right\rbrace}, \label{R_n s}\\
	&R_{n}(z;s+1)  &&= \frac{\mathscr{E}_{s} \left\lbrace \begin{matrix} 0&n+1 \\  n+1&n+2 \end{matrix} \right\rbrace}{\mathscr{E}_{s+1} \left\lbrace \begin{matrix}  n & n+1& n+2 \\  n & n+1& n+2 \end{matrix} \right\rbrace}, \label{R_n s+1} \\ 	&zR_{n}(z;s-1) &&= \frac{\mathscr{E}_s \left\lbrace \begin{matrix} 0 & n+1 \\ 0& n+2 \end{matrix} \right\rbrace}{\mathscr{E}_{s-1} \left\lbrace \begin{matrix}  n& n+1& n+2 \\   n& n+1& n+2 \end{matrix} \right\rbrace},\label{zR_n s-1}
	\end{alignat}
\end{minipage}

\noindent\begin{minipage}{.5\linewidth}
	\begin{alignat}{2}
	&zP_{n}(z;r+1) &&= \frac{\mathscr{D}_{r} \left\lbrace \begin{matrix} 0 & n+1 \\ 0 & n+2 \end{matrix} \right\rbrace}{\mathscr{D}_{r+1} \left\lbrace \begin{matrix} n & n+1& n+2 \\  n & n+1& n+2 \end{matrix}, \label{zP_n r+1} \right\rbrace},
	\end{alignat}	
\end{minipage}%
\begin{minipage}{.5\linewidth}
	\begin{alignat}{2}
	&zR_{n}(z;s-2) &&= \frac{\mathscr{E}_s \left\lbrace \begin{matrix} n & n+1 \\ 0& n+2 \end{matrix} \right\rbrace}{\mathscr{E}_{s-2} \left\lbrace \begin{matrix}  n& n+1& n+2 \\   n& n+1& n+2 \end{matrix} \right\rbrace}, \label{zR_n s-2}
	\end{alignat}
\end{minipage}

\noindent\begin{minipage}{.5\linewidth}
	\begin{alignat}{2}
 & Q^*_{n+1}(z;r) &&= \frac{\mathscr{D}_r \left\lbrace \begin{matrix} n+2 \\  n+1 \end{matrix} \right\rbrace}{\mathscr{D}_r \left\lbrace \begin{matrix}  n+1& n+2 \\   n+1& n+2 \end{matrix} \right\rbrace}, \label{Q^*_n+1 r} \\ 
	& Q^*_n(z;r+1) &&= \frac{\mathscr{D}_r \left\lbrace \begin{matrix} 0 & n+2 \\  0& n+1 \end{matrix} \right\rbrace}{\mathscr{D}_{r+1} \left\lbrace \begin{matrix} n& n+1& n+2 \\  n& n+1& n+2 \end{matrix} \right\rbrace}, \label{Q^*_n r+1}\\
	& Q^*_n(z;r+2) &&= \frac{\mathscr{D}_r \left\lbrace \begin{matrix} 0 & n+2 \\  n& n+1 \end{matrix} \right\rbrace}{\mathscr{D}_{r+2} \left\lbrace \begin{matrix} n& n+1& n+2 \\  n& n+1& n+2 \end{matrix} \right\rbrace}, \label{Q^*_n r+2}\\
	& zQ^*_n(z;r)&&= \frac{\mathscr{D}_r \left\lbrace \begin{matrix} n+1& n+2 \\  n& n+1 \end{matrix} \right\rbrace}{\mathscr{D}_r \left\lbrace \begin{matrix} n& n+1& n+2 \\  n& n+1& n+2 \end{matrix} \right\rbrace}, \label{zQ^*_n r} \\
	& zQ^*_n(z;r-1)&&= \frac{\mathscr{D}_r \left\lbrace \begin{matrix} n+1& n+2 \\  0& n+1 \end{matrix} \right\rbrace}{\mathscr{D}_{r-1} \left\lbrace \begin{matrix} n& n+1& n+2 \\  n& n+1& n+2 \end{matrix} \right\rbrace}, \label{zQ^*_n r-1}
	\end{alignat}	
\end{minipage}%
\begin{minipage}{.5\linewidth}
	\begin{alignat}{2}
	&	S^*_{n+1}(z;s) &&= \frac{\mathscr{E}_s \left\lbrace \begin{matrix} n+2 \\  n+1 \end{matrix} \right\rbrace}{\mathscr{E}_s \left\lbrace \begin{matrix}  n+1& n+2 \\   n+1& n+2 \end{matrix} \right\rbrace},  \label{S^*_n+1 s}\\
	&	S^*_n(z;s-1) &&= \frac{\mathscr{E}_s \left\lbrace \begin{matrix} 0& n+2 \\  0& n+1 \end{matrix} \right\rbrace}{\mathscr{E}_{s-1} \left\lbrace \begin{matrix} n& n+1& n+2 \\  n& n+1& n+2 \end{matrix} \right\rbrace}, \label{S^*_n s-1}\\
	&	S^*_n(z;s+1)&&= \frac{\mathscr{E}_s \left\lbrace \begin{matrix} 0& n+2 \\  n& n+1 \end{matrix} \right\rbrace}{\mathscr{E}_{s+1} \left\lbrace \begin{matrix} n& n+1& n+2 \\  n& n+1& n+2 \end{matrix} \right\rbrace}, \label{S^*_n s+1}\\
	& 	zS^*_n(z;s-2)&&= \frac{\mathscr{E}_r \left\lbrace \begin{matrix} n+1& n+2 \\  0& n+1 \end{matrix} \right\rbrace}{\mathscr{E}_{s-2} \left\lbrace \begin{matrix} n& n+1& n+2 \\  n& n+1& n+2 \end{matrix} \right\rbrace}, \label{zS^*_n s-2}\\
	& 	zS^*_n(z;s)&&= \frac{\mathscr{E}_s \left\lbrace \begin{matrix} n+1& n+2 \\  n& n+1 \end{matrix} \right\rbrace}{\mathscr{E}_s \left\lbrace \begin{matrix} n& n+1& n+2 \\  n& n+1& n+2 \end{matrix} \right\rbrace}. \label{zS^*_n s}
	\end{alignat}
\end{minipage}

\subsection{The reproducing kernel in terms of the master determinants}

In the same spirit as the $j-k$ bi-orthogonal polynomials on the unit circle, we define the reproducing kernels for the $2j-k$ and $j-2k$ systems.
\begin{definition} \normalfont
The reproducing kernel for the $2j-k$ and $j-2k$ systems respectively are respectively defined as 

\noindent
	\begin{minipage}{.5\linewidth}
		\begin{alignat}{2}
		&K_{n}(z,\mathcal{z};r) && := \sum_{j=0}^{n} \frac{1}{h^{(r)}_{j}}Q_{j}(z;r)P_{j}(\mathcal{z};r), \label{RepKer3}
		\end{alignat}	
	\end{minipage}	
	\begin{minipage}{.5\linewidth}
		\begin{alignat}{2}
		&L_{n}(z,\mathcal{z};s) &&:= \sum_{j=0}^{n} \frac{1}{g^{(s)}_{j}}S_{j}(z;s)R_{j}(\mathcal{z};s). \label{RepKer3 j-2k}
		\end{alignat}	
	\end{minipage}
\end{definition}

It is easy to see that $K_{n}(z,\mathcal{z};r)$ and $L_{n}(z,\mathcal{z};s)$ has the \textit{reproducing properties} as described in the following equations:
\begin{equation}\label{RepKer1''}
\int_{\T} K_{n}(z,\zeta;r) Q_{\ell}(\zeta^{-2};r) \ze^{-r} w(\ze) \frac{\dd \zeta}{2\pi \ic \zeta} = 
\begin{cases}
Q_{\ell}(z;r), & 0 \leq \ell \leq n, \\
0, & \ell > n,
\end{cases} 
\end{equation}
\begin{equation}\label{ReKer2''}
\int_{\T}   K_{n}(\ze^{-2},\mathcal{z};r) P_{\ell}(\ze;r) \ze^{-r} w(\ze)  \frac{\dd \ze}{2\pi \ic \ze} =\begin{cases}
P_{\ell}(\mathcal{z};r), & 0 \leq \ell \leq n, \\
0, & \ell>n,
\end{cases} 
\end{equation}
\begin{equation}\label{RepKer1'' j-2k}
\int_{\T} L_{n}(z,\zeta^2;s) S_{\ell}(\zeta^{-1};s) \ze^{-s} w(\ze) \frac{\dd \zeta}{2\pi \ic \zeta} = 
\begin{cases}
S_{\ell}(z;s), & 0 \leq \ell \leq n, \\
0, & \ell > n,
\end{cases} 
\end{equation}
and
\begin{equation}\label{ReKer2'' j-2k}
\int_{\T}   L_{n}(\ze^{-1},\mathcal{z};r) R_{\ell}(\ze^2;s) \ze^{-s} w(\ze)  \frac{\dd \ze}{2\pi \ic \ze} =\begin{cases}
R_{\ell}(\mathcal{z};s), & 0 \leq \ell \leq n, \\
0, & \ell>n.
\end{cases} 
\end{equation}
Key properties flowing from these definitions are the normalisation relation
\begin{equation}
	\int_{\T} \frac{\dd \ze}{2 \pi \ic \ze} w(\ze) \ze^{-r} K_n(\ze^{-2},\ze;r) = n+1,
\label{kernel_norm}
\end{equation}
and the projection relation
\begin{equation}
	\int_{\T} \frac{\dd \ze}{2 \pi \ic \ze} w(\ze) \ze^{-r} K_n(z_{1},\ze;r) K_n(\ze^{-2},z_{2};r) = K_n(z_{1},z_{2};r).
\label{kernel_projection}
\end{equation}
Identical relations apply for the $ L_{n} $ kernel with the obvious modifications.

The above definitions of the reproducing kernels follows the pattern of the Toeplitz case and the following result linking them to 
bordered moment determinants reinforces this similarity. 
Although the Toeplitz case is not quoted often it is almost identical after making the obvious modifications.
\begin{theorem}
The reproducing kernels $K_{n}(z,\mathcal{z};r)$ and  $L_{n}(z,\mathcal{z};s)$ can be represented in terms of the master determinants $\mathscr{D}_r(z;\mathcal{z})$ and $\mathscr{E}_s(z;\mathcal{z})$ as
	
\noindent\begin{minipage}{.5\linewidth}
	\begin{alignat}{2}
	&z^{n+1}K_{n+1}(z^{-1},\mathcal{z};r) &&= - \frac{\mathscr{D}_r(z;\mathcal{z})}{D^{(r)}_{n+2}}, \label{Rep Ker Master 2j-k}
	\end{alignat}	
\end{minipage}	
\begin{minipage}{.5\linewidth}
	\begin{alignat}{2}
	&z^{n+1}L_{n+1}(z^{-1},\mathcal{z};s) &&= - \frac{\mathscr{E}_s(z;\mathcal{z})}{E^{(s)}_{n+2}}. \label{Rep Ker Master j-2k} 
	\end{alignat}	
\end{minipage}
\end{theorem}

\begin{proof}
	We only give the proof for \eqref{Rep Ker Master 2j-k}, as \eqref{Rep Ker Master j-2k} can be obtained in a similar way. Define \begin{equation}\label{ReproducingKer}
	\widehat{K}_{n}(z,\mathcal{z};r) := -\frac{1}{D^{(r)}_{n+1}}\det 
	\begin{pmatrix}
	w_{r} & w_{r-1}  & \cdots & w_{r-n} & 1\\
	w_{r+2} & w_{r+1}  & \cdots & w_{r-n+2} & z\\
	\vdots & \vdots  & \vdots & \vdots & \vdots\\
	w_{r+2n} & w_{r+2n-1} &  \cdots & w_{r+n} & z^{n} \\
	1  & \mathcal{z} &  \cdots & \mathcal{z}^{n} & 0
	\end{pmatrix}.
	\end{equation}
	One can easily establish that $\widehat{K}$ has the same reproducing properties as $K$. To this end, its action on monomials is given by
	\begin{equation}\label{RepKer1}
	\int_{\T}  \widehat{K}_{n}(z,\zeta;r)  \zeta^{-2\ell} \ze^{-r} w(\ze) \frac{\dd \zeta}{2\pi \ic \zeta} = z^{\ell}, \quad 0 \leq \ell \leq n ,
	\end{equation}
	and
	\begin{equation}\label{ReKer2}
	\int_{\T}   \widehat{K}_{n}(\ze^{-2},\mathcal{z};r) \ze^{\ell} \ze^{-r} w(\ze)  \frac{\dd \ze}{2\pi \ic \ze} = \mathcal{z}^{\ell}, \quad 0 \leq \ell \leq n .
	\end{equation}
	Therefore, we immediately get
	\begin{equation}\label{RepKer1'}
	\int_{\T}  \widehat{K}_{n}(z,\zeta;r) Q_{\ell}(\zeta^{-2};r) \ze^{-r} w(\ze) \frac{\dd \zeta}{2\pi \ic \zeta} = Q_{\ell}(z;r), \quad 0 \leq \ell \leq n ,
	\end{equation}
	and
	\begin{equation}\label{ReKer2'}
	\int_{\T}   \widehat{K}_{n}(\ze^{-2},\mathcal{z};r) P_{\ell}(\ze;r)\ze^{-r} w(\ze)  \frac{\dd \ze}{2\pi \ic \ze} = P_{\ell}(\mathcal{z};r), \quad 0 \leq \ell \leq n .
	\end{equation}
	Now we show that 
	\begin{equation}\label{RepKerDetRep}
	\widehat{K}_{n}(z,\mathcal{z};r) = K_{n}(z,\mathcal{z};r),
	\end{equation}
	for all $z,\mathcal{z} \in \C$ and $r\in \Z$. 	
	Let us consider the following $(n+2)\times(n+2)$ matrices
	\begin{equation}
	\boldsymbol{\widehat{K}_{n}}(z,\ze;r) = 
	\begin{pmatrix}
	w_{r} & w_{r-1}  &  \cdots & w_{r-n} & 1\\
	w_{r+2} & w_{r+1}  & \cdots & w_{r-n+2} & z\\
	\vdots & \vdots  & \vdots & \vdots & \vdots\\
	w_{r+2n} & w_{r+2n-1} & \cdots & w_{r+n} & z^{n} \\
	1  & \ze  & \cdots & \ze^{n} & 0
	\end{pmatrix},
	\end{equation}
	\begin{equation}
	\boldsymbol{\widehat{\mathcal{P}}}^{(r)}_{n} := 
	\begin{pmatrix}
	1 & 0 & \cdots & 0 & 0\\
	p^{(r)}_{1,0} & 1 & \cdots & 0 & 0\\
	\vdots & \vdots & \ddots & \vdots & \vdots \\
	p^{(r)}_{n,0} & p^{(r)}_{n,1} & \cdots & 1 & 0 \\
	0 & 0 & \cdots & 0 & 1
	\end{pmatrix}, \qandq 	\boldsymbol{\widehat{\mathcal{Q}}}^{(r)}_{n} := 
	\begin{pmatrix}
	1 & 0 & \cdots & 0 & 0\\
	q^{(r)}_{1,0} & 1 & \cdots & 0 & 0\\
	\vdots & \vdots & \ddots & \vdots & \vdots \\
	q^{(r)}_{n,0} & q^{(r)}_{n,1} & \cdots & 1 & 0 \\
	0 & 0 & \cdots & 0 & 1
	\end{pmatrix},
	\end{equation}
	Now observe that
	\begin{equation}
	\begin{split}
	\boldsymbol{\widehat{\mathcal{Q}}}^{(r)}_{n}
	\boldsymbol{\widehat{K}_{n}}(z,\ze;r)
	\left(\boldsymbol{\widehat{\mathcal{P}}}^{(r)}_{n}\right)^T & = 
	\begin{pmatrix}
	\boldsymbol{\mathcal{Q}}^{(r)}_{n} & \boldsymbol{0}_{n+1} \\
	\boldsymbol{0}^T_{n+1} & 1
	\end{pmatrix} 
	\begin{pmatrix}
	\boldsymbol{D}_{n+1} & \boldsymbol{Z}_{n}(z) \\
	\boldsymbol{Z}^T_{n}(\ze) & 0
	\end{pmatrix} 
	\begin{pmatrix} 
	\left(\boldsymbol{\mathcal{P}}^{(r)}_{n}\right)^T & \boldsymbol{0}_{n+1} \\
	\boldsymbol{0}^T_{n+1} & 1
	\end{pmatrix} \\ & = 
	\begin{pmatrix}
\boldsymbol{\mathcal{Q}}^{(r)}_{n} \boldsymbol{D}_{n+1}  \left(\boldsymbol{\mathcal{P}}^{(r)}_{n}\right)^T & \boldsymbol{\mathcal{Q}}^{(r)}_{n} \boldsymbol{Z}_{n}(z) \\[5pt]
	\boldsymbol{Z}^T_{n}(\ze) \left(\boldsymbol{\mathcal{P}}^{(r)}_{n}\right)^T & 0
	\end{pmatrix} = 
	\begin{pmatrix}
	\boldsymbol{h}^{(r)}_{n}	 & \boldsymbol{Q}_n(z;r) \\
	\boldsymbol{P}^T_n(\ze;r) & 0
	\end{pmatrix},
	\end{split}
	\end{equation}
	where the objects in the above equation are defined in subsection \ref{LDU}.
	Taking the determinant of both sides of the above equation yields \eqref{RepKerDetRep}. Finally by combining \eqref{RepKerDetRep} and \eqref{ReproducingKer} and recalling \eqref{DDD} we obtain \eqref{Rep Ker Master 2j-k}.
\end{proof}

\section{Recurrence relations}\label{Sec Rec Rel}

At this point we turn to the recurrence $ n \mapsto n+1 $ structures in the $2j-k$ and $j-2k$ systems, focusing initially on the scalar forms.
Throughout this section we primarily use the Dodgson condensation identity \eqref{DODGSON} in order to prove the following results, 
which furnish the tools that are capable of extension to the more general $pj-qk$ systems. 

Our first result yields extensions of the second order scalar difference equations given in Eqs.(2.23,24) of \cite{FW_2006} for the $2j-k$ system.

\begin{theorem}\label{THM degree rec 2j-k}
The third order pure-degree recurrence relations for the $2j-k$ polynomials are given by
	\begin{equation}\label{P pure n rec}
	P_{n+3}(z;r) -(\de_{n+2}^{(r)}+\de_{n+1}^{(r-1)})P_{n+2}(z;r) + (\de_{n+1}^{(r-1)}\de_{n+1}^{(r)} - z^2)P_{n+1}(z;r)+ (\de_n^{(r)}+\eta^{(r-2)}_n)z^2P_{n}(z;r)=0,
	\end{equation}
and
	\begin{equation}\label{Q* pure n rec}
		\begin{split}
		Q^*_{n+3}(z;r) &  -  (1+\be^{(r)}_{n+2} z) Q^*_{n+2}(z;r) +  (\be^{(r)}_{n+1}+\al^{(r+1)}_{n+1}+\be^{(r+1)}_{n+1}+\al^{(r+2)}_{n+1})z Q^*_{n+1}(z;r) \\ & -(\be^{(r+1)}_{n+1}+\al^{(r+2)}_{n+1})  (\be^{(r)}_{n}+\al^{(r+1)}_{n})z^2 Q^*_{n}(z;r) = 0,
		\end{split}
	\end{equation}
where	
	
\noindent\begin{minipage}{.5\linewidth}
	\begin{alignat}{2}
		&\de^{(r)}_{n} &&= -\frac{h^{(r-1)}_n}{h^{(r)}_{n}}, \label{delta} \\
		& \eta^{(r)}_n &&= \frac{D_n^{(r+2)}D_{n+1}^{(r-1)}}{D^{(r)}_{n+1}D^{(r+1)}_n}, \label{eta}
	\end{alignat}	
	\end{minipage}	
	\begin{minipage}{.5\linewidth}
		\begin{alignat}{2}
		&\be^{(r)}_{n} &&= -\frac{h^{(r+2)}_n}{h^{(r)}_{n}}, \label{beta} \\
		& \al^{(r)}_n&&= \frac{D^{(r-1)}_n D^{(r+2)}_{n+1}}{D^{(r)}_{n+1}D^{(r+1)}_n}. \label{alpha}
		\end{alignat}	
	\end{minipage}
	
\end{theorem}

\begin{proof}
Let us consider the following Dodgson condensation identities (see \eqref{DODGSON}):
	\begin{equation}\label{Dodgson2}\begin{split}
	& \mathscr{D}_r\left\lbrace \begin{matrix}  n+1  \\  n+2 \end{matrix} \right\rbrace \cdot \mathscr{D}_r\left\lbrace \begin{matrix}  n & n+1 & n+2 \\  0 & n+1 & n+2 \end{matrix} \right\rbrace =  \mathscr{D}_r \left\lbrace \begin{matrix} n &  n+1   \\ 0 &  n+2  \end{matrix} \right\rbrace \cdot \mathscr{D}_r\left\lbrace \begin{matrix}   n+1  & n+2  \\  n+1 &  n+2  \end{matrix} \right\rbrace - \mathscr{D}_r\left\lbrace \begin{matrix} n &  n+1   \\  n+1 &  n+2  \end{matrix} \right\rbrace \cdot \mathscr{D}_r\left\lbrace \begin{matrix}   n+1  & n+2  \\ 0 &  n+2 \end{matrix} \right\rbrace,
	\end{split}
	\end{equation}
	
	\begin{equation}\label{Dodgson3}\begin{split}
	& \mathscr{D}_r\left\lbrace \begin{matrix}  n+1 \\ n+2 \end{matrix} \right\rbrace \cdot \mathscr{D}_r\left\lbrace \begin{matrix}  0 &  n+1  & n+2 \\  0 & n+1 &  n+2  \end{matrix} \right\rbrace =  \mathscr{D}_r \left\lbrace \begin{matrix} 0 &  n+1   \\ 0 &  n+2  \end{matrix} \right\rbrace \cdot \mathscr{D}_r\left\lbrace \begin{matrix}   n+1  & n+2  \\  n+1 & n+2  \end{matrix} \right\rbrace - \mathscr{D}_r\left\lbrace \begin{matrix} 0 &  n+1   \\  n+1 &  n+2  \end{matrix} \right\rbrace \cdot \mathscr{D}_r\left\lbrace \begin{matrix}   n+1  & n+2  \\ 0 &  n+2  \end{matrix} \right\rbrace.
	\end{split}
	\end{equation}
The identity \eqref{Dodgson2} can be written as
	\begin{equation}\label{11}
	\frac{\mathscr{D}_r\left\lbrace \begin{matrix}  n+1  \\  n+2 \end{matrix} \right\rbrace}{\mathscr{D}_r \left\lbrace \begin{matrix}  n+1& n+2 \\   n+1& n+2 \end{matrix} \right\rbrace}  \cdot \mathscr{D}_r\left\lbrace \begin{matrix}  n &  n+1  & n+2 \\  0 & n+1 &  n+2  \end{matrix} \right\rbrace =  \mathscr{D}_r \left\lbrace \begin{matrix} n &  n+1  \\ 0 & n+2 \end{matrix} \right\rbrace  - \frac{\mathscr{D}_r\left\lbrace \begin{matrix} n &  n+1  \\  n+1 &  n+2  \end{matrix} \right\rbrace}{\mathscr{D}_r \left\lbrace \begin{matrix}  n+1& n+2 \\   n+1& n+2 \end{matrix} \right\rbrace}  \cdot \mathscr{D}_r\left\lbrace \begin{matrix}   n+1 & n+2  \\ 0 &  n+2 \end{matrix} \right\rbrace.
	\end{equation}
So from \eqref{P_n+1 r}, \eqref{P_n r}, \eqref{zP_n r-1}, and \eqref{D DD E EE} we have
	\begin{equation}
	D^{(r-1)}_n P_{n+1}(z;r) = D^{(r-1)}_n zP_n(z;r-1)-\frac{D_n^{(r)}D_{n+1}^{(r-1)}}{D_{n+1}^{(r)}}P_n(z;r).
	\end{equation}
Therefore from \eqref{h} we have
	\begin{equation}\label{1st rec P}
		P_{n+1}(z;r) = zP_n(z;r-1)+\de^{(r)}_nP_n(z;r), \qquad \de^{(r)}_{n} = -\frac{h^{(r-1)}_n}{h^{(r)}_{n}}.
	\end{equation}
Let us also write \eqref{Dodgson3} as 
	\begin{equation}\label{22}
	\frac{\mathscr{D}_r\left\lbrace \begin{matrix}  n+1  \\  n+2 \end{matrix} \right\rbrace}{\mathscr{D}_r \left\lbrace \begin{matrix}  n+1& n+2 \\   n+1& n+2 \end{matrix} \right\rbrace}  \cdot \mathscr{D}_r\left\lbrace \begin{matrix}  0 &  n+1  & n+2 \\  0 & n+1 &  n+2  \end{matrix} \right\rbrace =  \mathscr{D}_r \left\lbrace \begin{matrix} 0 &  n+1  \\ 0 & n+2 \end{matrix} \right\rbrace  - \frac{\mathscr{D}_r\left\lbrace \begin{matrix} 0 &  n+1  \\  n+1 &  n+2  \end{matrix} \right\rbrace}{\mathscr{D}_r \left\lbrace \begin{matrix}  n+1& n+2 \\   n+1& n+2 \end{matrix} \right\rbrace}  \cdot \mathscr{D}_r\left\lbrace \begin{matrix}   n+1 & n+2  \\ 0 &  n+2 \end{matrix} \right\rbrace.
	\end{equation}
Therefore from \eqref{P_n+1 r}, \eqref{P_n r+2}, \eqref{zP_n r+1}, and \eqref{D DD E EE} we arrive at
	\begin{equation}
	D^{(r+1)}_nP_{n+1}(z;r)=D^{(r+1)}_nzP_n(z;r+1)-\frac{D_n^{(r+2)}D_{n+1}^{(r-1)}}{D^{(r)}_{n+1}}P_{n}(z;r+2),
	\end{equation}
and thus
	\begin{equation}\label{2nd rec P}
		P_{n+1}(z;r)=zP_n(z;r+1)-\eta^{(r)}_n P_{n}(z;r+2), \qquad \eta^{(r)}_n = \frac{D_n^{(r+2)}D_{n+1}^{(r-1)}}{D^{(r)}_{n+1}D^{(r+1)}_n}.
	\end{equation}	
From \eqref{1st rec P} we have
	\begin{equation}\label{00}
	zP_{n}(z;r-1)=P_{n+1}(z;r)-\de_n^{(r)}P_n(z;r).
	\end{equation}
From this, we immediately get
	\begin{equation}
	zP_{n}(z;r-2)=P_{n+1}(z;r-1)-\de_n^{(r-1)}P_n(z;r-1).
	\end{equation}
Using \eqref{00} again we get
	\begin{equation}
	zP_{n}(z;r-2)=z^{-1}[P_{n+2}(z;r)-\de_{n+1}^{(r)}P_{n+1}(z;r)]-\de_n^{(r-1)}z^{-1}[P_{n+1}(z;r)-\de_{n}^{(r)}P_{n}(z;r)],	
	\end{equation}
and therefore
	\begin{equation}
	zP_{n}(z;r-2) = z^{-1}P_{n+2}(z;r)-z^{-1}(\de_{n+1}^{(r)}+\de_n^{(r-1)})P_{n+1}(z;r)+\de_n^{(r-1)}\de_{n}^{(r)}z^{-1}P_{n}(z;r).
	\end{equation}
On the other hand, from \eqref{2nd rec P} we have
	\begin{equation}
	P_{n+1}(z;r-2)=zP_{n}(z;r-1)-\eta^{(r-2)}_n P_{n}(z;r).
	\end{equation}
Combining the last two equations with \eqref{00} and some simplification gives \eqref{P pure n rec}.
	
In order to prove \eqref{Q* pure n rec} we need to consider the following Dodgson condensation identities:
	\begin{equation}\label{Dodgson5}\begin{split}
	& \mathscr{D}_r\left\lbrace \begin{matrix}  n+2 \\  n+1  \end{matrix} \right\rbrace \cdot \mathscr{D}_r\left\lbrace \begin{matrix} 0 & n+1 & n+2  \\  n &  n+1 & n+2 \end{matrix} \right\rbrace = \mathscr{D}_r \left\lbrace \begin{matrix} 0 &  n+2   \\ n &  n+1  \end{matrix} \right\rbrace \cdot \mathscr{D}_r\left\lbrace \begin{matrix} n+1 &  n+2   \\   n+1 & n+2 \end{matrix} \right\rbrace - \mathscr{D}_r\left\lbrace \begin{matrix} 0 &  n+2   \\  n+1  & n+2 \end{matrix} \right\rbrace \cdot \mathscr{D}_r\left\lbrace \begin{matrix} n+1 &  n+2   \\  n &  n+1  \end{matrix} \right\rbrace,
	\end{split}
	\end{equation}
	\begin{equation}\label{Dodgson6}\begin{split}
	& \mathscr{D}_r\left\lbrace \begin{matrix}  n+2  \\  n+1  \end{matrix} \right\rbrace \cdot \mathscr{D}_r\left\lbrace \begin{matrix} 0 & n+1 & n+2 \\  0 & n+1 & n+2 \end{matrix} \right\rbrace = \mathscr{D}_r \left\lbrace \begin{matrix} 0 & n+2  \\ 0 & n+1  \end{matrix} \right\rbrace \cdot \mathscr{D}_r\left\lbrace \begin{matrix} n+1 & n+2  \\   n+1 & n+2 \end{matrix} \right\rbrace - \mathscr{D}_r\left\lbrace \begin{matrix} 0 & n+2  \\ n+1  & n+2 \end{matrix} \right\rbrace \cdot \mathscr{D}_r\left\lbrace \begin{matrix} n+1 &  n+2  \\  0 &  n+1 \end{matrix} \right\rbrace.
	\end{split}
	\end{equation}
Let us write \eqref{Dodgson5} as 
	\begin{equation}\label{1}
	\frac{\mathscr{D}_r \left\lbrace \begin{matrix} n+2 \\  n+1 \end{matrix} \right\rbrace}{\mathscr{D}_r \left\lbrace \begin{matrix}  n+1& n+2 \\   n+1& n+2 \end{matrix} \right\rbrace} \cdot \mathscr{D}_r\left\lbrace \begin{matrix} 0 & n+1 &  n+2  \\  n &  n+1  & n+2 \end{matrix} \right\rbrace = \mathscr{D}_r \left\lbrace \begin{matrix} 0 & n+2  \\ n &  n+1  \end{matrix} \right\rbrace - \frac{\mathscr{D}_r\left\lbrace \begin{matrix} n+1 & n+2  \\  n &  n+1 \end{matrix} \right\rbrace}{\mathscr{D}_r \left\lbrace \begin{matrix}  n+1& n+2 \\   n+1& n+2 \end{matrix} \right\rbrace} \cdot \mathscr{D}_r \left\lbrace \begin{matrix} 0 & n+2 \\  n+1 & n+2 \end{matrix} \right\rbrace.  
	\end{equation}
So from  \eqref{Q^*_n+1 r}, \eqref{Q^*_n r+2}, \eqref{zQ^*_n r}, and \eqref{D DD E EE} we have 
	\begin{equation}
	D^{(r+2)}_{n}Q^*_{n+1}(z;r)= D^{(r+2)}_n Q^*_n(z;r+2)-\frac{D^{(r)}_n}{D^{(r)}_{n+1}} D^{(r+2)}_{n+1} z Q^*_{n}(z;r),
	\end{equation}
and thus from \eqref{h} we get 
	\begin{equation}\label{1st rec Q^*}
		Q^*_{n+1}(z;r)=  Q^*_n(z;r+2)+ \be^{(r)}_n z Q^*_{n}(z;r), \qquad \be^{(r)}_{n} = -\frac{h^{(r+2)}_n}{h^{(r)}_{n}}.
	\end{equation}	 
Now, let us write \eqref{Dodgson6} as
\begin{equation}\label{2}
\frac{\mathscr{D}_r \left\lbrace \begin{matrix} n+2 \\  n+1 \end{matrix} \right\rbrace}{\mathscr{D}_r \left\lbrace \begin{matrix}  n+1& n+2 \\   n+1& n+2 \end{matrix} \right\rbrace} \cdot \mathscr{D}_r\left\lbrace \begin{matrix} 0 & n+1 &  n+2  \\  0 &  n+1  & n+2 \end{matrix} \right\rbrace = \mathscr{D}_r \left\lbrace \begin{matrix} 0 & n+2  \\ 0 &  n+1  \end{matrix} \right\rbrace - \frac{\mathscr{D}_r\left\lbrace \begin{matrix} n+1 & n+2  \\  0 &  n+1 \end{matrix} \right\rbrace}{\mathscr{D}_r \left\lbrace \begin{matrix}  n+1& n+2 \\   n+1& n+2 \end{matrix} \right\rbrace} \cdot \mathscr{D}_r \left\lbrace \begin{matrix} 0 & n+2 \\  n+1 & n+2 \end{matrix} \right\rbrace.  
\end{equation}
So from  \eqref{Q^*_n+1 r}, \eqref{Q^*_n r+1}, \eqref{zQ^*_n r-1}, and \eqref{D DD E EE} we obtain 
\begin{equation}
D^{(r+1)}_n Q^*_{n+1}(z;r)= D^{(r+1)}_nQ^*_{n}(z;r+1) - \frac{D^{(r-1)}_n D^{(r+2)}_{n+1}}{D^{(r)}_{n+1}}zQ^*_{n}(z;r-1),
\end{equation}
and thus
\begin{equation}\label{2nd rec Q^*}
	Q^*_{n+1}(z;r)= Q^*_{n}(z;r+1) -\al^{(r)}_n zQ^*_{n}(z;r-1), \qquad \al^{(r)}_n= \frac{D^{(r-1)}_n D^{(r+2)}_{n+1}}{D^{(r)}_{n+1}D^{(r+1)}_n}.
\end{equation}
Now, let us find the pure-$n$ recurrence relation for $Q^*$s. From \eqref{1st rec Q^*} we have
\begin{equation}\label{Q^* r+1tor-1}
Q^*_n(z;r+1)=    Q^*_{n+1}(z;r-1)- \be^{(r-1)}_n z Q^*_{n}(z;r-1).
\end{equation}
From \eqref{2nd rec Q^*} and \eqref{Q^* r+1tor-1} we get
\begin{equation}\label{Q^*n+1 r}
Q^*_{n+1}(z;r) = Q^*_{n+1}(z;r-1)-  (\be^{(r-1)}_n+\al^{(r)}_n)z Q^*_{n}(z;r-1), 
\end{equation}
and immediately
\begin{equation}\label{Q^*n r+1}
Q^*_{n}(z;r+1) = Q^*_{n}(z;r)-  (\be^{(r)}_{n-1}+\al^{(r+1)}_{n-1})z Q^*_{n-1}(z;r), 
\end{equation}
and
\begin{equation}\label{Q^*n r+2}
Q^*_{n}(z;r+2) = Q^*_{n}(z;r+1)-  (\be^{(r+1)}_{n-1}+\al^{(r+2)}_{n-1})z Q^*_{n-1}(z;r+1). 
\end{equation}
Using \eqref{Q^*n r+1} we have
\begin{equation}
\begin{split}
Q^*_{n}(z;r+2) & = Q^*_{n}(z;r)-  (\be^{(r)}_{n-1}+\al^{(r+1)}_{n-1})z Q^*_{n-1}(z;r)  - (\be^{(r+1)}_{n-1}+\al^{(r+2)}_{n-1})z \left[Q^*_{n-1}(z;r)-  (\be^{(r)}_{n-2}+\al^{(r+1)}_{n-2})z Q^*_{n-2}(z;r) \right].
\end{split}
\end{equation}
Plugging this into \eqref{1st rec Q^*} yields
\begin{equation}
\begin{split}
Q^*_{n+1}(z;r) & =  \be^{(r)}_n z Q^*_{n}(z;r)+ Q^*_{n}(z;r)-  (\be^{(r)}_{n-1}+\al^{(r+1)}_{n-1})z Q^*_{n-1}(z;r) \\ & - (\be^{(r+1)}_{n-1}+\al^{(r+2)}_{n-1})z \left[Q^*_{n-1}(z;r)-  (\be^{(r)}_{n-2}+\al^{(r+1)}_{n-2})z Q^*_{n-2}(z;r) \right].
\end{split}
\end{equation}
Replacing $n$ by $n+2$ and rearranging terms gives the desired recurrence relation \eqref{Q* pure n rec}. 
\end{proof}

Our second result gives extensions of the second order scalar difference equations given in Eqs.(2.23,24) of \cite{FW_2006} for the $j-2k$ system.

\begin{theorem}\label{THM degree rec j-2k}
The third order pure-degree recurrence relations for the $j-2k$ polynomials are given by
\begin{equation}\label{R pure n rec}
	\begin{split}
R_{n+3}(z;s)
& - \left(z+\varkappa^{(s)}_{n+2}\right)R_{n+2}(z;s)
+ \left( \varkappa^{(s)}_{n+1}+\rho^{(s-1)}_{n+1}+\rho^{(s)}_{n+2}+\varkappa^{(s)}_{n+2}\right)z R_{n+1}(z;s)
\\ & - \left( \rho^{(s)}_{n+2}+\varkappa^{(s)}_{n+2}\right)\left( \varkappa^{(s)}_{n}+ \rho^{(s-1)}_{n}  \right)zR_{n}(z;s) = 0,	
	\end{split}
\end{equation}
and
\begin{equation}\label{S* pure n rec}
	S^*_{n+3}(z;s)
	- \left( \ga^{(s)}_{n+2}+ \ga^{(s+1)}_{n+1}   \right) z S^*_{n+2}(z;s)
	- \left(1-\ga^{(s+1)}_{n+1}\ga^{(s)}_{n+1} z^2 \right)S^*_{n+1}(z;s)
	+ \left(\theta^{(s+1)}_{n+1}+\ga^{(s+1)}_{n+1}  \right)zS^*_n(z;s) = 0 .
\end{equation}
where
	
\noindent
	\begin{minipage}{.5\linewidth}
		\begin{alignat}{2}
		&\varkappa^{(s)}_{n} &&= -\frac{g^{(s-2)}_n}{g^{(s)}_{n}}, \label{kappa} \\
		& \rho^{(s)}_n &&= \frac{E_n^{(s+1)}E_{n+1}^{(s-2)}}{E^{(s)}_{n+1}E^{(s-1)}_n}, \label{rho}
		\end{alignat}	
	\end{minipage}	
	\begin{minipage}{.5\linewidth}
		\begin{alignat}{2}
		&\ga^{(s)}_{n} &&= -\frac{g^{(s+1)}_n}{g^{(s)}_{n}}, \label{gamma} \\
		& \theta^{(s)}_n&&= \frac{E^{(s-2)}_n E^{(s+1)}_{n+1}}{E^{(s)}_{n+1}E^{(s-1)}_n}. \label{theta}
		\end{alignat}	
	\end{minipage}
\end{theorem}
\begin{proof}
In order to prove \eqref{R pure n rec} we use the following Dodgson condensation identities
	\begin{equation}\label{RSDodgson2}\begin{split}
	& \mathscr{E}_s\left\lbrace \begin{matrix} n+1 \\ n+2 \end{matrix} \right\rbrace \cdot \mathscr{E}_s\left\lbrace \begin{matrix}  n & n+1 & n+2 \\  0 & n+1 & n+2 \end{matrix} \right\rbrace =  \mathscr{E}_s \left\lbrace \begin{matrix} n & n+1  \\ 0 & n+2 \end{matrix} \right\rbrace \cdot \mathscr{E}_s\left\lbrace \begin{matrix}  n+1 & n+2  \\  n+1 & n+2 \end{matrix} \right\rbrace - \mathscr{E}_s\left\lbrace \begin{matrix} n & n+1  \\  n+1 & n+2 \end{matrix} \right\rbrace \cdot \mathscr{E}_s\left\lbrace \begin{matrix}  n+1 & n+2  \\ 0 & n+2  \end{matrix} \right\rbrace.
	\end{split}
	\end{equation}
		\begin{equation}\label{RSDodgson3}\begin{split}
	& \mathscr{E}_s\left\lbrace \begin{matrix}  n+1  \\  n+2 \end{matrix} \right\rbrace \cdot \mathscr{E}_s\left\lbrace \begin{matrix}  0 &  n+1 & n+2 \\  0 & n+1 & n+2 \end{matrix} \right\rbrace =  \mathscr{E}_s \left\lbrace \begin{matrix} 0 & n+1  \\ 0 & n+2 \end{matrix} \right\rbrace \cdot \mathscr{E}_s\left\lbrace \begin{matrix} n+1  & n+2  \\  n+1 & n+2 \end{matrix} \right\rbrace - \mathscr{E}_s\left\lbrace \begin{matrix} 0 & n+1  \\  n+1 & n+2 \end{matrix} \right\rbrace \cdot \mathscr{E}_s\left\lbrace \begin{matrix}  n+1 & n+2  \\ 0 & n+2 \end{matrix} \right\rbrace.
	\end{split}
	\end{equation}
Write \eqref{RSDodgson2} as
	\begin{equation}\label{RS11}
	\frac{\mathscr{E}_s\left\lbrace \begin{matrix}  n+1  \\  n+2 \end{matrix} \right\rbrace}{\mathscr{E}_s \left\lbrace \begin{matrix}  n+1& n+2 \\   n+1& n+2 \end{matrix} \right\rbrace}  \cdot \mathscr{E}_s\left\lbrace \begin{matrix}  n &  n+1  & n+2 \\  0 & n+1 &  n+2  \end{matrix} \right\rbrace =  \mathscr{E}_s \left\lbrace \begin{matrix} n &  n+1  \\ 0 & n+2 \end{matrix} \right\rbrace  - \frac{\mathscr{E}_s\left\lbrace \begin{matrix} n &  n+1  \\  n+1 &  n+2  \end{matrix} \right\rbrace}{\mathscr{E}_s \left\lbrace \begin{matrix}  n+1& n+2 \\   n+1& n+2 \end{matrix} \right\rbrace}  \cdot \mathscr{E}_s\left\lbrace \begin{matrix}   n+1 & n+2  \\ 0 &  n+2 \end{matrix} \right\rbrace.
	\end{equation}
So using \eqref{R_n+1 s}, \eqref{R_n s}, \eqref{zR_n s-2}, and \eqref{D DD E EE} we have
	\begin{equation}
	E^{(s-2)}_n R_{n+1}(z;s) = E^{(s-2)}_n zR_n(z;s-2)-\frac{E_n^{(s)}E_{n+1}^{(s-2)}}{E_{n+1}^{(s)}}R_n(z;s).
	\end{equation}
Therefore from \eqref{H} we have
	\begin{equation}\label{1st rec R}
		R_{n+1}(z;s) = zR_n(z;s-2)+\varkappa^{(s)}_nR_n(z;s), \qquad \varkappa^{(s)}_{n} = -\frac{g^{(s-2)}_n}{g^{(s)}_{n}}.
	\end{equation}
Now, lets us write \eqref{RSDodgson3} as
	\begin{equation}\label{RS22}
	\frac{\mathscr{E}_s\left\lbrace \begin{matrix}  n+1  \\  n+2 \end{matrix} \right\rbrace}{\mathscr{E}_s \left\lbrace \begin{matrix}  n+1& n+2 \\   n+1& n+2 \end{matrix} \right\rbrace}  \cdot \mathscr{E}_s\left\lbrace \begin{matrix}  0 &  n+1  & n+2 \\  0 & n+1 &  n+2  \end{matrix} \right\rbrace =  \mathscr{E}_s \left\lbrace \begin{matrix} 0 &  n+1  \\ 0 & n+2 \end{matrix} \right\rbrace  - \frac{\mathscr{E}_s\left\lbrace \begin{matrix} 0 &  n+1  \\  n+1 &  n+2  \end{matrix} \right\rbrace}{\mathscr{E}_s \left\lbrace \begin{matrix}  n+1& n+2 \\   n+1& n+2 \end{matrix} \right\rbrace}  \cdot \mathscr{E}_s\left\lbrace \begin{matrix}   n+1 & n+2  \\ 0 &  n+2 \end{matrix} \right\rbrace.
	\end{equation}
Therefore from \eqref{R_n+1 s}, \eqref{R_n s+1}, \eqref{zR_n s-1}, and \eqref{D DD E EE} we arrive at
	\begin{equation}
	E^{(s-1)}_nR_{n+1}(z;s)=E^{(s-1)}_nzR_n(z;s-1)-\frac{E_n^{(s+1)}E_{n+1}^{(s-2)}}{E^{(s)}_{n+1}}R_{n}(z;s+1),
	\end{equation}
and thus
	\begin{equation}\label{2nd rec R}
		R_{n+1}(z;s)=zR_n(z;s-1)-\rho^{(s)}_n R_{n}(z;s+1), \qquad \rho^{(s)}_n = \frac{E_n^{(s+1)}E_{n+1}^{(s-2)}}{E^{(s)}_{n+1}E^{(s-1)}_n}.
	\end{equation}	
From \eqref{2nd rec R} we have
	\begin{equation}\label{99}
	R_{n+2}(z;s)=zR_{n+1}(z;s-1)-\rho^{(s)}_{n+1} R_{n+1}(z;s+1),
	\end{equation}
and by writing \eqref{2nd rec R} for $s-1$ we obtain
	\begin{equation}
	zR_n(z;s-2)=R_{n+1}(z;s-1)+\rho^{(s-1)}_n R_{n}(z;s).
	\end{equation}
Combining this with \eqref{1st rec R} yields
	\begin{equation}\label{88}
	R_{n+1}(z;s-1) =R_{n+1}(z;s) -( \varkappa^{(s)}_n+ \rho^{(s-1)}_n) R_{n}(z;s). 
	\end{equation}
Plugging this into \eqref{99} gives
	\begin{equation}
	R_{n+2}(z;s)=z(R_{n+1}(z;s) -( \varkappa^{(s)}_n+ \rho^{(s-1)}_n) R_{n}(z;s) ) -\rho^{(s)}_{n+1} R_{n+1}(z;s+1).
	\end{equation}
Now, write \eqref{1st rec R} for $s+1$ to get
	\begin{equation}
	R_{n+1}(z;s+1) = zR_n(z;s-1)+\varkappa^{(s+1)}_nR_n(z;s+1).
	\end{equation}
After combining the last two equations we arrive at
	\begin{equation}
	R_{n+2}(z;s)=z(R_{n+1}(z;s) -( \varkappa^{(s)}_n+ \rho^{(s-1)}_n) R_{n}(z;s) ) -\rho^{(s)}_{n+1} zR_n(z;s-1)-\rho^{(s)}_{n+1}\varkappa^{(s+1)}_nR_n(z;s+1).
	\end{equation}
Note that from \eqref{2nd rec R} we have
	\begin{equation}
	R_{n}(z;s+1)=\frac{1}{\rho^{(s)}_n}\left[zR_n(z;s-1)-R_{n+1}(z;s)\right].
	\end{equation}
Combine the last two equations to get
	\begin{equation}
	\begin{split}
	R_{n+2}(z;s)&=z(R_{n+1}(z;s) -( \varkappa^{(s)}_n+ \rho^{(s-1)}_n) R_{n}(z;s) ) -\rho^{(s)}_{n+1} zR_n(z;s-1) - \frac{\rho^{(s)}_{n+1}\varkappa^{(s+1)}_n}{\rho^{(s)}_n}\left[zR_n(z;s-1)-R_{n+1}(z;s)\right].
	\end{split}
	\end{equation}
Rearranging terms and using $\rho^{(s)}_{n+1}\varkappa^{(s+1)}_n = \rho^{(s)}_n \varkappa^{(s)}_{n+1} $ yields
	\begin{equation}
	\begin{split}
	R_{n+2}(z;s)&=\left(z+\varkappa^{(s)}_{n+1}\right)R_{n+1}(z;s) -( \varkappa^{(s)}_n+ \rho^{(s-1)}_n)z R_{n}(z;s)  -\left( \rho^{(s)}_{n+1}+\varkappa^{(s)}_{n+1}\right)zR_n(z;s-1).
	\end{split}
	\end{equation}
By writing \eqref{88} for $n-1$ we obtain
	\begin{equation}
	R_{n}(z;s-1) =R_{n}(z;s) -( \varkappa^{(s)}_{n-1}+ \rho^{(s-1)}_{n-1}) R_{n-1}(z;s). 
	\end{equation}
Now, we combine the last two equations to arrive at
	\begin{equation}
	\begin{split}
	R_{n+2}(z;s)&=\left(z+\varkappa^{(s)}_{n+1}\right)R_{n+1}(z;s) -( \varkappa^{(s)}_n+ \rho^{(s-1)}_n)z R_{n}(z;s)  -\left( \rho^{(s)}_{n+1}+\varkappa^{(s)}_{n+1}\right)\left(zR_{n}(z;s) -( \varkappa^{(s)}_{n-1}+ \rho^{(s-1)}_{n-1}) zR_{n-1}(z;s) \right).
	\end{split}
	\end{equation}
regroup terms and replace $n$ by $n+1$ to arrive at the desired recurrence relation \eqref{R pure n rec}.
	
Finally, proving \eqref{S* pure n rec} demands using the following Dodgson condensation identities
\begin{equation}\label{RSDodgson5}
\begin{split}
& \mathscr{E}_s\left\lbrace \begin{matrix} n+2 \\ n+1 \end{matrix} \right\rbrace \cdot \mathscr{E}_s\left\lbrace \begin{matrix} 0 & n+1 & n+2 \\  n & n+1 & n+2 \end{matrix} \right\rbrace =  \mathscr{E}_s \left\lbrace \begin{matrix} 0 & n+2  \\ n & n+1 \end{matrix} \right\rbrace \cdot \mathscr{E}_s\left\lbrace \begin{matrix} n+1 & n+2  \\   n+1  & n+2 \end{matrix} \right\rbrace - \mathscr{E}_s\left\lbrace \begin{matrix} 0 & n+2  \\ n+1 & n+2 \end{matrix} \right\rbrace \cdot \mathscr{E}_s\left\lbrace \begin{matrix} n+1 & n+2  \\  n & n+1 \end{matrix} \right\rbrace.
\end{split}
\end{equation}

\begin{equation}\label{RSDodgson6}
\begin{split}
& \mathscr{E}_s\left\lbrace \begin{matrix} n+2 \\ n+1 \end{matrix} \right\rbrace \cdot \mathscr{E}_s\left\lbrace \begin{matrix} 0 & n+1 & n+2 \\  0 & n+1 & n+2 \end{matrix} \right\rbrace =  \mathscr{E}_s \left\lbrace \begin{matrix} 0 & n+2  \\ 0 & n+1 \end{matrix} \right\rbrace \cdot \mathscr{E}_s\left\lbrace \begin{matrix} n+1 & n+2  \\  n+1 & n+2 \end{matrix} \right\rbrace - \mathscr{E}_s\left\lbrace \begin{matrix} 0 & n+2  \\  n+1 & n+2 \end{matrix} \right\rbrace \cdot \mathscr{E}_s\left\lbrace \begin{matrix} n+1 & n+2  \\  0 & n+1 \end{matrix} \right\rbrace.
\end{split}
\end{equation}
Write \eqref{RSDodgson5} as
\begin{equation}\label{RS1}
\frac{\mathscr{E}_s \left\lbrace \begin{matrix} n+2 \\  n+1 \end{matrix} \right\rbrace}{\mathscr{E}_s \left\lbrace \begin{matrix}  n+1& n+2 \\   n+1& n+2 \end{matrix} \right\rbrace} \cdot \mathscr{E}_s\left\lbrace \begin{matrix} 0 & n+1 &  n+2  \\  n &  n+1  & n+2 \end{matrix} \right\rbrace = \mathscr{E}_s \left\lbrace \begin{matrix} 0 & n+2  \\ n &  n+1  \end{matrix} \right\rbrace - \frac{\mathscr{E}_s\left\lbrace \begin{matrix} n+1 & n+2  \\  n &  n+1 \end{matrix} \right\rbrace}{\mathscr{E}_s \left\lbrace \begin{matrix}  n+1& n+2 \\   n+1& n+2 \end{matrix} \right\rbrace} \cdot \mathscr{E}_s \left\lbrace \begin{matrix} 0 & n+2 \\  n+1 & n+2 \end{matrix} \right\rbrace.  
\end{equation}
So we have
\begin{equation}
E^{(s+1)}_{n}S^*_{n+1}(z;s)= E^{(s+1)}_n S^*_n(z;s+1)-\frac{E^{(s)}_n}{E^{(s)}_{n+1}} E^{(s+1)}_{n+1} z S^*_{n}(z;s).
\end{equation}
and thus from \eqref{H} we have 
\begin{equation}\label{1st rec S^*}
	S^*_{n+1}(z;s)=  S^*_n(z;s+1)+ \ga^{(s)}_n z S^*_{n}(z;s), \qquad \ga^{(s)}_{n} = -\frac{g^{(s+1)}_n}{g^{(s)}_{n}}.
\end{equation}Now, rewrite \eqref{RSDodgson6} as
\begin{equation}\label{RS2}
\frac{\mathscr{E}_s \left\lbrace \begin{matrix} n+2 \\  n+1 \end{matrix} \right\rbrace}{\mathscr{E}_s \left\lbrace \begin{matrix}  n+1& n+2 \\   n+1& n+2 \end{matrix} \right\rbrace} \cdot \mathscr{E}_s\left\lbrace \begin{matrix} 0 & n+1 &  n+2  \\  0 &  n+1  & n+2 \end{matrix} \right\rbrace = \mathscr{E}_s \left\lbrace \begin{matrix} 0 & n+2  \\ 0 &  n+1  \end{matrix} \right\rbrace - \frac{\mathscr{E}_s\left\lbrace \begin{matrix} n+1 & n+2  \\  0 &  n+1 \end{matrix} \right\rbrace}{\mathscr{E}_s \left\lbrace \begin{matrix}  n+1& n+2 \\   n+1& n+2 \end{matrix} \right\rbrace} \cdot \mathscr{E}_s \left\lbrace \begin{matrix} 0 & n+2 \\  n+1 & n+2 \end{matrix} \right\rbrace .
\end{equation}
So
\begin{equation}
E^{(s-1)}_n S^*_{n+1}(z;s)= E^{(s-1)}_nS^*_{n}(z;s-1) - \frac{E^{(s-2)}_n E^{(s+1)}_{n+1}}{E^{(s)}_{n+1}}zS^*_{n}(z;s-2),
\end{equation}
and thus
\begin{equation}\label{2nd rec S^*}
	S^*_{n+1}(z;s)= S^*_{n}(z;s-1) -\theta^{(s)}_n zS^*_{n}(z;s-2),
\qquad \theta^{(s)}_n= \frac{E^{(s-2)}_n E^{(s+1)}_{n+1}}{E^{(s)}_{n+1}E^{(s-1)}_n}.
\end{equation}
Now let us find the pure degree-recurrence relations for $S^*$'s. Writing \eqref{1st rec S^*} for $n+1$ we get 
	\begin{equation}
	S^*_{n+2}(z;s)=  S^*_{n+1}(z;s+1)+ \ga^{(s)}_{n+1} z S^*_{n+1}(z;s),
	\end{equation}
Now we write \eqref{2nd rec S^*} for $s+1$:
	\begin{equation}
	S^*_{n+1}(z;s+1)= S^*_{n}(z;s) -\theta^{(s+1)}_n zS^*_{n}(z;s-1).
	\end{equation}
Combining the last two equations yields
	\begin{equation}
	S^*_{n+2}(z;s)=  S^*_{n}(z;s) -\theta^{(s+1)}_n zS^*_{n}(z;s-1)+ \ga^{(s)}_{n+1} z S^*_{n+1}(z;s).
	\end{equation}
Write this for $n\mapsto n+1$:
	\begin{equation}
	S^*_{n+3}(z;s)=  S^*_{n+1}(z;s) -\theta^{(s+1)}_{n+1} zS^*_{n+1}(z;s-1)+ \ga^{(s)}_{n+2} z S^*_{n+2}(z;s),
	\end{equation}
and write \eqref{1st rec S^*} for $r\mapsto r-1$ to obtain
	\begin{equation}
	S^*_{n+1}(z;s-1)=  S^*_n(z;s)+ \ga^{(s-1)}_n z S^*_{n}(z;s-1).
	\end{equation}
Combining the last two equations gives
	\begin{equation}
	S^*_{n+3}(z;s)=  S^*_{n+1}(z;s) -\theta^{(s+1)}_{n+1} zS^*_n(z;s) -\theta^{(s+1)}_{n+1} \ga^{(s-1)}_n z^2 S^*_{n}(z;s-1)+ \ga^{(s)}_{n+2} z S^*_{n+2}(z;s).
	\end{equation}
Now write \eqref{2nd rec S^*} for $s\mapsto s+1$ and solve for $zS_{n}(z;s-1)$:
	\begin{equation}
	zS^*_{n}(z;s-1)=\frac{1}{\theta^{(s+1)}_n} \left( S^*_{n}(z;s) - S^*_{n+1}(z;s+1) \right).
	\end{equation}
Combine the last two equations to get
	\begin{equation}
	S^*_{n+3}(z;s)=  S^*_{n+1}(z;s) -\theta^{(s+1)}_{n+1} zS^*_n(z;s) - \frac{\theta^{(s+1)}_{n+1} \ga^{(s-1)}_n z}{\theta^{(s+1)}_n} \left( S^*_{n}(z;s) - S^*_{n+1}(z;s+1) \right)+ \ga^{(s)}_{n+2} z S^*_{n+2}(z;s).
	\end{equation}
By using $\theta^{(s+1)}_{n+1} \ga^{(s-1)}_n = \theta^{(s+1)}_n \ga^{(s+1)}_{n+1}  $ and rearranging terms we obtain
	\begin{equation}
	S^*_{n+3}(z;s)= \ga^{(s)}_{n+2} z S^*_{n+2}(z;s)+  S^*_{n+1}(z;s) -\left(\theta^{(s+1)}_{n+1}+\ga^{(s+1)}_{n+1}  \right)zS^*_n(z;s) + \ga^{(s+1)}_{n+1}  zS^*_{n+1}(z;s+1).
	\end{equation}
Write \eqref{1st rec S^*} for $n \mapsto n+1$ and solve for $S^*_{n+1}(z;s+1)$:
	\begin{equation}
	S^*_{n+1}(z;s+1)=    S^*_{n+2}(z;s)- \ga^{(s)}_{n+1} z S^*_{n+1}(z;s).
	\end{equation}
Now, we combine the last two equations and regroup the terms to obtain \eqref{S* pure n rec}. 
\end{proof}

\begin{remark}[Mixed Recurrence Relations] \normalfont
Recall that we can replace a $2j-k$ polynomial in terms of a $j-2k$ polynomial using Theorem \ref{Thm S*-P & Q*-R}.  Using such replacements in several identities found in this section, we can arrive at recurrence relations involving both $2j-k$ and $j-2k$ polynomials.  To that end, \eqref{1st rec P} can be written as
	\begin{equation}
		P_{n+1}(z;r) = \frac{(-1)^nE_{n}^{(r+n-3)}}{D_{n}^{(r-1)}}zS^*_n(z;r+n-3)+\de^{(r)}_nP_n(z;r).
	\end{equation}
Also \eqref{2nd rec P} can be written as one of the following three identities
	\begin{align}
		P_{n+1}(z;r) & =\frac{(-1)^nE_{n}^{(r+n-1)}}{D_{n}^{(r+1)}}zS^*_n(z;r+n-1)-\eta^{(r)}_n P_{n}(z;r+2),\\
		P_{n+1}(z;r)&=zP_n(z;r+1)-\eta^{(r)}_n \frac{(-1)^nE_{n}^{(r+n)}}{D_{n}^{(r+2)}}S^*_n(z;r+n),\\
		P_{n+1}(z;r)&=\frac{(-1)^nE_{n}^{(r+n-1)}}{D_{n}^{(r+1)}}zS^*_n(z;r+n-1)-\eta^{(r)}_n \frac{(-1)^nE_{n}^{(r+n)}}{D_{n}^{(r+2)}}S^*_n(z;r+n).
	\end{align}
In an identical way we can obtain the following mixed recurrence relations listed below
	\begin{equation}\label{1st rec Q^* eq}
		Q^*_{n+1}(z;r)=  \frac{(-1)^nE_n^{(r+n+3)}}{D_{n}^{(r+2)}} R_n(z;r+n+3)+ \be^{(r)}_n z Q^*_{n}(z;r),
	\end{equation}
	\begin{equation}\label{2nd rec Q^* eq1}
	Q^*_{n+1}(z;r)= \frac{(-1)^nE_n^{(r+n+2)}}{D_{n}^{(r+1)}} R_n(z;r+n+2) -\al^{(r)}_n zQ^*_{n}(z;r-1),
    \end{equation}
    \begin{equation}\label{2nd rec Q^* eq2}
    	Q^*_{n+1}(z;r)= Q^*_{n}(z;r+1) -\al^{(r)}_n \frac{(-1)^nE_n^{(r+n)}}{D_{n}^{(r-1)}} zR_n(z;r+n),
    \end{equation}
    \begin{equation}\label{2nd rec Q^* eq3}
Q^*_{n+1}(z;r)= \frac{(-1)^nE_n^{(r+n+2)}}{D_{n}^{(r+1)}} R_n(z;r+n+2) -\al^{(r)}_n \frac{(-1)^nE_n^{(r+n)}}{D_{n}^{(r-1)}} zR_n(z;r+n),
\end{equation}	
\begin{equation}\label{1st rec R eq}
	R_{n+1}(z;s) = \frac{(-1)^nD_{n}^{(s-n-3)}}{E_{n}^{(s-2)}}zQ^*_n(z;s-n-3)+\varkappa^{(s)}_nR_n(z;s),
\end{equation}
\begin{equation}\label{2nd rec R eq1}
	R_{n+1}(z;s)=\frac{(-1)^nD_{n}^{(s-n-2)}}{E_{n}^{(s-1)}}zQ^*_n(z;s-n-2)-\rho^{(s)}_n R_{n}(z;s+1),
\end{equation}
\begin{equation}\label{2nd rec R eq2}
R_{n+1}(z;s)=zR_n(z;s-1)-\rho^{(s)}_n \frac{(-1)^nD_{n}^{(s-n)}}{E_{n}^{(s+1)}}Q^*_n(z;s-n),
\end{equation}
\begin{equation}\label{2nd rec R eq3}
	R_{n+1}(z;s)=\frac{(-1)^nD_{n}^{(s-n-2)}}{E_{n}^{(s-1)}}zQ^*_n(z;s-n-2)-\rho^{(s)}_n \frac{(-1)^nD_{n}^{(s-n)}}{E_{n}^{(s+1)}}Q^*_n(z;s-n),
\end{equation}

\begin{equation}\label{1st rec S^* eq}
	S^*_{n+1}(z;s)=  \frac{(-1)^nD^{(s-n+3)}_n}{E_{n}^{(s+1)}}P_n(z;s-n+3)+ \ga^{(s)}_n z S^*_{n}(z;s),
\end{equation}

\begin{equation}\label{2nd rec S^* eq1}
			S^*_{n+1}(z;s)= \frac{(-1)^nD^{(s-n+1)}_n}{E_{n}^{(s-1)}}P_n(z;s-n+1) -\theta^{(s)}_n zS^*_{n}(z;s-2),
\end{equation}
		
\begin{equation}\label{2nd rec S^* eq2}
			S^*_{n+1}(z;s)= S^*_{n}(z;s-1) -\theta^{(s)}_n \frac{(-1)^nD^{(s-n)}_n}{E_{n}^{(s-2)}}zP_n(z;s-n),
\end{equation}
and
\begin{equation}\label{2nd rec S^* eq3}
			S^*_{n+1}(z;s)= \frac{(-1)^nD^{(s-n+1)}_n}{E_{n}^{(s-1)}}P_n(z;s-n+1) -\theta^{(s)}_n \frac{(-1)^nD^{(s-n)}_n}{E_{n}^{(s-2)}}zP_n(z;s-n).
\end{equation}
Here the identities \eqref{1st rec Q^* eq}, \eqref{1st rec R eq}, and \eqref{1st rec S^* eq} are respectively equivalent to \eqref{1st rec Q^*},  \eqref{1st rec R}, and \eqref{1st rec S^*} . Identities \eqref{2nd rec Q^* eq1}, \eqref{2nd rec Q^* eq2} and \eqref{2nd rec Q^* eq3} are equivalent to \eqref{2nd rec Q^*}, identities \eqref{2nd rec R eq1}, \eqref{2nd rec R eq2} and \eqref{2nd rec R eq3} are equivalent to \eqref{2nd rec R}, and  identities \eqref{2nd rec S^* eq1}, \eqref{2nd rec S^* eq2} and \eqref{2nd rec S^* eq3} are equivalent to \eqref{2nd rec S^*}.
\end{remark}

\begin{theorem}\label{offset-RR}
The third order pure-offset recurrence relations for the $2j-k$ and $j-2k$ polynomials are given by
	\begin{equation}\label{P pure r rec}
	\eta^{(r+1)}_n P_{n}(z;r+3) - zP_n(z;r+2) + \de^{(r+1)}_nP_n(z;r+1) + zP_n(z;r) =0,
	\end{equation}
	\begin{equation}\label{Q pure r rec}
	Q^*_n(z;r+3) - Q^*_{n}(z;r+2) + \be^{(r+1)}_n z Q^*_{n}(z;r+1)  +\al^{(r+1)}_n zQ^*_{n}(z;r)=0,
	\end{equation}
\begin{equation}\label{R pure s rec}
\rho^{(s+2)}_n R_{n}(z;s+3)+\varkappa^{(s+2)}_nR_n(z;s+2)- zR_n(z;s+1)+ zR_n(z;s) =0,
\end{equation}	and
\begin{equation}\label{S* pure s rec}
S^*_n(z;s+3)+ \ga^{(s+2)}_n z S^*_{n}(z;s+2) - S^*_{n}(z;s+1) +\theta^{(s+2)}_n zS^*_{n}(z;s)=0.
\end{equation}
\end{theorem}	
\begin{proof}
	Combining \eqref{1st rec P} and \eqref{2nd rec P} we will get
	\begin{equation}\label{P r}
	\eta^{(r)}_n P_{n}(z;r+2) - zP_n(z;r+1) + zP_n(z;r-1) + \de^{(r)}_nP_n(z;r)=0.
	\end{equation}
	Also from \eqref{1st rec Q^*} and \eqref{2nd rec Q^*} we obtain
	\begin{equation}\label{Q r}
	Q^*_n(z;r+2) - Q^*_{n}(z;r+1) + \be^{(r)}_n z Q^*_{n}(z;r)  +\al^{(r)}_n zQ^*_{n}(z;r-1)=0.
	\end{equation}
	writing \eqref{P r} and \eqref{Q r} under the replacement $r \mapsto r+1$ respectively gives \eqref{P pure r rec} and \eqref{Q pure r rec}. From \eqref{1st rec R} and \eqref{2nd rec R} we have
	\begin{equation}\label{R s}
	zR_n(z;s-2)+\varkappa^{(s)}_nR_n(z;s) - zR_n(z;s-1)+\rho^{(s)}_n R_{n}(z;s+1)=0.
	\end{equation}
	From \eqref{1st rec S^*} and \eqref{2nd rec S^*} we find 
	\begin{equation}\label{S s}
	S^*_n(z;s+1)+ \ga^{(s)}_n z S^*_{n}(z;s) - S^*_{n}(z;s-1) +\theta^{(s)}_n zS^*_{n}(z;s-2)=0.
	\end{equation}
	 The replacement $s \mapsto s+2$ in \eqref{R s} and \eqref{S s} respectively gives \eqref{R pure s rec} and \eqref{S* pure s rec}.
\end{proof}

\subsection{Equivalent Dodgson Condensation identities} 	If one starts with 
\begin{equation}\label{Dodgson1'}
\begin{split}
& \mathscr{D}_r\left\lbrace \begin{matrix}  0  \\  n+2  \end{matrix} \right\rbrace \cdot \mathscr{D}_r\left\lbrace \begin{matrix} 0 &  n+1  & n+2 \\  0 & n+1 &  n+2  \end{matrix} \right\rbrace =  \mathscr{D}_r\left\lbrace \begin{matrix} 0 &  n+1  \\ 0 & n+2  \end{matrix} \right\rbrace \cdot \mathscr{D}_r\left\lbrace \begin{matrix}  0  & n+2  \\  n+1 &  n+2  \end{matrix} \right\rbrace - \mathscr{D}_r\left\lbrace \begin{matrix} 0 &  n+1   \\  n+1 &  n+2  \end{matrix} \right\rbrace \cdot \mathscr{D}_r\left\lbrace \begin{matrix}   0  & n+2  \\ 0 &  n+2  \end{matrix} \right\rbrace,
\end{split}
\end{equation}
after similar simplifications one arrives at 
\begin{equation}\label{Rec 1'}
P_{n+1}(z;r+2)  = zP_n(z;r+1) - \frac{ h^{(r+1)}_{n}}{ h^{(r+2)}_{n}} P_n(z;r+2).
\end{equation}
But this is equivalent to \eqref{1st rec P} as one gets the above equation from \eqref{1st rec P} under the replacement $r\mapsto r+2$. Similarly, If we start with
\begin{equation}\label{Dodgson2'}
\begin{split}
& \mathscr{D}_r\left\lbrace \begin{matrix}  n+2  \\  0  \end{matrix} \right\rbrace \cdot \mathscr{D}_r\left\lbrace \begin{matrix} 0 &  n+1  & n+2 \\  0 & n+1 &  n+2  \end{matrix} \right\rbrace = \mathscr{D}_r\left\lbrace \begin{matrix} 0 &  n+2  \\ 0 & n+1  \end{matrix} \right\rbrace \cdot \mathscr{D}_r\left\lbrace \begin{matrix}  n+1  & n+2  \\  0 &  n+2  \end{matrix} \right\rbrace - \mathscr{D}_r\left\lbrace \begin{matrix} n+1 &  n+2   \\  0 &  n+1  \end{matrix} \right\rbrace \cdot \mathscr{D}_r\left\lbrace \begin{matrix}   0  & n+2  \\ 0 &  n+2  \end{matrix} \right\rbrace.
\end{split}
\end{equation}
After simplifications, this translates to 
\begin{equation}
Q^*_{n+1}(z;r-1)=  Q^*_n(z;r+1)-\frac{h^{(r+1)}_n}{h^{(r-1)}_{n}} z Q^*_{n}(z;r-1).
\end{equation}
However, this is equivalent to \eqref{1st rec Q^*} if we replace $r$ by $r-1$ in \eqref{1st rec Q^*}. Also note that 
\begin{equation}\label{Dodgson3'}
\begin{split}
& \mathscr{E}_s\left\lbrace \begin{matrix}  0  \\  n+2  \end{matrix} \right\rbrace \cdot \mathscr{E}_s\left\lbrace \begin{matrix} 0 &  n+1  & n+2 \\  0 & n+1 &  n+2  \end{matrix} \right\rbrace =  \mathscr{E}_s\left\lbrace \begin{matrix} 0 &  n+1  \\ 0 & n+2  \end{matrix} \right\rbrace \cdot \mathscr{E}_s\left\lbrace \begin{matrix}  0  & n+2  \\  n+1 &  n+2  \end{matrix} \right\rbrace - \mathscr{E}_s\left\lbrace \begin{matrix} 0 &  n+1   \\  n+1 &  n+2  \end{matrix} \right\rbrace \cdot \mathscr{E}_s\left\lbrace \begin{matrix}   0  & n+2  \\ 0 &  n+2  \end{matrix} \right\rbrace.
\end{split}
\end{equation}
This reduces after simplifications to
\begin{equation}
R_{n+1}(z;s+1)=zR_n(z;s-1)-\frac{g^{(s-1)}_n}{g^{(s+1)}_n}R_n(z;s+1),
\end{equation}
which is equivalent to \eqref{1st rec R} if one replaces $s$ with $s+1$ in \eqref{1st rec R}. And finally
\begin{equation}\label{Dodgson4'}
\begin{split}
& \mathscr{E}_s\left\lbrace \begin{matrix}  n+2  \\  0  \end{matrix} \right\rbrace \cdot \mathscr{E}_s\left\lbrace \begin{matrix} 0 &  n+1  & n+2 \\  0 & n+1 &  n+2  \end{matrix} \right\rbrace =  \mathscr{E}_s\left\lbrace \begin{matrix} 0 &  n+2  \\ 0 & n+1  \end{matrix} \right\rbrace \cdot \mathscr{E}_s\left\lbrace \begin{matrix}  n+1  & n+2  \\  0 &  n+2  \end{matrix} \right\rbrace - \mathscr{E}_s\left\lbrace \begin{matrix} n+1 &  n+2   \\  0 &  n+1  \end{matrix} \right\rbrace \cdot \mathscr{E}_s\left\lbrace \begin{matrix}   0  & n+2  \\ 0 &  n+2  \end{matrix} \right\rbrace
\end{split}
\end{equation}
reduces after simplifications to
\begin{equation}
S^*_{n+1}(z;s-2)=S^*_n(z;s-1)-\frac{g^{(s-1)}_n}{g^{(s-2)}_n}zS^*_n(z;s-2),
\end{equation}
which is equivalent to \eqref{1st rec S^*} if one replaces $s$ with $s-2$ in \eqref{1st rec S^*}. In a similar fashion one can find four other Dodgson Condensation identities which give rise to  recurrence relations equivalent to \eqref{2nd rec P}, \eqref{2nd rec Q^*}, \eqref{2nd rec R}, and \eqref{2nd rec S^*}, these Dodgson Condensation identities are respectively given by
\begin{alignat}{3}
	& \mathscr{D}_r\left\lbrace \begin{matrix}  0  \\  n+2  \end{matrix} \right\rbrace \cdot \mathscr{D}_r\left\lbrace \begin{matrix} 0 &  1  & n+2 \\  0 & n+1 &  n+2  \end{matrix} \right\rbrace &&=  \mathscr{D}_r\left\lbrace \begin{matrix} 0 &  1  \\ 0 & n+2  \end{matrix} \right\rbrace \cdot \mathscr{D}_r\left\lbrace \begin{matrix}  0  & n+2  \\  n+1 &  n+2  \end{matrix} \right\rbrace &&- \mathscr{D}_r\left\lbrace \begin{matrix} 0 &  1   \\  n+1 &  n+2  \end{matrix} \right\rbrace \cdot \mathscr{D}_r\left\lbrace \begin{matrix}   0  & n+2  \\ 0 &  n+2  \end{matrix} \right\rbrace,
	\\ & \mathscr{D}_r\left\lbrace \begin{matrix}  n+2  \\  0  \end{matrix} \right\rbrace \cdot \mathscr{D}_r\left\lbrace \begin{matrix} 0 &  n+1  & n+2 \\  0 & 1 &  n+2  \end{matrix} \right\rbrace &&=  \mathscr{D}_r\left\lbrace \begin{matrix}  0 & n+2 \\ 0 &  1  \end{matrix} \right\rbrace \cdot \mathscr{D}_r\left\lbrace \begin{matrix}   n+1 &  n+2 \\  0  & n+2 \end{matrix} \right\rbrace &&- \mathscr{D}_r\left\lbrace \begin{matrix}  n+1 &  n+2 \\  0 &  1   \end{matrix} \right\rbrace \cdot \mathscr{D}_r\left\lbrace \begin{matrix}  0 &  n+2 \\   0  & n+2   \end{matrix} \right\rbrace,
	\\ & \mathscr{E}_s\left\lbrace \begin{matrix}  0  \\  n+2  \end{matrix} \right\rbrace \cdot \mathscr{E}_s\left\lbrace \begin{matrix} 0 &  1  & n+2 \\  0 & n+1 &  n+2  \end{matrix} \right\rbrace &&=  \mathscr{E}_s\left\lbrace \begin{matrix} 0 &  1  \\ 0 & n+2  \end{matrix} \right\rbrace \cdot \mathscr{E}_s\left\lbrace \begin{matrix}  0  & n+2  \\  n+1 &  n+2  \end{matrix} \right\rbrace &&- \mathscr{E}_s\left\lbrace \begin{matrix} 0 &  1   \\  n+1 &  n+2  \end{matrix} \right\rbrace \cdot \mathscr{E}_s\left\lbrace \begin{matrix}   0  & n+2  \\ 0 &  n+2  \end{matrix} \right\rbrace,
	\\ & \mathscr{E}_s\left\lbrace \begin{matrix}  n+2  \\  0  \end{matrix} \right\rbrace \cdot \mathscr{E}_s\left\lbrace \begin{matrix} 0 &  n+1  & n+2 \\  0 & 1 &  n+2  \end{matrix} \right\rbrace &&=  \mathscr{E}_s\left\lbrace \begin{matrix}  0 & n+2 \\ 0 &  1  \end{matrix} \right\rbrace \cdot \mathscr{E}_s\left\lbrace \begin{matrix}   n+1 &  n+2 \\  0  & n+2 \end{matrix} \right\rbrace &&- \mathscr{E}_s\left\lbrace \begin{matrix}  n+1 &  n+2 \\  0 &  1   \end{matrix} \right\rbrace \cdot \mathscr{E}_s\left\lbrace \begin{matrix}  0 &  n+2 \\   0  & n+2   \end{matrix} \right\rbrace.
\end{alignat}

\subsection{Relationships of polynomial tails, recurrence coefficients and determinants}\label{subsec poly tails, rec rels and dets}

There are many redundancies amongst the various coefficients that have been defined thus far. 
In this subsection we provide the necessary relationships between these coefficients to remove the redundancy, 
but also to link these to certain polynomial coefficients, in particular the tail coefficients,
and to the determinants and thus the norms of the polynomials. Such relations form an algorithmic path to compute these norms.

 
 \begin{lemma}
 	There exist interrelationships between the recurrence coefficients $\al^{(r)}_n$, $\be^{(r)}_n$, $\de^{(r)}_n$, and $\eta^{(r)}_n$ and also between $\varkappa^{(s)}_n$, $\rho^{(s)}_n$, $\ga^{(s)}_n$, and $\theta^{(s)}_n$ as described by
 	
 	 \noindent\begin{minipage}{.5\linewidth}
 		\begin{alignat}{2}
 		&\de^{(r)}_n \al^{(r)}_n &&= \be^{(r)}_n \eta^{(r)}_n, \label{2j-k interrels}
 		\end{alignat}	
 	\end{minipage}	
 	\begin{minipage}{.5\linewidth}
 		\begin{alignat}{2}
 		&\varkappa^{(s)}_n \theta^{(s)}_n &&= \ga^{(s)}_n \rho^{(s)}_n. \label{j-2k interrels}
 		\end{alignat}	
 	\end{minipage}
 \end{lemma}
\begin{proof}
	These relationships simply follow from the definitions of the recurrence relations in terms of the $2j-k$ and $j-2k$ determinants.
\end{proof}
 \begin{lemma}
 The $2j-k$ and the $j-2k$ systems each have only two independent recurrence coefficients due to the following interrelationships:
 	
 \noindent\begin{minipage}{.5\linewidth}
 		\begin{alignat}{2}
 		&\beta^{(r)}_{n} &&= -\frac{1}{\delta^{(r+1)}_{n}\delta^{(r+2)}_{n}}, \label{abcd} \\
 		& \alpha^{(r)}_{n} && = -\frac{\eta^{(r)}_{n}}{\delta^{(r)}_{n}\delta^{(r+1)}_{n}\delta^{(r+2)}_{n}}, \label{ABCD}
 		\end{alignat}	
 	\end{minipage}	
 	\begin{minipage}{.5\linewidth}
 		\begin{alignat}{2}
 		&\varkappa^{(s)}_n &&= -\frac{1}{\gamma^{(s-1)}_{n}\gamma^{(s-2)}_{n}}, \label{EFGH} \\
 		& \rho^{(s)}_{n}  &&= - \frac{\theta^{(s)}_{n}}{\gamma^{(s)}_{n}\gamma^{(s-1)}_{n}\gamma^{(s-2)}_{n}}  . \label{efgh}
 		\end{alignat}	
 	\end{minipage}
 \end{lemma}
 \begin{proof}
 	From \eqref{1st rec Q^*} we have
 	\begin{equation}
 	\be^{(r)}_{n} = -\frac{h^{(r+2)}_n}{h^{(r)}_{n}} = - \frac{h^{(r+1)}_n}{h^{(r)}_{n}} \frac{h^{(r+2)}_n}{h^{(r+1)}_{n}} = -\frac{1}{\delta^{(r+1)}_{n}\delta^{(r+2)}_{n}},
 	\end{equation}
 	by the definition of $\delta^{(r)}_{n}$ in \eqref{1st rec P}. From \eqref{2nd rec Q^*} we have
 	\begin{equation}
 	\al^{(r)}_n= \frac{D^{(r-1)}_n D^{(r+2)}_{n+1}}{D^{(r)}_{n+1}D^{(r+1)}_n} = \frac{D^{(r-1)}_n D^{(r+2)}_{n}}{D^{(r)}_{n}D^{(r+1)}_n} \frac{h^{(r+2)}_n}{h^{(r)}_n}. 
 	\end{equation}
 	Also from \eqref{2nd rec P} we have
 	\begin{equation}
 	\eta^{(r)}_n = \frac{D_n^{(r+2)}D_{n+1}^{(r-1)}}{D^{(r)}_{n+1}D^{(r+1)}_n}= \frac{D_n^{(r+2)}D_{n}^{(r-1)}}{D^{(r)}_{n}D^{(r+1)}_n}\frac{h^{(r-1)}_n}{h^{(r)}_n}.
 	\end{equation}
 	Combining the last two equations gives
 	\begin{equation}
 	\al^{(r)}_n = \frac{h^{(r+2)}_n}{h^{(r-1)}_n} 	\eta^{(r)}_n = \frac{h^{(r+2)}_n}{h^{(r+1)}_n} \frac{h^{(r+1)}_n}{h^{(r)}_n} \frac{h^{(r)}_n}{h^{(r-1)}_n} \eta^{(r)}_n = -\frac{\eta^{(r)}_n}{\delta^{(r)}_{n}\delta^{(r+1)}_{n}\delta^{(r+2)}_{n}},
 	\end{equation}
 	again by the definition of $\delta^{(r)}_{n}$ in \eqref{1st rec P}. The relations \eqref{EFGH} and \eqref{efgh} can be proven similarly.
 \end{proof}
 
 The following result relates the tail coefficients of the polynomials to products of certain recurrence relation coefficients.
 
\begin{lemma}\label{polytails-reccoeff}
The polynomial tails can be expressed in terms of the recurrence coefficients as follows
	
\noindent
	 \begin{minipage}{.5\linewidth}
	 	\begin{alignat}{2}
	 	&P_n(0;r) &&= \prod^{n-1}_{\ell=0}\de^{(r)}_{\ell}, \label{Pn0r} \\
	 	& Q_n(0;r)  &&= \prod^{n-1}_{\ell=0}\be^{(r)}_{\ell}, \label{Qn0r}
	 	\end{alignat}	
	 \end{minipage}	
	 \begin{minipage}{.5\linewidth}
	 	\begin{alignat}{2}
	 	&R_n(0;s) &&= \prod^{n-1}_{\ell=0}\varkappa^{(s)}_{\ell}, \label{Rn0s} \\
	 	& S_n(0;s)  &&= \prod^{n-1}_{\ell=0}\ga^{(s)}_{\ell}. \label{Sn0s}
	 	\end{alignat}	
	 \end{minipage}
\end{lemma}
\begin{proof}
	These identities are immediate consequences of the recurrence relations \eqref{1st rec P}, \eqref{1st rec Q^*}, \eqref{1st rec R}, and \eqref{1st rec S^*}. For example evaluating \eqref{1st rec P} at $z=0$ gives $P_{n+1}(0;r)=\de^{(r)}_nP_n(0;r)$ and thus \eqref{Pn0r} follows because $P_0(0;r)=1$. Now, if we match the coefficients of $z^{n+1}$ in \eqref{1st rec Q^*} we obtain $Q_{n+1}(0;r)=\be^{(r)}_{n}Q_{n}(0;r)$, which yields \eqref{Qn0r} because  $Q_0(0;r)=1$. The identities \eqref{Rn0s} and \eqref{Sn0s} can be shown similarly using respectively \eqref{1st rec R}, and \eqref{1st rec S^*}.
\end{proof}
The next set of relations are just the inverse relations to the ones in the preceding Lemma.
\begin{corollary}
The polynomial tails satisfy the following pure-$n$ recurrence relations
	
\noindent\begin{minipage}{.5\linewidth}
		\begin{alignat}{2}
		&P_{n+1}(0;r) &&= \de^{(r)}_{n}P_{n}(0;r), \label{Pn+10rPn0r} \\
		&Q_{n+1}(0;r) &&= \be^{(r)}_{n}Q_{n}(0;r), \label{Qn+10rQn0r}
		\end{alignat}	
	\end{minipage}	
	\begin{minipage}{.5\linewidth}
		\begin{alignat}{2}
		&R_{n+1}(0;s) &&= \varkappa^{(s)}_{n}R_{n}(0;s), \label{Rn+10sRn0s} \\
		&S_{n+1}(0;s) &&= \ga^{(r)}_{n}S_{n}(0;s). \label{Sn+10sSn0s}
		\end{alignat}	
	\end{minipage}
\end{corollary}

\begin{corollary}
The following relationships hold between the  $2j-k$ and $j-2k$ polynomial tails
	
\noindent\begin{minipage}{.5\linewidth}
		\begin{equation}
		Q_n(0;r) = \frac{(-1)^n}{P_n(0;r+1)P_n(0;r+2)},\label{Qn0rPn0}
		\end{equation}	
	\end{minipage}	
	\begin{minipage}{.5\linewidth}
		\begin{equation}
		R_n(0;s) = \frac{(-1)^n}{S_n(0;s-1)S_n(0;s-2)}.\label{Rn0rSn0}
		\end{equation}	
	\end{minipage}
\end{corollary}
\begin{proof}
	From \eqref{abcd} and \eqref{EFGH} we have
	\begin{equation}
	\prod_{\ell=0}^{n-1} \be^{(r)}_{\ell} = \frac{(-1)^n}{\prod_{\ell=0}^{n-1}\delta^{(r+1)}_{\ell}\prod_{\ell=0}^{n-1}\delta^{(r+2)}_{\ell}}, \qandq 		\prod_{\ell=0}^{n-1} \varkappa^{(s)}_{\ell} = \frac{(-1)^n}{\prod_{\ell=0}^{n-1}\ga^{(s-1)}_{\ell}\prod_{\ell=0}^{n-1}\ga^{(s-2)}_{\ell}}.
	\end{equation}
	Now \eqref{Qn0rPn0} and \eqref{Rn0rSn0} follow from \eqref{Pn0r},  \eqref{Qn0r}, \eqref{Rn0s}, and \eqref{Sn0s}.
\end{proof}
The following result relates the tail coefficients of the polynomials to certain ratios of determinants. This is the extension of the Toeplitz relations for the reflection or Verblunsky coefficients given by Eqs. (2.19) of \cite{FW_2006}.
\begin{lemma}\label{Lem tails det ratios}
The polynomial tails can be expressed in terms of the ratios of $2j-k$ and $j-2k$ determinants as follows
	
\noindent\begin{minipage}{.5\linewidth}
		\begin{alignat}{2}
		&P_n(0;r) &&= (-1)^n \frac{D^{(r-1)}_n}{D^{(r)}_n}, \label{Pn0r1} \\
		& Q_n(0;r)  &&= (-1)^n \frac{D^{(r+2)}_n}{D^{(r)}_n}, \label{Qn0r1}
		\end{alignat}	
	\end{minipage}	
	\begin{minipage}{.5\linewidth}
		\begin{alignat}{2}
		&R_n(0;s) &&= (-1)^n \frac{E^{(s-2)}_n}{E^{(s)}_n}, \label{Rn0s1} \\
		& S_n(0;s)  &&= (-1)^n \frac{E^{(s+1)}_n}{E^{(s)}_n}. \label{Sn0s1}
		\end{alignat}	
	\end{minipage}
\end{lemma}
\begin{proof}
	These are immediate consequences of the identities \eqref{OP11}, \eqref{OP22}, \eqref{OP11 R}, and \eqref{OP22 S}.
\end{proof}
\begin{lemma}
The polynomial tails satisfy the following pure offset recurrence relations
	
\noindent\begin{minipage}{.5\linewidth}
		\begin{alignat}{2}
		&P_n(0;r+2) &&= - \frac{\de^{(r)}_n}{\eta^{(r)}_n}P_n(0;r), \label{Pn0r+2} \\
		& Q_n(0;r-1)  &&= - \frac{\be^{(r)}_n}{\al^{(r)}_n}Q_n(0;r), \label{Qn0r-1}
		\end{alignat}	
	\end{minipage}	
	\begin{minipage}{.5\linewidth}
		\begin{alignat}{2}
		&R_n(0;s+1) &&= - \frac{\varkappa^{(s)}_n}{\rho^{(s)}_n}R_n(0;s), \label{Rn0s+1} \\
		& S_n(0;s-2)  &&= - \frac{\ga^{(s)}_n}{\theta^{(s)}_n}S_n(0;s). \label{Sn0s-2}
		\end{alignat}	
	\end{minipage}
\end{lemma}
\begin{proof}
	These are immediate consequences of the identities \eqref{P pure r rec}, \eqref{Q pure r rec}, \eqref{R pure s rec}, and \eqref{S* pure s rec} by matching the constant terms and coefficients of $z^{n+1}$.
\end{proof}
Lemma \ref{polytails-reccoeff} establishes the construction of the tails of the orthogonal polynomials, provided the data of $2j-k$ and $j-2k$ determinants (and thus the constants $\de_n^{(r)}$,$\be_n^{(r)}$,$\varkappa_n^{(r)}$, and $\ga_n^{(r)}$) are given. The following theorem establishes this in the reverse direction: how to construct the $2j-k$ and $j-2k$ determinants from the data given as the tails of the corresponding orthogonal polynomials. This result is the extension of the Toeplitz result given by Eq.(2.20) of \cite{FW_2006}.
\begin{theorem}\label{construct_DE}
The $2j-k$ and $j-2k$ determinants can be constructed from the tails of the orthogonal polynomials as
	\begin{equation}\label{DnfromtailsofP}
	D^{(r)}_n = w^n_r \prod_{\ell=0}^{n-1} \frac{1}{P_{\ell}(0;r+1)} \prod_{\nu=1}^{\ell} \left(\frac{P_{\nu}(0;r)}{P_{\nu}(0;r+2)}-1\right),
	\end{equation} 
and
	\begin{equation}\label{EnfromtailsofS}
	E^{(s)}_n = w^n_s \prod_{\ell=0}^{n-1} \frac{1}{S_{\ell}(0;s-1)} \prod_{\nu=1}^{\ell} \left(\frac{S_{\nu}(0;s)}{S_{\nu}(0;s-2)}-1\right).
	\end{equation}	 
\end{theorem}
\begin{proof}
Let $1\leq \nu \leq n-1$ be a fixed integer and let $\boldsymbol{\mathscr{D}}_r$ be the $(\nu+3)\times(\nu+3)$ matrix given by \eqref{DDD}. Now consider the following Dodgson condensation identity:
	\begin{equation}\label{DodgsonPQtails}
	\begin{split}
	& \mathscr{D}_r\left\lbrace \begin{matrix}  \nu+2  \\  \nu+2 \end{matrix} \right\rbrace \cdot \mathscr{D}_r\left\lbrace \begin{matrix}  0 & \nu+1 & \nu+2 \\  0 & \nu+1 & \nu+2 \end{matrix} \right\rbrace =  \mathscr{D}_r \left\lbrace \begin{matrix} 0 &  \nu+2   \\ 0 &  \nu+2  \end{matrix} \right\rbrace \cdot \mathscr{D}_r\left\lbrace \begin{matrix}   \nu+1  & \nu+2  \\  \nu+1 &  \nu+2  \end{matrix} \right\rbrace - \mathscr{D}_r\left\lbrace \begin{matrix} 0 &  \nu+2   \\  \nu+1 &  \nu+2  \end{matrix} \right\rbrace \cdot \mathscr{D}_r\left\lbrace \begin{matrix}   \nu+1  & \nu+2  \\ 0 &  \nu+2 \end{matrix} \right\rbrace.
	\end{split}
	\end{equation}
In view of \eqref{DDD}, \eqref{h}, \eqref{Pn0r1} and \eqref{Qn0r1}, the above identity can be written as 
	\begin{equation}
	h^{(r)}_{\nu+1}\frac{D^{(r+1)}_{\nu}}{D^{(r)}_{\nu+1}}=\frac{D^{(r+1)}_{\nu+1}}{D^{(r)}_{\nu+1}}-Q_{\nu+1}(0;r)P_{\nu+1}(0;r). 		
	\end{equation}
Using \eqref{h} again and replacing $\nu \mapsto \nu-1$ yields
	\begin{equation}
		\left(\frac{h^{(r)}_{\nu}}{h^{(r+1)}_{\nu-1}}-1\right)\frac{D^{(r+1)}_{\nu}}{D^{(r)}_{\nu}}=-Q_{\nu}(0;r)P_{\nu}(0;r).
	\end{equation}
Now we use \eqref{Pn0r1} to rewrite this as 
		\begin{equation}
	\frac{h^{(r)}_{\nu}}{h^{(r+1)}_{\nu-1}}=1-(-1)^{\nu}Q_{\nu}(0;r)P_{\nu}(0;r)P_{\nu}(0;r+1) = 1- \frac{P_{\nu}(0;r)}{P_{\nu}(0;r+2)},
	\end{equation}
where we have used \eqref{Qn0rPn0}. Therefore
	\begin{equation}
		 \frac{\prod_{\nu=1}^{\ell}h^{(r)}_{\nu}}{\prod_{\nu=1}^{\ell}h^{(r+1)}_{\nu-1}} = \prod_{\nu=1}^{\ell} \left(1- \frac{P_{\nu}(0;r)}{P_{\nu}(0;r+2)}\right).
	\end{equation}
From \eqref{Dets from norms} we will get
	\begin{equation}
	\frac{h^{(r)}_\ell}{h^{(r)}_0}\frac{D^{(r)}_{\ell}}{D^{(r+1)}_{\ell}} = \prod_{\nu=1}^{\ell} \left(1- \frac{P_{\nu}(0;r)}{P_{\nu}(0;r+2)}\right).
	\end{equation}
Using \eqref{Pn0r1} and noticing that $h^{(r)}_0=D^{(r)}_1=w_r$, we have 
	\begin{equation}
	h^{(r)}_\ell = w_r \frac{(-1)^\ell}{P_{\ell}(0;r+1)} \prod_{\nu=1}^{\ell} \left(1- \frac{P_{\nu}(0;r)}{P_{\nu}(0;r+2)}\right).
	\end{equation}
Finally, using \eqref{Dets from norms} we obtain \eqref{DnfromtailsofP}. The identity \eqref{EnfromtailsofS} can be proven identically if one starts with the following  Dodgson condensation identity:
	\begin{equation}\label{DodgsonRStails}
	\begin{split}
	& \mathscr{E}_s\left\lbrace \begin{matrix}  \nu+2  \\  \nu+2 \end{matrix} \right\rbrace \cdot \mathscr{E}_s\left\lbrace \begin{matrix}  0 & \nu+1 & \nu+2 \\  0 & \nu+1 & \nu+2 \end{matrix} \right\rbrace =  \mathscr{E}_s \left\lbrace \begin{matrix} 0 &  \nu+2   \\ 0 &  \nu+2  \end{matrix} \right\rbrace \cdot \mathscr{E}_s\left\lbrace \begin{matrix}   \nu+1  & \nu+2  \\  \nu+1 &  \nu+2  \end{matrix} \right\rbrace - \mathscr{E}_s\left\lbrace \begin{matrix} 0 &  \nu+2   \\  \nu+1 &  \nu+2  \end{matrix} \right\rbrace \cdot \mathscr{E}_s\left\lbrace \begin{matrix}   \nu+1  & \nu+2  \\ 0 &  \nu+2 \end{matrix} \right\rbrace.
	\end{split}
	\end{equation}
\end{proof}

\begin{remark} \normalfont
So far we have treated the offsets as arbitrary integers, 
however from the pure $r,s$ recurrences of Theorem \ref{offset-RR} it is clear that only those offsets that are distinct modulo 3 are independent.
Thus we can use any three consecutive values of $ r,s$. 
This is the analogue to the Toeplitz case where the reflection or Verblunsky coefficients corresponded to the two offsets of $ \pm 1 $.
This fact is reinforced by the results of Theorem \ref{construct_DE} which show that the moment determinants $ D_{n}^{(0)}, E_{n}^{(0)} $ can be constructed from polynomial tails
with $ r=0, 1, 2 $ and $ s=0, -1, -2 $ respectively.   
\end{remark}

\section{Multiple integral representations}\label{Sec MultInt}
The joint eigenvalue PDF of \eqref{1m2_JPDF} is fundamental from our viewpoint and it is from this that we construct various marginal distributions,
otherwise known as $n$-particle correlation functions, through integration over a set of "internal" variables, 
the complement of $n$ "external" variables. 
All aspects of the theory will have representations of this form: the determinants, the bi-orthogonal polynomials and the reproducing kernels. 
Thus our first result of this nature is the $0$-point correlation or normalisation integral.
This is the extension of the Toeplitz result, Eq.(2.2), of \cite{FW_2006}.
\begin{theorem}
	The moment determinants, and therefore the normalisation of the bi-orthogonal polynomials, possess the multiple integral representations
	\begin{equation}\label{MultIntD}
	D_{n}^{(r)} = \frac{1}{n!} \int_{\T} \frac{\dd \ze_1}{2 \pi \ic \ze_1}\int_{\T} \frac{\dd \ze_2}{2 \pi \ic \ze_2} \cdots \int_{\T} \frac{\dd \ze_n}{2 \pi \ic \ze_n} \prod_{j=1}^{n}w(\ze_j) \ze^{-r}_j \prod_{1\leq j<k\leq n} (\ze_k-\ze_j)(\ze^{-2}_k-\ze^{-2}_j),
	\end{equation}
	and
	\begin{equation}\label{MultIntE}
	E_{n}^{(s)} = \frac{1}{n!} \int_{\T} \frac{\dd \ze_1}{2 \pi \ic \ze_1}\int_{\T} \frac{\dd \ze_2}{2 \pi \ic \ze_2} \cdots \int_{\T} \frac{\dd \ze_n}{2 \pi \ic \ze_n} \prod_{j=1}^{n}w(\ze_j) \ze^{-s}_j \prod_{1\leq j<k\leq n} (\ze^2_k-\ze^2_j)(\ze^{-1}_k-\ze^{-1}_j).
	\end{equation}
\end{theorem}
\begin{proof}
We have
	\begin{equation}
	\prod_{1\leq j<k\leq n} (\ze_k-\ze_j) = \underset{1\leq j,k\leq n}{\det} \{\ze^{k-1}_j\}= \underset{1\leq j,k\leq n}{\det} \{P_{k-1}(\ze_j)\},
	\end{equation}
and
	\begin{equation}
	\prod_{1\leq j<k\leq n} (\ze^{-2}_k-\ze^{-2}_j)= \underset{1\leq j,k\leq n}{\det} \{\ze^{-2(k-1)}_j\}= \underset{1\leq j,k\leq n}{\det} \{Q_{k-1}(\ze^{-2}_j)\}.
	\end{equation}
So
	\begin{equation}
	\begin{split}
	& \frac{1}{n!} \int_{\T} \frac{\dd \ze_1}{2 \pi \ic \ze_1}\int_{\T} \frac{\dd \ze_2}{2 \pi \ic \ze_2} \cdots \int_{\T} \frac{\dd \ze_n}{2 \pi \ic \ze_n} \prod_{j=1}^{n}w(\ze_j) \ze^{-r}_j \prod_{1\leq j<k\leq n} (\ze_k-\ze_j)(\ze^{-2}_k-\ze^{-2}_j) \\
	& = \frac{1}{n!} \int_{\T} \frac{\dd \ze_1}{2 \pi \ic \ze_1}\int_{\T} \frac{\dd \ze_2}{2 \pi \ic \ze_2} \cdots \int_{\T} \frac{\dd \ze_n}{2 \pi \ic \ze_n} \prod_{j=1}^{n}w(\ze_j) \ze^{-r}_j \underset{1\leq j,k\leq n}{\det} \{P_{k-1}(\ze_j)\} \underset{1\leq j,k\leq n}{\det} \{Q_{k-1}(\ze^{-2}_j)\}  \\ & =  \int_{\T} \frac{\dd \ze_1}{2 \pi \ic \ze_1}\int_{\T} \frac{\dd \ze_2}{2 \pi \ic \ze_2} \cdots \int_{\T} \frac{\dd \ze_n}{2 \pi \ic \ze_n} \prod_{j=1}^{n}w(\ze_j) \ze^{-r}_j  \underset{1\leq j,k\leq n}{\det} \{P_{k-1}(\ze_j)\}  \prod_{j=1}^n Q_{j-1}(\ze^{-2}_j)  \\ & =  \int_{\T} \frac{\dd \ze_1}{2 \pi \ic \ze_1}\int_{\T} \frac{\dd \ze_2}{2 \pi \ic \ze_2} \cdots \int_{\T} \frac{\dd \ze_n}{2 \pi \ic \ze_n} \prod_{j=1}^{n}w(\ze_j) \ze^{-r}_j  \underset{1\leq j,k\leq n}{\det} \{P_{k-1}(\ze_j) Q_{j-1}(\ze^{-2}_j)\}  \\ & = \underset{1\leq j,k\leq n}{\det} \left\{ \int_{\T} \frac{\dd \ze}{2 \pi \ic \ze} w(\ze)\ze^{-r} P_{k-1}(\ze) Q_{j-1}(\ze^{-2})\right\} = \underset{0\leq j,k\leq n-1}{\det} \left\{ h^{(r)}_j \de_{j,k} \right\} = \prod_{j=0}^{n-1} h^{(r)}_j = D^{(r)}_n,
	\end{split}
	\end{equation} 
where we have used \eqref{PQorth} and \eqref{h}. The formula \eqref{MultIntE} can be established similarly using \eqref{RSorth} and \eqref{H}.
\end{proof} 
So, in particular we obtain \eqref{Dd} and \eqref{Ee} in view of \eqref{DD}, \eqref{EE}, \eqref{MultIntD} and \eqref{MultIntE}. The relationship \eqref{E&D rel} can be also found using the multiple integral formulae above. To that end notice that
\begin{equation}
(\ze_k-\ze_j)(\ze^{-2}_k-\ze^{-2}_j) = \frac{1}{\ze_k\ze_j}(\ze^{-1}_k-\ze^{-1}_j)(\ze^{2}_k-\ze^{2}_j).
\end{equation}
Therefore
\begin{equation}
\prod_{1\leq j<k\leq n} (\ze_k-\ze_j)(\ze^{-2}_k-\ze^{-2}_j)  = \psi \prod_{1\leq j<k\leq n} (\ze^{-1}_k-\ze^{-1}_j)(\ze^{2}_k-\ze^{2}_j),
\end{equation}
where 
\begin{equation}
\psi = \prod_{1\leq j<k\leq n} \frac{1}{\ze_k\ze_j} = \prod_{j=1}^{n} \ze^{1-n}_j,
\end{equation}
because 
\begin{equation}
\psi^2 = \prod_{1 \leq j \neq k \leq n} \frac{1}{\ze_k\ze_j} = \frac{\di  \prod_{1 \leq j , k \leq n} \frac{1}{\ze_k\ze_j}}{\di  \prod_{1 \leq j = k \leq n} \frac{1}{\ze_k\ze_j}} = \prod_{j=1}^{n} \ze^2_j \prod_{1 \leq j , k \leq n} \ze^{-1}_k\ze^{-1}_j = \prod_{j=1}^{n} \ze^{2-2n}_j.
\end{equation}
Note that 
\begin{equation}
\begin{split}
\mathcal{D}_n[w(\ze)\ze^{-r}]  & =\frac{1}{n!} \int_{\T} \frac{\dd \ze_1}{2 \pi \ic \ze_1}\int_{\T} \frac{\dd \ze_2}{2 \pi \ic \ze_2} \cdots \int_{\T} \frac{\dd \ze_n}{2 \pi \ic \ze_n} \prod_{j=1}^{n}w(\ze_j) \ze^{-r}_j \prod_{1\leq j<k\leq n} (\ze_k-\ze_j)(\ze^{-2}_k-\ze^{-2}_j) \\ & = \frac{1}{n!} \int_{\T} \frac{\dd \ze_1}{2 \pi \ic \ze_1}\int_{\T} \frac{\dd \ze_2}{2 \pi \ic \ze_2} \cdots \int_{\T} \frac{\dd \ze_n}{2 \pi \ic \ze_n} \prod_{j=1}^{n}w(\ze_j) \ze^{1-n-r}_j \prod_{1\leq j<k\leq n} (\ze^{-1}_k-\ze^{-1}_j)(\ze^{2}_k-\ze^{2}_j) \\ & = \mathcal{E}_n[w(\ze)\ze^{-(n+r-1)}],
\end{split}
\end{equation}
which is equivalent to \eqref{E&D rel}. More generally we have the following result.

\begin{theorem}
For any integrable function $f$ we have:
	\begin{equation}
	\mathcal{D}_n[w(\ze)\ze^{-r} f(\ze)] = 	\mathcal{E}_n[w(\ze)\ze^{-n-r+1} f(\ze)],
	\end{equation}
or equivalently 
	\begin{equation}
	\mathcal{E}_n[w(\ze)\ze^{-s} f(\ze)]= \mathcal{D}_n[w(\ze)\ze^{-s-1+n} f(\ze)].
	\end{equation}
\end{theorem}

\subsection{Multiple integral formulae for the bi-orthogonal polynomials}
It is well known that the bi-orthogonal polynomials in the Toeplitz system can be written as averages of characteristic polynomials over the unitary group,
see Eq.(2.16,18) of \cite{FW_2006}. In the following Theorem we give the analogous results for the $2j-k$ and $j-2k$ polynomials.
Another interpretation of such integrals is as rational modifications of the weight, known as Christoffel-Darboux-Uvarov transformations of the underlying 
bi-orthogonal system in the Toeplitz case, which was exhaustively treated in \cite{W_2009}.
\begin{theorem}\label{Thm Mult Int Polys}
We have the following multiple integral representations for $2j-k$ and $j-2k$  polynomials:

	\noindent
	\begin{minipage}{.5\linewidth}
		\begin{alignat}{2}
		&P_n(z;r) &&= \frac{1}{D^{(r)}_n} \mathcal{D}_n[w(\ze)\ze^{-r}(z-\ze)], \label{MultIntPolysP} \\
		& Q_n(z;r)  &&= \frac{1}{D^{(r)}_n} \mathcal{D}_n[w(\ze)\ze^{-r}(z-\ze^{-2})], \label{MultIntPolysQ}
		\end{alignat}	
	\end{minipage}	
	\begin{minipage}{.5\linewidth}
		\begin{alignat}{2}
		&R_n(z;s) &&= \frac{1}{E^{(s)}_n} \mathcal{E}_n[w(\ze)\ze^{-s}(z-\ze^2)], \label{MultIntPolysR} \\
		& S_n(z;s)  &&= \frac{1}{E^{(s)}_n} \mathcal{E}_n[w(\ze)\ze^{-s}(z-\ze^{-1})]. \label{MultIntPolysS}
		\end{alignat}	
	\end{minipage}
\end{theorem}
\begin{proof}
We only prove the formula for $P_n$ as the other ones can be found similarly.
	\begin{equation}
	\begin{split}
	&	\mathcal{D}_n[w(\ze)\ze^{-r}(z-\ze)]  \\ & = \frac{1}{n!} \int_{\T} \frac{\dd \ze_1}{2 \pi \ic \ze_1} \cdots \int_{\T} \frac{\dd \ze_n}{2 \pi \ic \ze_n} \prod_{j=1}^{n}w(\ze_j) \ze^{-r}_j \prod_{j=1}^{n}(z-\ze_j)\prod_{1\leq j<k\leq n} (\ze_k-\ze_j)(\ze^{-2}_k-\ze^{-2}_j)  \\ & = 
	\frac{(-1)^n}{n!} \int_{\T} \frac{\dd \ze_1}{2 \pi \ic \ze_1} \cdots \int_{\T} \frac{\dd \ze_n}{2 \pi \ic \ze_n} \prod_{j=1}^{n}w(\ze_j) \ze^{-r}_j \prod_{0\leq j<k\leq n} (\ze_k-\ze_j) \prod_{1\leq j<k\leq n} (\ze^{-2}_k-\ze^{-2}_j),
	\end{split}
	\end{equation}
where we have denoted $z$ by $\ze_0$. Therefore
	\begin{equation}
	\begin{split}
	&	\mathcal{D}_n[w(\ze)\ze^{-r}(z-\ze)] \\ & = \frac{(-1)^n}{n!} \int_{\T} \frac{\dd \ze_1}{2 \pi \ic \ze_1} \cdots \int_{\T} \frac{\dd \ze_n}{2 \pi \ic \ze_n} \prod_{j=1}^{n}w(\ze_j) \ze^{-r}_j \underset{\substack{1 \leq j \leq n \\ 1 \leq k\leq n+1}}{\det} 
	\begin{pmatrix}
	P_{k-1}(\ze_0) \\ P_{k-1}(\ze_j)
	\end{pmatrix} \underset{\substack{1 \leq j \leq n \\ 1 \leq k\leq n}}{\det} \{Q_{k-1}(\ze^{-2}_j)\}  \\ & =
	(-1)^n \int_{\T} \frac{\dd \ze_1}{2 \pi \ic \ze_1} \cdots \int_{\T} \frac{\dd \ze_n}{2 \pi \ic \ze_n} \prod_{j=1}^{n}w(\ze_j) \ze^{-r}_j \underset{\substack{1 \leq j \leq n \\ 1 \leq k\leq n+1}}{\det} 
	\begin{pmatrix}
	P_{k-1}(\ze_0) \\ P_{k-1}(\ze_j)
	\end{pmatrix} \prod_{j=1}^{n} Q_{j-1}(\ze^{-2}_j)  \\
	& = (-1)^n \underset{\substack{1 \leq j \leq n \\ 1 \leq k\leq n+1}}{\det} 
	\begin{pmatrix}
	P_{k-1}(\ze_0) \\ \int_{\T} \frac{\dd \ze}{2 \pi \ic \ze} w(\ze) \ze^{-r} P_{k-1}(\ze) Q_{j-1}(\ze^{-2}) 
	\end{pmatrix} 
	 \\ & = (-1)^n 	\det 
	 \begin{pmatrix}
	P_0(z)  & \cdots & P_{n-1}(z)& P_n(z) \\
	h^{(r)}_0  & \cdots & 0 & 0 \\
	\vdots& \ddots & \vdots & \vdots  \\
	0  & \cdots & h^{(r)}_{n-1} & 0
	\end{pmatrix} = P_n(z;r) \prod_{j=0}^{n-1} h^{(r)}_j = D^{(r)}_n P_n(z;r).
	\end{split}
	\end{equation}
\end{proof}
\subsection{Multiple integral formulae for the reproducing kernels}
We will require the following important lemma, which is known as the "integration over successive variables", in a subsequent result.
This is the extension of the result in the Toeplitz case, which is easy to derive following the line of Proposition 5.1.2 of \cite{For_2010}.
\begin{lemma}[Gaudin's Theorem]\label{intofdetof-rep-ker}
	The integration over the final variable $\zeta_m$ of the $m\times m$ determinant constructed from reproducing kernel entries has the evaluation
	\begin{equation}\label{12}
	\int_{\T} \frac{\dd \ze_m}{2 \pi \ic \ze_m} w(\ze_m) \ze^{-r}_m \underset{\substack{1 \leq j \leq m \\ 1 \leq k\leq m}}{\det}\left[K_n(\ze^{-2}_k,\ze_j;r)\right] = (n-m+2) \underset{\substack{1 \leq j \leq m-1 \\ 1 \leq k\leq m-1}}{\det}\left[K_n(\ze^{-2}_k,\ze_j;r)\right],
	\end{equation}
	and 
	\begin{equation}\label{14}
		\int_{\T} \frac{\dd \ze_m}{2 \pi \ic \ze_m} w(\ze_m) \ze^{-s}_m \underset{\substack{1 \leq j \leq m \\ 1 \leq k\leq m}}{\det}\left[L_n(\ze^{-1}_k,\ze^2_j;s)\right] = (n-m+2) \underset{\substack{1 \leq j \leq m-1 \\ 1 \leq k\leq m-1}}{\det}\left[L_n(\ze^{-1}_k,\ze^2_j;s)\right].
	\end{equation}
\end{lemma} 
\begin{proof}
Let us use $j$ and $k$ respectively as indices of rows and columns. Expanding the determinant in the integrand along the last row gives
\begin{equation}
	\underset{\substack{1 \leq j \leq m \\ 1 \leq k\leq m}}{\det}\left[K_n(\ze^{-2}_k,\ze_j;r)\right] = \sum^{m-1}_{k=1} (-1)^{m+k} K_n(\ze^{-2}_k,\ze_m;r) \underset{\substack{1 \leq j\leq m-1 \\ 1 \leq \ell \leq m \\ \ell \neq k}}{\det} \left[ K_n(\ze^{-2}_{\ell},\ze_j;r) \right] + K_n(\ze^{-2}_m,\ze_m;r) \underset{\substack{1 \leq j \leq m-1 \\ 1 \leq k\leq m-1}}{\det} \left[K_n(\ze^{-2}_k,\ze_j;r)\right]
\end{equation}
\begin{equation}
	= \sum^{m-1}_{k=1} (-1)^{m+k}  \underset{\substack{1 \leq j\leq m-1 \\ 1 \leq \ell \leq m-1 \\ \ell \neq k}}{\det} \left[ K_n(\ze^{-2}_{\ell},\ze_j;r) ,K_n(\ze^{-2}_k,\ze_m;r) K_n(\ze^{-2}_{m},\ze_j;r)  \right] + K_n(\ze^{-2}_m,\ze_m;r) \underset{\substack{1 \leq j \leq m-1 \\ 1 \leq k\leq m-1}}{\det} \left[K_n(\ze^{-2}_k,\ze_j;r)\right].
\end{equation}
Recalling the normalisation \eqref{kernel_norm} and the projection relations \eqref{kernel_projection} for the reproducing kernel we deduce	
\begin{equation}
	\begin{split}
	\int_{\T} \frac{\dd \ze_m}{2 \pi \ic \ze_m} w(\ze_m) \ze^{-r}_m \underset{\substack{1 \leq j \leq m \\ 1 \leq k\leq m}}{\det}\left[K_n(\ze^{-2}_k,\ze_j;r)\right] & = \sum^{m-1}_{k=1} (-1)^{m+k}  \underset{\substack{1 \leq j\leq m-1 \\ 1 \leq \ell \leq m-1 \\ \ell \neq k}}{\det} \left[ K_n(\ze^{-2}_{\ell},\ze_j;r) ,K_n(\ze^{-2}_k,\ze_j;r)  \right] \\ & + (n+1) \underset{\substack{1 \leq j \leq m-1 \\ 1 \leq k\leq m-1}}{\det} \left[K_n(\ze^{-2}_k,\ze_j;r)\right],
	\end{split}
\end{equation}
By $m-k-1$ successive swaps of adjacent columns we have 
	\[ \underset{\substack{1 \leq j\leq m-1 \\ 1 \leq \ell \leq m-1 \\ \ell \neq k}}{\det} \left[ K_n(\ze^{-2}_{\ell},\ze_j;r) ,K_n(\ze^{-2}_k,\ze_j;r)  \right] = (-1)^{m-k-1} \underset{\substack{1 \leq j \leq m-1 \\ 1 \leq \ell \leq m-1}}{\det} \left[K_n(\ze^{-2}_\ell,\ze_j;r)\right].  \]
Equation \eqref{12} now follows by combining the last two equations. The proof of \eqref{14} is similar and we do not provide the details here.
\end{proof}

\begin{theorem}
The reproducing kernels for $2j-k$ and $j-2k$ systems  have the following multiple integral representations:
	\begin{equation}\label{MultIntRepKers}
	\begin{split}
	K_n(z_2,z_1;r) & = \frac{1}{D^{(r)}_{n+1}} \mathcal{D}_n[w(\ze)\ze^{-r}(z_1-\ze)(z_2-\ze^{-2})],
	\end{split}
	\end{equation}
and
	\begin{equation}\label{MultIntRepKers2}
		L_n(z_2,z_1;s) = \frac{1}{E^{(s)}_{n+1}} \mathcal{E}_n[w(\ze)\ze^{-s}(z_1-\ze^{2})(z_2-\ze^{-1})]
						= \left( z_1z_2 \right)^{n} K_n(z_{1}^{-1},z_{2}^{-1};s-n).
	\end{equation}
\end{theorem}
\begin{proof} 
For simplicity of notation, in this proof we denote $P_n(z;r),Q_n(z;r)$ and $K_n(z_2,z_1;r)$ respectively by $P_n(z),Q_n(z)$ and $K_n(z_2,z_1)$. 
We have
	\begin{equation*}
	\begin{split}
	& \mathcal{D}_n[w(\ze)\ze^{-r}(\ze-z_1)(\ze^{-2}-z_2)]  \\ & = \frac{1}{n!} \int_{\T} \frac{\dd \ze_1}{2 \pi \ic \ze_1} \cdots \int_{\T} \frac{\dd \ze_n}{2 \pi \ic \ze_n} \prod_{j=1}^{n}w(\ze_j) \ze^{-r}_j \prod_{j=1}^{n}(\ze_j-z_1)(\ze_j^{-2}-z_2)\prod_{1\leq j<k\leq n} (\ze_k-\ze_j)(\ze^{-2}_k-\ze^{-2}_j)  \\ & = \frac{1}{n!} \int_{\T} \frac{\dd \ze_1}{2 \pi \ic \ze_1} \cdots \int_{\T} \frac{\dd \ze_n}{2 \pi \ic \ze_n} \prod_{j=1}^{n}w(\ze_j) \ze^{-r}_j \prod_{0\leq j<k\leq n} (\ze_k-\ze_j)(\ze^{-2}_k-\ze^{-2}_j),
	\end{split}
	\end{equation*}
where we have temporarily denoted $z_1$ and $z_2$ respectively by $\ze_0$ and $\ze^{-2}_0$. Therefore
	\begin{equation*}
	\begin{split}
	& \mathcal{D}_n[w(\ze)\ze^{-r}(\ze-z_1)(\ze^{-2}-z_2)]  \\ & = \frac{1}{n!} \int_{\T} \frac{\dd \ze_1}{2 \pi \ic \ze_1} \cdots \int_{\T} \frac{\dd \ze_n}{2 \pi \ic \ze_n} \prod_{j=1}^{n}w(\ze_j) \ze^{-r}_j \underset{\substack{1 \leq j \leq n \\ 1 \leq k\leq n+1}}{\det} \begin{pmatrix}
	P_{k-1}(\ze_0) \\ P_{k-1}(\ze_j)
	\end{pmatrix} \underset{\substack{1 \leq j \leq n \\ 1 \leq k\leq n+1}}{\det} \begin{pmatrix}
	Q_{k-1}(\ze^{-2}_0) \\ Q_{k-1}(\ze^{-2}_j)
	\end{pmatrix}  \\ & = \frac{1}{n!} \int_{\T} \frac{\dd \ze_1}{2 \pi \ic \ze_1} \cdots \int_{\T} \frac{\dd \ze_n}{2 \pi \ic \ze_n} \prod_{j=1}^{n}w(\ze_j) \ze^{-r}_j \underset{\substack{1 \leq j \leq n \\ 1 \leq k\leq n+1}}{\det} \begin{pmatrix}
	P_{k-1}(\ze_0) \\ P_{k-1}(\ze_j)
	\end{pmatrix} \underset{\substack{1 \leq j \leq n \\ 1 \leq k\leq n+1}}{\det} \begin{pmatrix}
	Q_{k-1}(\ze^{-2}_0) \\ Q_{k-1}(\ze^{-2}_j)
	\end{pmatrix}^T  \\ & =
	\frac{D^{(r)}_{n+1}}{n!} \int_{\T} \frac{\dd \ze_1}{2 \pi \ic \ze_1} \cdots \int_{\T} \frac{\dd \ze_n}{2 \pi \ic \ze_n} \prod_{j=1}^{n}w(\ze_j) \ze^{-r}_j 
	\det \begin{pmatrix}
	P_0(\ze_0) & \cdots & P_n(\ze_0) \\
	\vdots & \vdots & \vdots \\
	P_0(\ze_n) & \cdots & P_n(\ze_n) \\
	\end{pmatrix}	 	 \det \begin{pmatrix}
	\frac{Q_0(\ze^{-2}_0)}{h^{(r)}_0} & \cdots & \frac{Q_0(\ze^{-2}_n)}{h^{(r)}_0}  \\
	\vdots & \vdots & \vdots \\
	\frac{Q_n(\ze^{-2}_0)}{h^{(r)}_n}  & \cdots & \frac{Q_n(\ze^{-2}_n)}{h^{(r)}_n} \\
	\end{pmatrix} \\ & =
	\frac{D^{(r)}_{n+1}}{n!} \int_{\T} \frac{\dd \ze_1}{2 \pi \ic \ze_1} \cdots \int_{\T} \frac{\dd \ze_n}{2 \pi \ic \ze_n} \prod_{j=1}^{n}w(\ze_j) \ze^{-r}_j \underset{\substack{1 \leq j \leq n \\ 1 \leq k\leq n}}{\det} \begin{pmatrix}
	\di \sum_{\nu=0}^{n} \frac{P_{\nu}(\ze_0) Q_{\nu}(\ze^{-2}_0)}{h^{(r)}_{\nu}} &  \di 	\sum_{\nu=0}^{n} \frac{P_{\nu}(\ze_0) Q_{\nu}(\ze^{-2}_k)}{h^{(r)}_{\nu}} \\ \di  \sum_{\nu=0}^{n} \frac{P_{\nu}(\ze_j) Q_{\nu}(\ze^{-2}_0)}{h^{(r)}_{\nu}} & \di  \sum_{\nu=0}^{n} \frac{P_{\nu}(\ze_j) Q_{\nu}(\ze^{-2}_k)}{h^{(r)}_{\nu}}
	\end{pmatrix}   \\ & = 	\frac{D^{(r)}_{n+1}}{n!} \int_{\T} \frac{\dd \ze_1}{2 \pi \ic \ze_1} \cdots \int_{\T} \frac{\dd \ze_n}{2 \pi \ic \ze_n} \prod_{j=1}^{n}w(\ze_j) \ze^{-r}_j \underset{\substack{1 \leq j \leq n \\ 1 \leq k\leq n}}{\det} \begin{pmatrix}
	K_n(\ze^{-2}_0,\ze_0) &  K_n(\ze^{-2}_k,\ze_0) \\ K_n(\ze^{-2}_0,\ze_j) & K_n(\ze^{-2}_k,\ze_j)
	\end{pmatrix}.
	\end{split}
	\end{equation*}
Now we apply Theorem \ref{intofdetof-rep-ker} :
	\begin{equation}
	\begin{split}
	& \mathcal{D}_n[w(\ze)\ze^{-r}(\ze-z_1)(\ze^{-2}-z_2)] = (n+2-(n+1)) \\ &  \times 	\frac{D^{(r)}_{n+1}}{n!} \int_{\T} \frac{\dd \ze_1}{2 \pi \ic \ze_1} \cdots \int_{\T} \frac{\dd \ze_{n-1}}{2 \pi \ic \ze_{n-1}} \prod_{j=1}^{n-1}w(\ze_j) \ze^{-r}_j \underset{\substack{1 \leq j \leq n-1 \\ 1 \leq k\leq n-1}}{\det} \begin{pmatrix}
	K_n(\ze^{-2}_0,\ze_0) &  K_n(\ze^{-2}_k,\ze_0) \\ K_n(\ze^{-2}_0,\ze_j) & K_n(\ze^{-2}_k,\ze_j)
	\end{pmatrix}.
	\end{split}
	\end{equation}
Repeating this $n-1$ times yields
	\begin{equation}
	\mathcal{D}_n[w(\ze)\ze^{-r}(\ze-z_1)(\ze^{-2}-z_2)] = \prod_{j=2}^{n+1} (n+2-j) \frac{D^{(r)}_{n+1}}{n!} 
	K_n(\ze^{-2}_0,\ze_0),
	\end{equation} and thus $	\mathcal{D}_n[w(\ze)\ze^{-r}(\ze-z_1)(\ze^{-2}-z_2)] = D^{(r)}_{n+1}
	K_n(z_2,z_1)$. The equation \eqref{MultIntRepKers2} can be established similarly and we do not provide the details here.
\end{proof}

\begin{remark} \normalfont
One can also see the consistency of multiple integral formulae for the reproducing kernels with the definitions \eqref{RepKer3} and \eqref{RepKer3 j-2k} by considering the leading asymptotic behaviour as $z$ or $\mathcal{z}$ tend to infinity. 
For example as $z \to \infty$ we have
	\[K_{n}(z,\mathcal{z};r) := \sum_{j=0}^{n} \frac{1}{h^{(r)}_{j}}Q_{j}(z;r)P_{j}(\mathcal{z};r) = \frac{z^n}{h^{(r)}_{n}}P_{j}(\mathcal{z};r) \left(1+\mathcal{O}\left(\frac{1}{z}\right)\right), \qquad z \to \infty, \]
and, expectedly, from the right hand side of \eqref{MultIntRepKers} we find the same asymptotic behaviour:
	\[ \frac{1}{D^{(r)}_{n+1}} \mathcal{D}_n[w(\ze)\ze^{-r}(\mathcal{z}-\ze)(z-\ze^{-2})] = \frac{1}{D^{(r)}_{n+1}} \mathcal{D}_n[w(\ze)\ze^{-r}(\mathcal{z}-\ze)z]\left(1+\mathcal{O}\left(\frac{1}{z}\right)\right)   \]
	\[ \hspace{8mm} = \frac{z^n}{h^{(r)}_nD^{(r)}_{n}} \mathcal{D}_n[w(\ze)\ze^{-r}(\mathcal{z}-\ze)]\left(1+\mathcal{O}\left(\frac{1}{z}\right)\right) = \frac{z^n}{h^{(r)}_{n}}P_{j}(\mathcal{z};r) \left(1+\mathcal{O}\left(\frac{1}{z}\right)\right), \qquad z \to \infty,  \]
by \eqref{MultIntPolysP}.
Similar consistency checks can be done for $K_{n}(z,\mathcal{z};r)$ as $\mathcal{z} \to \infty$, 
and also for $L_{n}(z,\mathcal{z};r)$ as $z \to \infty$ and separately as  $\mathcal{z} \to \infty$.
\end{remark}

\begin{theorem}\label{Thm6.7}
We have	
	\begin{alignat}{2}
		&P_{n}(z;r) &&= P_{n}(z;r-2) + \frac{h^{(r-2)}_n}{Q_n(0;r-2)} \sum_{\ell=0}^{n-1} \frac{Q_{\ell}(0;r-2)}{h^{(r-2)}_{\ell}} P_{\ell}(z;r-2), \label{Pr+2 Pls} \\
		&Q_{n}(z;r) &&= Q_{n}(z;r+1)+ \frac{h^{(r+1)}_n}{P_n(0;r+1)} \sum_{\ell=0}^{n-1} \frac{P_{\ell}(0;r+1)}{h^{(r+1)}_{\ell}}Q_{\ell}(z;r+1), \label{Qr-1 Qls} \\
		&R_{n}(z;s) &&= R_{n}(z;s-1)+\frac{g^{(s-1)}_n}{S_n(0;s-1)}\sum_{\ell=0}^{n-1} \frac{S_{\ell}(0;s-1)}{g^{(s-1)}_{\ell}} R_{\ell}(z;s-1), \label{Rs+1 Rls} \\
		&S_{n}(z;s) &&=S_{n}(z;s+2)+\frac{g^{(s+2)}_n}{R_n(0;s+2)} \sum_{\ell=0}^{n-1} \frac{R_{\ell}(0;s+2)}{g^{(s+2)}_{\ell}} S_{\ell}(z;s+2). \label{Ss-2 Sls}
	\end{alignat}
\end{theorem}
\begin{proof}
These identities can be proven if one considers special arguments for the reproducing kernel. Indeed, to prove \eqref{Pr+2 Pls} we consider
	\begin{equation}\label{K0zr}
		K_n(0,z;r) = \sum_{\ell=0}^{n} \frac{1}{h^{(r)}_{\ell}}Q_{\ell}(0;r)P_{\ell}(z;r),
	\end{equation} 
by \eqref{RepKer3}.	On the other hand from \eqref{MultIntRepKers} we have
	\begin{equation}\label{K0zrg}
			K_n(0,z;r) = \frac{(-1)^n}{D^{(r)}_{n+1}} \mathcal{D}_n[w(\ze)\ze^{-r-2}(z-\ze)] =  (-1)^n\frac{D^{(r+2)}_n}{D^{(r)}_{n+1}} P_n(z;r+2),
	\end{equation}
where the last equality follows from \eqref{MultIntPolysP}. From \eqref{Qn0r1} we have  
	\begin{equation}\label{Qn0r/hnr}
		\frac{Q_n(0;r)}{h^{(r)}_n} = (-1)^n\frac{D^{(r+2)}_n}{D^{(r)}_{n+1}}.
	\end{equation}
Now \eqref{Pr+2 Pls} follows from \eqref{K0zr}, \eqref{K0zrg}, and \eqref{Qn0r/hnr}. The identity \eqref{Qr-1 Qls} can be proven by considering: 
		\begin{equation}\label{K0zr2}
		K_n(z,0;r) = \sum_{\ell=0}^{n} \frac{1}{h^{(r)}_{\ell}}Q_{\ell}(z;r)P_{\ell}(0;r),
		\end{equation}
and 
		\begin{equation}\label{K0zrg2}
			K_n(z,0;r) = \frac{(-1)^n}{D^{(r)}_{n+1}} \mathcal{D}_n[w(\ze)\ze^{-r+1}(z-\ze^{-2})]=(-1)^n\frac{D^{(r-1)}_n}{D^{(r)}_{n+1}} Q_n(z;r-1),
		\end{equation}
by \eqref{MultIntRepKers} and \eqref{MultIntPolysQ}. Equation \eqref{Qr-1 Qls} now follows from \eqref{K0zr2}, \eqref{K0zrg2} and 
		\begin{equation}\label{Pn0r/hnr}
		\frac{P_n(0;r)}{h^{(r)}_n} = (-1)^n\frac{D^{(r-1)}_n}{D^{(r)}_{n+1}},
		\end{equation}
		which is obtained by \eqref{Pn0r1}.
		\eqref{Rs+1 Rls} and \eqref{Ss-2 Sls} can be proven similarly, respectively by considering $ L_n(0,z;s)$ and $L_n(z,0;s)$ and employing
		\begin{equation}\label{Rn0s/Hns}
			\frac{S_n(0;s)}{g^{(s)}_n} = (-1)^n\frac{E^{(s+1)}_n}{E^{(s)}_{n+1}}, \qandq \frac{R_n(0;s)}{g^{(s)}_n} = (-1)^n\frac{E^{(s-2)}_n}{E^{(s)}_{n+1}},
		\end{equation}
		which are consequences of \eqref{Rn0s1} and \eqref{Sn0s1} respectively.
\end{proof}
\begin{corollary}\label{Cor661}
The following mixed recurrence relations hold

\noindent\begin{minipage}{.5\linewidth}
	\begin{alignat}{2}
	&P_{n}(z;r) &&=P_{n}(z;r+2) + \frac{h^{(r)}_{n}}{h^{(r+2)}_{n-1}}
	P_{n-1}(z;r+2), \label{Pr Pr+2s} \\ 
	&Q_{n}(z;r) &&=Q_{n}(z;r-1) + \frac{h^{(r)}_{n}}{h^{(r-1)}_{n-1}}
	Q_{n-1}(z;r-1) , \label{Qr Qr-1s}
		\end{alignat}	
	\end{minipage}	
	\begin{minipage}{.5\linewidth}
		\begin{alignat}{2}
		&R_{n}(z;s) &&=R_{n}(z;s+1) + \frac{g^{(s)}_{n}}{g^{(s+1)}_{n-1}}
		R_{n-1}(z;s+1), \label{Rs Rs+1s} \\
		&S_{n}(z;s) &&=S_{n}(z;s-2) + \frac{g^{(s)}_{n}}{g^{(s-2)}_{n-1}}
		S_{n-1}(z;s-2). \label{Ss Ss-2s}
		\end{alignat}	
	\end{minipage}	
\end{corollary}
\begin{proof} 
Writing \eqref{Pr+2 Pls} with the replacement $r \mapsto r+2$ for $n$ and $n-1$ and subtracting gives 
		\[\frac{Q_n(0;r)}{h^{(r)}_n}P_{n}(z;r+2) - \frac{Q_{n-1}(0;r)}{h^{(r)}_{n-1}}P_{n-1}(z;r+2)  = \frac{1}{h^{(r)}_{n}}Q_{n}(0;r) P_{n}(z;r),\] 
or equivalently
	\[P_{n}(z;r)=P_{n}(z;r+2) - \frac{h^{(r)}_{n}}{Q_{n}(0;r)}
	\frac{Q_{n-1}(0;r)}{h^{(r)}_{n-1}}P_{n-1}(z;r+2). \]
Now \eqref{Pr Pr+2s} immediately follows by recalling \eqref{Qn0r/hnr} and \eqref{h}. 
The other three identities can be shown in a similar fashion.
\end{proof}

For fixed offset values $\hat{r}, \hat{s} \in \Z$, let us consider the following $2j-k$ and $j-2k$ polynomial bases:
\begin{equation}\label{PolyCollections}
\mathscr{P}^{(\hat{r})}_n:=\{P_m(z;\hat{r})\}^{n}_{m=0}, \qquad 	\mathscr{Q}^{(\hat{r})}_n:=\{Q_m(z;\hat{r})\}^{n}_{m=0}, \qquad	\mathscr{R}^{(\hat{s})}_n:=\{R_m(z;\hat{s})\}^{n}_{m=0}, \qquad	\mathscr{S}^{(\hat{s})}_n:=\{S_m(z;\hat{s})\}^{n}_{m=0}.
\end{equation}
We obviously assume that the corresponding determinants do not vanish so that the polynomials exist (and are unique, see Theorems \ref{P and Q exist and are unique} and \ref{R and S exist and are unique}). The following theorem characterizes the number of polynomials needed to represent $P_n(z,r), \ Q_n(z,r), \ R_n(z,s)$ and $S_n(z,s)$ respectively in terms of polynomials in the bases $\mathscr{P}^{(\hat{r})}_n$, $\mathscr{Q}^{(\hat{r})}_n$,  $\mathscr{R}^{(\hat{s})}_n$, and $\mathscr{S}^{(\hat{s})}_n$ for various values of $\hat{r}$ and $\hat{s}$.
\begin{theorem}
Assume the non-vanishing requirements of Theorems \ref{P and Q exist and are unique} and \ref{R and S exist and are unique} so that the polynomial bases $\mathscr{P}^{(\hat{r})}_n$, $\mathscr{Q}^{(\hat{r})}_n$, $\mathscr{R}^{(\hat{s})}_n$, and $\mathscr{S}^{(\hat{s})}_n$ are well-defined as in \eqref{PolyCollections} and let $m \in \N$. Then 
	\begin{itemize}
		\item[1.] If it is desired to express $P_n(z;r)$ as a linear combination of polynomials in the basis $\mathscr{P}^{(r-2m)}_n$, then one needs to use \textit{all $n+1$ polynomials} in $\mathscr{P}^{(r-2m)}_n$. While if it is desired to express $P_n(z;r)$ in terms of polynomials in the basis $\mathscr{P}^{(r+2m)}_n$, then one needs to only use the following $m+1$ polynomials in the subset
			\[ \left\{ P_{n}(z;r+2m), P_{n-1}(z;r+2m), \cdots, P_{n-m}(z;r+2m) \right\} \subset \mathscr{P}^{(r+2m)}_n. \]
		\item[2.]  If it is desired to express $zP_n(z;r)$ as a linear combination of polynomials in $\mathscr{P}^{(r-2m+1)}_{n+1}$, then one needs to use \textit{all $n+2$ polynomials} in $\mathscr{P}^{(r-2m+1)}_{n+1}$. While if it is desired to express $zP_n(z;r)$ in terms of polynomials in $\mathscr{P}^{(r+2m-1)}_{n+1}$, then one needs to only use the $m+1$ polynomials in the subset
		\[ \{ P_{n+1}(z;r+2m-1), P_{n}(z;r+2m-1), \cdots, P_{n+1-m}(z;r+2m-1) \} \subset \mathscr{P}^{(r+2m-1)}_{n+1}. \]
		\item[3.] If it is desired to express $Q_n(z;r)$ as a linear combination of polynomials in $\mathscr{Q}^{(r+m)}_n$, then one needs to use \textit{all $n+1$ polynomials} in $\mathscr{Q}^{(r+m)}_n$. While if it is desired to express $Q_n(z;r)$ in terms of polynomials in $\mathscr{Q}^{(r-m)}_n$, then one needs to only use the $m+1$ polynomials in the subset
\[ \{ Q_{n}(z;r-m), Q_{n-1}(z;r-m), \cdots, Q_{n-m}(z;r-m) \} \subset \mathscr{Q}^{(r-m)}_n. \]
		\item[4.]  If it is desired to express $R_n(z;s)$ as a linear combination of polynomials in $\mathscr{R}^{(s-m)}_{n}$, then one needs to use \textit{all $n+1$ polynomials} in $\mathscr{R}^{(s-m)}_{n}$. While if it is desired to express $R_n(z;s)$ in terms of polynomials in $\mathscr{R}^{(s+m)}_{n}$, then one needs to only use the $m+1$ polynomials in the subset
\[ \{ R_{n}(z;s+m), R_{n-1}(z;s+m), \cdots, R_{n-m}(z;s+m) \} \subset \mathscr{R}^{(s+m)}_{n}. \]
		\item[5.] If it is desired to express $S_n(z;s)$ as a linear combination of polynomials in $\mathscr{S}^{(s+2m)}_n$, then one needs to use \textit{all $n+1$ polynomials} in $\mathscr{S}^{(s+2m)}_n$. While if it is desired to express $S_n(z;s)$ in terms of polynomials in $\mathscr{S}^{(s-2m)}_n$, then one needs to only use the $m+1$ polynomials in the subset
\[ \{ S_{n}(z;s-2m), S_{n-1}(z;s-2m), \cdots, S_{n-m}(z;s-2m) \} \subset \mathscr{S}^{(s-2m)}_n. \]
		\item[6.]  If it is desired to express $zS_n(z;s)$ as a linear combination of polynomials in $\mathscr{S}^{(s+2m-1)}_{n+1}$, then one needs to use \textit{all $n+2$ polynomials} in $\mathscr{S}^{(s+2m-1)}_{n+1}$. While if it is desired to express $zS_n(z;s)$ in terms of polynomials in $\mathscr{P}^{(s-2m+1)}_{n+1}$, then one needs to only use the $m+1$ polynomials in the subset
\[ \{ S_{n+1}(z;s-2m+1), S_{n}(z;s-2m+1), \cdots, S_{n+1-m}(z;s-2m+1) \} \subset \mathscr{S}^{(s-2m+1)}_{n+1}. \]
	\end{itemize}
\end{theorem}
\begin{proof}
Theorem \ref{Thm6.7} establishes the statement of this theorem for expressing $P_n(z,r), \ Q_n(z,r), \ R_n(z,s)$ and $S_n(z,s)$ respectively in terms of polynomials in $\mathscr{P}^{(r-2)}_n$, $\mathscr{Q}^{(r+1)}_n$,  $\mathscr{R}^{(s-1)}_n$, and $\mathscr{S}^{(s+2)}_n$, and the Corollary \ref{Cor661} establishes this for expressing $P_n(z,r), \ Q_n(z,r), \ R_n(z,s)$ and $S_n(z,s)$ respectively in terms of polynomials in  $\mathscr{P}^{(r+2)}_n$, $\mathscr{Q}^{(r-1)}_n$, $\mathscr{R}^{(s+1)}_n$, and $\mathscr{S}^{(s-2)}_n$. The proof for general $m$ follows from iterations of the aforementioned results of Theorem \ref{Thm6.7} and Corollary \ref{Cor661}, and the non-vanishing determinant requirements of Theorems \ref{P and Q exist and are unique} and \ref{R and S exist and are unique} ensure that the coefficients involved all exist and are non-zero. Let us now prove the first item for expressing $P_n(z;r)$ in the basis $\mathscr{P}^{(r-2m)}_n$ via induction. The base of the induction for $m=1$ clearly holds because of \eqref{Pr+2 Pls}. Passing from the induction hypothesis to the desired result is obvious because by \eqref{Pr+2 Pls} we have $P_{\ell}(z;r-2(m-1)) =  \sum_{\nu=0}^{\ell} \hat{A}_{\nu,\ell} P_{\nu}(z;r-2m)$, where the constants $\hat{A}_{\nu,\ell}$ are also described in \eqref{Pr+2 Pls}. Thus the induction hypothesis  $	 	P_{n}(z;r) = \sum_{\ell=0}^{n} A_{m,n} P_{\ell}(z;r-2(m-1))$, immediately implies the desired result.	To prove the first item for expressing $P_n(z;r)$ in the basis $\mathscr{P}^{(r+2m)}_n$ we use induction again. The base of the induction clearly holds due to \eqref{Pr Pr+2s}. Now the induction hypothesis  
	\begin{equation}
	P_{n}(z;r) = \sum_{\ell=0}^{m-1} B_{m-1,n} P_{n-\ell}(z;r+2(m-1)),
	\end{equation}
in conjunction with 
	\begin{equation}
		P_{n-\ell}(z;r+2(m-1)) = P_{n-\ell}(z;r+2m) + \frac{h^{(r+2(m-1))}_{n-\ell}}{h^{(r+2m)}_{n-\ell-1}}
		P_{n-\ell-1}(z;r+2m),
	\end{equation}
yields the desired result. The proofs of the other items follow similarly, where the only extra needed relations are

\noindent\begin{minipage}{.5\linewidth}
	\begin{alignat}{2}
	&zP_{n}(z;r) &&=P_{n+1}(z;r+1) - \de_n^{(r+1)}P_{n}(z;r+1), \label{zPnrPnr+1}
	\end{alignat}	
\end{minipage}	
\begin{minipage}{.5\linewidth}
	\begin{alignat}{2}
	&zS_{n}(z;s) &&=S_{n+1}(z;s-1) - \ga_n^{(s-1)}S_{n}(z;s-1). \label{zSnsSns-1}
	\end{alignat}	
\end{minipage}	
which are obtained	from \eqref{1st rec P} and \eqref{1st rec S^*} respectively.	
\end{proof}
In the Toeplitz case there exist relations which allow one of the polynomials to be expressed in terms of a linear combination of two of the other one,
see Eq.(2.21,22) of \cite{FW_2006}. These two relations are also the components of the first order recurrence, Eq.(2.70), in the matrix variable, Eq.(2.69).
Furthermore these two relations are somewhat symmetrical. However for the present system we find the following single, unsymmetrical relation, 
of $ Q_n $ in terms of a bilinear combination of the $ P_n $, and the analogous relation for the $ R_n, S_n $ system.
\begin{theorem}
It holds that
	\begin{equation}\label{222}
	Q_n(z^{2};r)=\frac{D^{(r+2)}_n}{D^{(r)}_n}\frac{z^{2n+1}}{2}\left[ P_{n+1}(z^{-1};r+2)P_n(-z^{-1};r+2) - P_{n+1}(-z^{-1};r+2)P_n(z^{-1};r+2) \right],
	\end{equation}
	\begin{equation}\label{333}
	R_n(z^{2};s)=\frac{E^{(s-2)}_n}{E^{(s)}_n}\frac{z^{2n+1}}{2}\left[ S_{n+1}(z^{-1};s-2)S_n(-z^{-1};s-2) - S_{n+1}(-z^{-1};s-2)S_n(z^{-1};s-2) \right].
	\end{equation} 
\end{theorem}
\begin{proof} 
We only prove \eqref{222} as \eqref{333} can be proven in a similar way. From \eqref{MultIntPolysQ} we have
	\begin{equation}\label{111}
	Q_n(z^2;r)  = \frac{1}{D^{(r)}_n} \mathcal{D}_n[w(\ze)\ze^{-r}(z^2-\ze^{-2})].
	\end{equation}
Using $z-\ze^{-1} = - \ze^{-1}z(z^{-1}-\ze)$, $z+\ze^{-1} = - \ze^{-1}z(-z^{-1}-\ze)$,  \eqref{111} and \eqref{DD} we have
	\begin{equation}
	\begin{split}
	&D^{(r)}_n Q_n(z^2;r) 
	\\ & = \frac{z^{2n}}{n!} \int_{\T} \frac{\dd \ze_1}{2 \pi \ic \ze_1} \cdots \int_{\T} \frac{\dd \ze_n}{2 \pi \ic \ze_n} \prod_{j=1}^{n}w(\ze_j)\ze^{-r-2}_j \prod_{j=1}^{n}(z^{-1}-\ze_j)(-z^{-1}-\ze_j) \prod_{1\leq j<k\leq n} (\ze_k-\ze_j)(\ze^{-2}_k-\ze^{-2}_j).
	\end{split}
	\end{equation}
Now, we temporarily denote $z^{-1}$ and $-z^{-1}$ respectively by $\ze_{n+1}$ and $\ze_{n+2}$, and thus
	\begin{equation}
	\begin{split}
	&D^{(r)}_n Q_n(z^2;r)  
	\\ & = \frac{z^{2n}}{n!} \int_{\T} \frac{\dd \ze_1}{2 \pi \ic \ze_1} \cdots \int_{\T} \frac{\dd \ze_n}{2 \pi \ic \ze_n} \prod_{j=1}^{n}w(\ze_j)\ze^{-r-2}_j  \frac{1}{\ze_{n+2}-\ze_{n+1}}\prod_{1\leq j<k\leq n+2} (\ze_k-\ze_j)\prod_{1\leq j<k\leq n} (\ze^{-2}_k-\ze^{-2}_j) 
	\\ & = -\frac{1}{2}\frac{z^{2n+1}}{n!} \int_{\T} \frac{\dd \ze_1}{2 \pi \ic \ze_1} \cdots \int_{\T} \frac{\dd \ze_n}{2 \pi \ic \ze_n} \prod_{j=1}^{n}w(\ze_j)\ze^{-r-2}_j  \prod_{1\leq j<k\leq n+2} (\ze_k-\ze_j)\prod_{1\leq j<k\leq n} (\ze^{-2}_k-\ze^{-2}_j) 
	\\ & =  -\frac{1}{2}\frac{z^{2n+1}}{n!} \int_{\T} \frac{\dd \ze_1}{2 \pi \ic \ze_1} \cdots \int_{\T} \frac{\dd \ze_n}{2 \pi \ic \ze_n} \prod_{j=1}^{n}w(\ze_j)\ze^{-r-2}_j \underset{\substack{1 \leq j \leq n \\ 1 \leq k\leq n+2}}{\det} 
	\begin{pmatrix}
	P_{k-1}(\ze_j;r+2) \\ P_{k-1}(z^{-1};r+2) \\ P_{k-1}(-z^{-1};r+2)
	\end{pmatrix} \underset{\substack{1 \leq j \leq n \\ 1 \leq k\leq n}}{\det} \{Q_{k-1}(\ze^{-2}_j;r+2)\} .
	\end{split}
	\end{equation}
We have used the offsets to be $r+2$, 
due to presence of $\ze^{-r-2}_j$ and the fact that we are about to use the bi-orthogonality property of $P$'s and $Q$'s. 
We have 
	\begin{equation}
	\begin{split}
	&D^{(r)}_n Q_n(z^2;r)  
	\\ &  = -\frac{1}{2}z^{2n+1} \int_{\T} \frac{\dd \ze_1}{2 \pi \ic \ze_1} \cdots \int_{\T} \frac{\dd \ze_n}{2 \pi \ic \ze_n} \prod_{j=1}^{n}w(\ze_j)\ze^{-r-2}_j \underset{\substack{1 \leq j \leq n \\ 1 \leq k\leq n+2}}{\det} \begin{pmatrix}
	P_{k-1}(\ze_j;r+2) \\ P_{k-1}(z^{-1};r+2) \\ P_{k-1}(-z^{-1};r+2)
	\end{pmatrix} \prod_{j=1}^{n} Q_{j-1}(\ze^{-2}_j;r+2) 
	\\ & = -\frac{1}{2}z^{2n+1}
	\underset{\substack{1 \leq j \leq n \\ 1 \leq k\leq n+2}}{\det} \begin{pmatrix}
	\di \int_{\T} \frac{\dd \ze}{2 \pi \ic \ze} w(\ze)\ze^{-r-2} P_{k-1}(\ze;r+2)Q_{j-1}(\ze^{-2};r+2) \\ P_{k-1}(z^{-1};r+2) \\ P_{k-1}(-z^{-1};r+2)
	\end{pmatrix} 
	\\ & = -\frac{1}{2}z^{2n+1}  \det \begin{pmatrix}
	h^{(r+2)}_0  & \cdots & 0 & 0 & 0 \\
	\vdots  & \ddots & \vdots & \vdots & \vdots \\
	0  & \cdots & h^{(r+2)}_{n-1} & 0 & 0 \\
	P_0(z^{-1};r+2)  & \cdots & P_{n-1}(z^{-1};r+2) & P_n(z^{-1};r+2) & P_{n+1}(z^{-1};r+2) \\
	P_0(-z^{-1};r+2)  & \cdots& P_{n-1}(-z^{-1};r+2) & P_n(-z^{-1};r+2) & P_{n+1}(-z^{-1};r+2) \\	
	\end{pmatrix}. 
	\end{split}
	\end{equation}
Therefore
	\begin{equation}
	D^{(r)}_n Q_n(z^2;r) = -\frac{1}{2}z^{2n+1}  \prod_{j=0}^{n-1} h^{(r+2)}_j \left[P_n(z^{-1};r+2)P_{n+1}(-z^{-1};r+2) - P_{n+1}(z^{-1};r+2)P_n(-z^{-1};r+2)  \right].
	\end{equation}
Noticing that $\prod_{j=0}^{n-1} h^{(r+2)}_j = D^{(r+2)}_n$ proves \eqref{222}.
\end{proof}
\begin{remark} \normalfont
Since \[z^{n+1}P_{n+1}(z^{-1};\cdot) \Bigg|_{z=0}=z^{n}P_{n}(z^{-1};\cdot) \Bigg|_{z=0}=1, \qandq -z^{n+1}P_{n+1}(-z^{-1};\cdot) \Bigg|_{z=0}=z^{n}P_{n}(-z^{-1};\cdot) \Bigg|_{z=0}=(-1)^{n}, \]
we observe that \eqref{222} is compatible with \eqref{Qn0r1}. A similar observation shows the compatibility of \eqref{333} with \eqref{Rn0s1}.
\end{remark}

\subsection{The Christoffel-Darboux Identity}\label{subsec CD}
Having defined the reproducing kernel via a sum over products of bi-orthogonal polynomials and deduced that these polynomials satisfy third order difference equations
we seek the explicit evaluation of the sum involved in the kernel. 
The well known Christoffel-Darboux sum in the Toeplitz result can be found in Prop. 2.5 of \cite{FW_2006}.
\begin{theorem}
The Christoffel-Darboux identity for the $2j-k$ and $j-2k$ systems can be respectively written as
	\begin{equation}\label{CD}
	K_n(z^2_2,z_1;r)=\frac{1}{2}\frac{D_n^{(r+2)}}{D_{n+1}^{(r)}} \frac{z^{2n+1}_2}{z^{2}_1-z^{-2}_2} \det \begin{pmatrix}
	P_n(-z^{-1}_2;r+2) & P_{n+1}(-z^{-1}_2;r+2) & P_{n+2}(-z^{-1}_2;r+2) \\[2pt]
	P_n(z^{-1}_2;r+2) & P_{n+1}(z^{-1}_2;r+2) & P_{n+2}(z^{-1}_2;r+2) \\[2pt]
	P_n(z_1;r+2) & P_{n+1}(z_1;r+2) & P_{n+2}(z_1;r+2) \\				
	\end{pmatrix},
	\end{equation}
and
	\begin{equation}\label{CDRS}
	L_n(z_2,z^2_1;s)=\frac{1}{2}\frac{E_n^{(s-2)}}{E_{n+1}^{(s)}} \frac{z^{2n+1}_1}{z^{-2}_1-z^{2}_2} \det \begin{pmatrix}
	S_n(z^{-1}_1;s-2) & S_{n+1}(z^{-1}_1;s-2) & S_{n+2}(z^{-1}_1;s-2) \\[2pt]
	S_n(-z^{-1}_1;s-2) & S_{n+1}(-z^{-1}_1;s-2) & S_{n+2}(-z^{-1}_1;s-2) \\[2pt]
	S_n(z_2;s-2) & S_{n+1}(z_2;s-2) & S_{n+2}(z_2;s-2) \\				
	\end{pmatrix}. 
	\end{equation}
\end{theorem}
\begin{proof}
Using \eqref{MultIntRepKers} we have
	\begin{equation}
	\begin{split}
	D^{(r)}_{n+1} K_n(z^2_2,z_1;r) & =  \mathcal{D}_n[w(\ze)\ze^{-r}(z_1-\ze)(z^2_2-\ze^{-2})] \\ & = \mathcal{D}_n[w(\ze)\ze^{-r}(z_1-\ze)(z_2-\ze^{-1})(z_2+\ze^{-1})]  \\ & = \mathcal{D}_n[w(\ze)\ze^{-r}(z_1-\ze)z^2_2\ze^{-2}(z^{-1}_2-\ze)(-z^{-1}_2-\ze)]\\ & = z^{2n}_2\mathcal{D}_n[w(\ze)\ze^{-r-2}(z_1-\ze)(z^{-1}_2-\ze)(-z^{-1}_2-\ze)].
	\end{split}
	\end{equation}
Let us temporarily denote $-z^{-1}_2$, $z^{-1}_2$, and $z_1$ respectively by $\ze_{n+1}$, $\ze_{n+2}$, and $\ze_{n+3}$. Note that
	\begin{equation}
	\begin{split}
	\prod_{1\leq j<k\leq n+3} (\ze_k-\ze_j)  = & (\ze_{n+3}-\ze_{n+2})(\ze_{n+3}-\ze_{n+1})(\ze_{n+2}-\ze_{n+1})  \prod_{j=1}^{n} (\ze_{n+3}-\ze_{j})  (\ze_{n+2}-\ze_{j})  (\ze_{n+1}-\ze_{j}) \prod_{1\leq j<k\leq n} (\ze_k-\ze_j).
	\end{split}
	\end{equation}
We have
	\begin{equation*}
	\begin{split}
	&		D^{(r)}_{n+1} K_n(z^2_2,z_1;r) = z^{2n}_2\mathcal{D}_n[w(\ze)\ze^{-r-2}(z_1-\ze)(z^{-1}_2-\ze)(-z^{-1}_2-\ze)]  \\ & = \frac{z_2^{2n}}{n!} \int_{\T} \frac{\dd \ze_1}{2 \pi \ic \ze_1} \cdots \int_{\T} \frac{\dd \ze_n}{2 \pi \ic \ze_n} \prod_{j=1}^{n}w(\ze_j)\ze^{-r-2}_j \prod_{j=1}^{n}(z_1-\ze_j)(z_2^{-1}-\ze_j)(-z_2^{-1}-\ze_j) \prod_{1\leq j<k\leq n} (\ze_k-\ze_j)(\ze^{-2}_k-\ze^{-2}_j)  \\ & = 
	\frac{z_2^{2n+1}}{n!2(z_1-z^{-1}_{2})(z_1+z^{-1}_{2})} \int_{\T} \frac{\dd \ze_1}{2 \pi \ic \ze_1} \cdots \int_{\T} \frac{\dd \ze_n}{2 \pi \ic \ze_n} \prod_{j=1}^{n}w(\ze_j)\ze^{-r-2}_j \prod_{1\leq j<k\leq n+3} (\ze_k-\ze_j) \prod_{1\leq j<k\leq n} (\ze^{-2}_k-\ze^{-2}_j)  \\ & = \frac{z_2^{2n+1}}{n!2(z_1-z^{-1}_{2})(z_1+z^{-1}_{2})} \int_{\T} \frac{\dd \ze_1}{2 \pi \ic \ze_1} \cdots \int_{\T} \frac{\dd \ze_n}{2 \pi \ic \ze_n} \prod_{j=1}^{n}w(\ze_j)\ze^{-r-2}_j 	\underset{\substack{1 \leq j \leq n \\ 1 \leq k\leq n+3}}{\det} \begin{pmatrix}
	P_{k-1}(\ze_j;r+2) \\[2pt] P_{k-1}(-z_2^{-1};r+2) \\[2pt] P_{k-1}(z_2^{-1};r+2) \\[2pt] P_{k-1}(z_1;r+2)
	\end{pmatrix} \underset{\substack{1 \leq j \leq n \\ 1 \leq k\leq n}}{\det} \{Q_{k-1}(\ze^{-2}_j;r+2)\}  \\  & =  \frac{z_2^{2n+1}}{2(z^2_1-z^{-2}_{2})} \int_{\T} \frac{\dd \ze_1}{2 \pi \ic \ze_1} \cdots \int_{\T} \frac{\dd \ze_n}{2 \pi \ic \ze_n} \prod_{j=1}^{n}w(\ze_j)\ze^{-r-2}_j  	\underset{\substack{1 \leq j \leq n \\ 1 \leq k\leq n+3}}{\det} \begin{pmatrix}
	P_{k-1}(\ze_j;r+2) \\[2pt] P_{k-1}(-z_2^{-1};r+2) \\[2pt] P_{k-1}(z_2^{-1};r+2) \\[2pt] P_{k-1}(z_1;r+2)
	\end{pmatrix} \prod_{j=1}^{n}Q_{j-1}(\ze^{-2}_j;r+2) \\
	\end{split}	
	\end{equation*}
	\begin{equation*}
	\begin{split}
	 & = 
	\frac{z_2^{2n+1}}{2(z^2_1-z^{-2}_{2})} 
	\underset{\substack{1 \leq j \leq n \\ 1 \leq k\leq n+2}}{\det} \begin{pmatrix}
	\di \int_{\T} \frac{\dd \ze}{2 \pi \ic \ze} w(\ze)\ze^{-r-2} P_{k-1}(\ze;r+2)Q_{j-1}(\ze^{-2};r+2) \\[2pt] P_{k-1}(-z_2^{-1};r+2) \\[2pt] P_{k-1}(z_2^{-1};r+2)
	\\[2pt] P_{k-1}(z_1;r+2)\end{pmatrix}  
	\\ & =
	\frac{z_2^{2n+1}}{2(z^2_1-z^{-2}_{2})}
	\det \begin{pmatrix}
	h^{(r+2)}_0 &  \cdots & 0 & 0 & 0 & 0 \\
	\vdots  & \ddots & \vdots & \vdots & \vdots & \vdots  \\
	0 &  \cdots & h^{(r+2)}_{n-1} & 0 & 0 & 0 \\[2pt]
	P_0(-z_2^{-1};r+2)  & \cdots & P_{n-1}(-z_2^{-1};r+2) & P_n(-z_2^{-1};r+2) & P_{n+1}(-z_2^{-1};r+2) & P_{n+2}(-z_2^{-1};r+2) \\[2pt]
	P_0(z_2^{-1};r+2)  & \cdots& P_{n-1}(z_2^{-1};r+2) & P_n(z_2^{-1};r+2) & P_{n+1}(z_2^{-1};r+2) & P_{n+2}(z_2^{-1};r+2) \\[2pt]	P_0(z_1;r+2)  & \cdots& P_{n-1}(z_1;r+2) & P_n(z_1;r+2) & P_{n+1}(z_1;r+2) & P_{n+2}(z_1;r+2) \\
	\end{pmatrix}.
	\end{split}	
	\end{equation*}
Now \eqref{CD} follows immediately by noticing $\prod_{j=0}^{n-1} h^{(r+2)}_j = D^{(r+2)}_n$.
\end{proof}
\begin{remark} \normalfont
We have found that the sum is most simply evaluated in terms of just the polynomials $P_n$ and is therefore unsymmetrical with respect to $Q_n$.
It is possible to replace certain terms in this determinant by $Q_n$ using \eqref{222} but not all in a symmetrical way.
\end{remark}

\section{Associated Functions}\label{Sec Associated Functions}
Our goal in this section is to define linearly independent associated functions corresponding to the polynomials $P_{n}(z;r)$, $Q_{n}(z;r)$, $R_{n}(z;s)$, and $S_{n}(z;s)$ each of which respectively satisfies the recurrence relations \eqref{P pure n rec}, \eqref{Q* pure n rec}, \eqref{R pure n rec}, \eqref{S* pure n rec} or their equivalents. 
To this end, let us define the associated functions

\begin{align}
\widecheck{P}_n(z;r) & := \int_{\T} \frac{\dd \ze}{2 \pi \ic \ze} w(\ze) \ze^{-r} \frac{\ze^2+z^2}{\ze^2-z^2}[P_n(\ze;r)-P_n(z;r)], \label{Pcirc} \\
\widecheck{Q}_n(z;r) & :=  \int_{\T} \frac{\dd \ze}{2 \pi \ic \ze} w(\ze) \ze^{-r} \frac{\ze+z}{\ze-z}[Q_n(\ze^{-2};r)-Q_n(z^{-2};r)], \label{Qcirc}\\
\widecheck{R}_n(z;s) & := \int_{\T} \frac{\dd \ze}{2 \pi \ic \ze} w(\ze) \ze^{-s} \frac{\ze+z}{\ze-z}[R_n(\ze^{2};s)-R_n(z^{2};s)], \label{Rcirc}\\
\widecheck{S}_n(z;s) &:=  \int_{\T} \frac{\dd \ze}{2 \pi \ic \ze} w(\ze) \ze^{-s} \frac{\ze^2+z^2}{\ze^2-z^2}[S_n(\ze^{-1};s)-S_n(z^{-1};s)]. \label{Scirc}
\end{align}
It is also convenient to write these associated functions as 

\noindent\begin{minipage}{.5\linewidth}
	\begin{alignat}{2}
	&\widecheck{P}_n(z;r) &&=\widehat{P}_n(z;r) - F_1(z;r)P_n(z;r), \label{Pr PPF} \\ 
	&\widecheck{Q}_n(z;r) &&=\widehat{Q}_n(z;r) - F_2(z;r)Q_n(z^{-2};r), \label{QQF}
	\end{alignat}	
\end{minipage}	
\begin{minipage}{.5\linewidth}
	\begin{alignat}{2}
	&\widecheck{R}_n(z;s) &&= \widehat{R}_n(z;s) - F_2(z;s)R_n(z^2;s), \label{RRF} \\
	&\widecheck{S}_n(z;s) &&=\widehat{S}_n(z;s) - F_1(z;s)S_n(z^{-1};s). \label{SSF}
	\end{alignat}	
\end{minipage}	
where

\noindent\begin{minipage}{.5\linewidth}
	\begin{alignat}{2}
	&\widehat{P}_n(z;r) &&:=\int_{\T} \frac{\dd \ze}{2 \pi \ic \ze} w(\ze) \ze^{-r} \frac{\ze^2+z^2}{\ze^2-z^2}P_n(\ze;r), \label{Phat} \\ 
	&\widehat{Q}_n(z;r) &&:=\int_{\T} \frac{\dd \ze}{2 \pi \ic \ze} w(\ze) \ze^{-r} \frac{\ze+z}{\ze-z}Q_n(\ze^{-2};r), \label{Qhat}
	\end{alignat}	
\end{minipage}	
\begin{minipage}{.5\linewidth}
	\begin{alignat}{2}
	&\widehat{R}_n(z;s) &&:= \int_{\T} \frac{\dd \ze}{2 \pi \ic \ze} w(\ze) \ze^{-s} \frac{\ze+z}{\ze-z}R_n(\ze^{2};s), \label{Rhat} \\
	&\widehat{S}_n(z;s) &&:=\int_{\T} \frac{\dd \ze}{2 \pi \ic \ze} w(\ze) \ze^{-s} \frac{\ze^2+z^2}{\ze^2-z^2}S_n(\ze^{-1};s). \label{Shat}
	\end{alignat}	
\end{minipage}	
and

\noindent\begin{minipage}{.5\linewidth}
	\begin{alignat}{2}
	&F_1(z;r) &&:=\int_{\T} \frac{\dd \ze}{2 \pi \ic \ze} w(\ze) \ze^{-r} \frac{\ze^2+z^2}{\ze^2-z^2}, \label{F1} 
\end{alignat}	
\end{minipage}	
\begin{minipage}{.5\linewidth}
	\begin{alignat}{2}
	&F_2(z;r) &&:=\int_{\T} \frac{\dd \ze}{2 \pi \ic \ze} w(\ze) \ze^{-r} \frac{\ze+z}{\ze-z}. \label{F2}
	\end{alignat}	
\end{minipage}
Notice that due to
\begin{equation}\label{relations}
	\frac{\ze^2+z^2}{\ze^2-z^2} = \frac{1}{2}\left[ \frac{\ze+z}{\ze-z}+\frac{\ze-z}{\ze+z} \right], 
\end{equation}
$F_1$ can be written in terms of $F_2$ as
\begin{equation}\label{F1 F2}
	F_1(z;r)=\frac{1}{2}\big(F_2(z;r)+F_2(-z;r)\big).
\end{equation}
For the same reason we have the following expressions for $\widehat{P}_n$ and $\widehat{S}_n$ 

\noindent\begin{minipage}{.5\linewidth}
	\begin{alignat}{2}
	&\widehat{P}_n(z;r) &&=\frac{1}{2}\left(P_n^{\dagger}(z;r)+P_n^{\dagger}(-z;r)\right), \label{PPP} 
	\end{alignat}	
\end{minipage}	
\begin{minipage}{.5\linewidth}
	\begin{alignat}{2}
	&\widehat{S}_n(z;s) &&=\frac{1}{2}\left(S_n^{\dagger}(z;s)+S_n^{\dagger}(-z;s)\right), \label{SSS}	
	\end{alignat}	
\end{minipage}		
where 

\noindent\begin{minipage}{.5\linewidth}
	\begin{alignat}{2}
	&P_n^{\dagger}(z;r) &&=\int_{\T} \frac{\dd \ze}{2 \pi \ic \ze} w(\ze) \ze^{-r} \frac{\ze+z}{\ze-z}P_n(\ze;r), \label{PP} 
	\end{alignat}	
\end{minipage}	
\begin{minipage}{.5\linewidth}
	\begin{alignat}{2}
	&S_n^{\dagger}(z;s) &&=\int_{\T} \frac{\dd \ze}{2 \pi \ic \ze} w(\ze) \ze^{-s} \frac{\ze+z}{\ze-z}S_n(\ze^{-1};s). \label{SS}	
	\end{alignat}	
\end{minipage}

In the following theorem we show that the associated functions above satisfy the same recurrence relations as the orthogonal polynomials. To this end, let us first define the following linear difference operators
\begin{equation}\label{L1}
		\mathcal{L}_1[f_n(z;r)] := f_{n+3}(z;r) -(\de_{n+2}^{(r)}+\de_{n+1}^{(r-1)})f_{n+2}(z;r) + (\de_{n+1}^{(r-1)}\de_{n+1}^{(r)} - z^2)f_{n+1}(z;r)+ (\de_n^{(r)}+\eta^{(r-2)}_n)z^2f_{n}(z;r),
\end{equation}
\begin{equation}\label{L2}
	\begin{split}
		\mathcal{L}_2[f_n(z;r)] & := z^2f_{n+3}(z;r) -(1+\be_{n+2}^{(r)}z^2)f_{n+2}(z;r) +  (\be^{(r)}_{n+1}+\al^{(r+1)}_{n+1}+\be^{(r+1)}_{n+1}+\al^{(r+2)}_{n+1})f_{n+1}(z;r) \\ &-(\be^{(r+1)}_{n+1}+\al^{(r+2)}_{n+1})(\be^{(r)}_{n}+\al^{(r+1)}_{n}) f_{n}(z;r),
	\end{split}
\end{equation}
\begin{equation}\label{L3}
\begin{split}
\mathcal{L}_3[f_n(z;s)] & := f_{n+3}(z;s) -\left(z^2+\varkappa^{(s)}_{n+2}\right)f_{n+2}(z;s) +\left( \varkappa^{(s)}_{n+1}+ \rho^{(s-1)}_{n+1}+\rho^{(s)}_{n+2}+\varkappa^{(s)}_{n+2}\right)z^2 f_{n+1}(z;s) \\ & -\left( \rho^{(s)}_{n+2}+\varkappa^{(s)}_{n+2}\right)\left( \varkappa^{(s)}_{n}+ \rho^{(s-1)}_{n}  \right)z^2 f_{n}(z;s),
\end{split}
\end{equation} and
\begin{equation}\label{L4}
\begin{split}
\mathcal{L}_4[f_n(z;s)] & := f_{n+3}(z;s) -\left( \ga^{(s)}_{n+2}+ \ga^{(s+1)}_{n+1}  \right)f_{n+2}(z;s) - \left(z^{-2}-\ga^{(s+1)}_{n+1} \ga^{(s)}_{n+1}\right) f_{n+1}(z;s) + \left(\theta^{(s+1)}_{n+1}+\ga^{(s+1)}_{n+1} \right) z^{-2} f_{n}(z;s).
\end{split}
\end{equation}
Therefore \eqref{P pure n rec}, \eqref{Q* pure n rec} \eqref{R pure n rec}, and \eqref{S* pure n rec} respectively read 
\begin{equation}\label{LPLQLRLS}
	\mathcal{L}_1[P_n(z;r)]=0, \qandq \mathcal{L}_2[Q_n(z^{-2};r)]=0, \qandq \mathcal{L}_3[R_n(z^{2};s)]=0, \qandq \mathcal{L}_4[S_n(z^{-1};s)]=0.
\end{equation}

\begin{theorem}\label{THM AF rec}
For $n \in \N$, in addition to \eqref{LPLQLRLS}, It holds that 
	
\noindent\begin{minipage}{.5\linewidth}
		\begin{alignat}{2}
		&\mathcal{L}_1[\widecheck{P}_n(z;r)] &&=\mathcal{L}_1[\widehat{P}_n(z;r)]=0, \label{AF P rec rel} \\ 
		&\mathcal{L}_2[\widecheck{Q}_n(z;r)] &&=\mathcal{L}_2[\widehat{Q}_n(z;r)]=0, \label{AF Q rec rel} 
		\end{alignat}	
	\end{minipage}	
\noindent\begin{minipage}{.5\linewidth}
	\begin{alignat}{2}
	&\mathcal{L}_3[\widecheck{R}_n(z;s)] &&=\mathcal{L}_3[\widehat{R}_n(z;s)]=0, \label{AF R rec rel} \\ 
	&\mathcal{L}_4[\widecheck{S}_n(z;s)] &&=\mathcal{L}_4[\widehat{S}_n(z;s)]=0. \label{AF S rec rel} 
	\end{alignat}	
\end{minipage}	
\end{theorem}
\begin{proof}
We have
	\begin{equation}
	\begin{split}
	& \mathcal{L}_1[\widecheck{P}_n(z;r)]  = \int_{\T} \frac{\dd \ze}{2 \pi \ic \ze} w(\ze) \ze^{-r} \frac{\ze^2+z^2}{\ze^2-z^2} \bigg( P_{n+3}(\ze;r) -(\de_{n+2}^{(r)}+\de_{n+1}^{(r-1)})P_{n+2}(\ze;r) + (\de_{n+1}^{(r-1)}\de_{n+1}^{(r)} - z^2)P_{n+1}(\ze;r) \\ & + (\de_n^{(r)}+\eta^{(r-2)}_n)z^2P_{n}(\ze;r)  -P_{n+3}(z;r) +(\de_{n+2}^{(r)}+\de_{n+1}^{(r-1)})P_{n+2}(z;r) - (\de_{n+1}^{(r-1)}\de_{n+1}^{(r)} - z^2)P_{n+1}(z;r)- (\de_n^{(r)}+\eta^{(r-2)}_n)z^2P_{n}(z;r) \bigg).
	\end{split}
	\end{equation}
Writing $z^2 = \ze^2 + (z^2-\ze^2)$  and using \eqref{P pure n rec} for polynomials in $z$ and $\ze$ we obtain
	\begin{equation}\label{LN}
	\begin{split}
	\mathcal{L}_1[\widecheck{P}_n(z;r)] & = \int_{\T} \frac{\dd \ze}{2 \pi \ic \ze} w(\ze) \ze^{-r} \frac{\ze^2+z^2}{\ze^2-z^2} \bigg( -(z^2-\ze^2) P_{n+1}(\ze;r)+(z^2-\ze^2)(\de_n^{(r)}+\eta^{(r-2)}_n)P_n(\ze;r) \bigg) \\ & =  \int_{\T} \frac{\dd \ze}{2 \pi \ic \ze} w(\ze) \ze^{-r} (\ze^2+z^2) \bigg(  P_{n+1}(\ze;r)-(\de_n^{(r)}+\eta^{(r-2)}_n)P_n(\ze;r) \bigg) \\ & = \int_{\T} \frac{\dd \ze}{2 \pi \ic \ze} w(\ze) \ze^{2-r} P_{n+1}(\ze;r) -(\de_n^{(r)}+\eta^{(r-2)}_n) \int_{\T} \frac{\dd \ze}{2 \pi \ic \ze} w(\ze) \ze^{2-r} P_n(\ze;r) \\ & = (-1)^n \left(- \frac{D^{(r-2)}_{n+2}}{D^{(r)}_{n+1}}- (\de_n^{(r)}+\eta^{(r-2)}_n) \frac{D^{(r-2)}_{n+1}}{D^{(r)}_{n}} \right),
	\end{split}
	\end{equation}
where we have used 
	\begin{equation}
	\int_{\T} \frac{\dd \ze}{2 \pi \ic \ze} w(\ze) \ze^{2-r} P_n(\ze;r) = (-1)^n\frac{D^{(r-2)}_{n+1}}{D^{(r)}_{n}},
	\end{equation}
which can be seen from \eqref{OP11}. Indeed, using the determinantal representation of $P_n(z;r)$ we can write
	\begin{equation}
	\begin{split}
	\int_{\T} \frac{\dd \ze}{2 \pi \ic \ze} w(\ze) \ze^{2-r} P_n(\ze;r) & = \frac{1}{D_{n}^{(r)}} \det \begin{pmatrix}
	w_{r} & w_{r-1} &  \cdots & w_{r-n} \\
	w_{r+2} & w_{r+1}  & \cdots & w_{r-n+2} \\
	\vdots & \vdots &  \vdots & \vdots \\
	w_{r+2n-2} & w_{r+2n-3} &  \cdots & w_{r+n-2} \\
	\int_{\T} \frac{\dd \ze}{2 \pi \ic \ze} w(\ze) \ze^{2-r} & \int_{\T} \frac{\dd \ze}{2 \pi \ic \ze} w(\ze) \ze^{3-r} & \cdots  & \int_{\T} \frac{\dd \ze}{2 \pi \ic \ze} w(\ze) \ze^{n+2-r}
	\end{pmatrix} \\ & = \frac{1}{D_{n}^{(r)}} \det \begin{pmatrix}
	w_{r} & w_{r-1}  & \cdots & w_{r-n} \\
	w_{r+2} & w_{r+1}  & \cdots & w_{r-n+2} \\
	\vdots & \vdots  & \vdots & \vdots \\
	w_{r+2n-2} & w_{r+2n-3} &  \cdots & w_{r+n-2} \\
	w_{r-2} & w_{r-3}  & \cdots  & w_{r-n-2}
	\end{pmatrix} =  (-1)^n\frac{D_{n+1}^{(r-2)}}{D_{n}^{(r)}}.
	\end{split}
	\end{equation}
Now we show that $\mathcal{L}_1[\widecheck{P}_n(z;r)]=0$. Recalling \eqref{delta},\eqref{eta}, and \eqref{h} we can rewrite \eqref{LN} after simplifications as
	\begin{equation}\label{LN1}
	\mathcal{L}_1[\widecheck{P}_n(z;r)] =	 \frac{(-1)^{n+1}}{D^{(r)}_{n+1}D^{(r-1)}_{n}} \left(D^{(r-2)}_{n+2}D^{(r-1)}_{n} - D^{(r-2)}_{n+1}D^{(r-1)}_{n+1} + D^{(r-3)}_{n+1}D^{(r)}_{n+1}  \right). 
	\end{equation}
Now consider the following Dodgson condensation identity:
	\begin{equation}
	\begin{split}
	\boldsymbol{\mathscr{D}}_{r-2} \left\lbrace \begin{matrix}  n+2 \\  n+2 \end{matrix} \right\rbrace 	\boldsymbol{\mathscr{D}}_{r-2} \left\lbrace \begin{matrix}  0 & n+1 & n+2 \\  0 & n+1 & n+2 \end{matrix} \right\rbrace  & = 	\boldsymbol{\mathscr{D}}_{r-2} \left\lbrace \begin{matrix}  0 & n+2 \\ 0 & n+2 \end{matrix} \right\rbrace 	\boldsymbol{\mathscr{D}}_{r-2} \left\lbrace \begin{matrix}   n+1 & n+2 \\  n+1 & n+2 \end{matrix} \right\rbrace  \\ & - \boldsymbol{\mathscr{D}}_{r-2} \left\lbrace \begin{matrix}  n+1 & n+2 \\ 0 & n+2 \end{matrix} \right\rbrace 	\boldsymbol{\mathscr{D}}_{r-2} \left\lbrace \begin{matrix}  0 & n+2 \\  n+1 & n+2 \end{matrix} \right\rbrace,
	\end{split}
	\end{equation}
which can be written as $$ D^{(r-2)}_{n+2}D^{(r-1)}_{n} =   D^{(r-2)}_{n+1}D^{(r-1)}_{n+1} - D^{(r-3)}_{n+1}D^{(r)}_{n+1}.$$Combining this with \eqref{LN1} yields $\mathcal{L}_1[\widecheck{P}_n(z;r)] = 0$. It is obvious that the same argument, in particular, shows that $	\mathcal{L}_1[\widehat{P}_n(z;r)] = 0$. The proof of \eqref{AF Q rec rel}, \eqref{AF R rec rel}, and \eqref{AF S rec rel} follow from similar considerations and we do not provide the details here. We just point out that the proof of \eqref{AF R rec rel} is achieved much more easily if one first proves that \[  \int_{\T} \frac{\dd \ze}{2 \pi \ic \ze} w(\ze) \ze^{-s} \frac{1}{\ze-z}R_n(\ze^{2};s) \] is annihilated by \eqref{L3}, and then deduce \eqref{AF R rec rel} simply in view of \eqref{relations}. Following this path one needs to show that 
	\[ \int_{\T} \frac{\dd\mu(\ze;s)}{2 \pi \ic \ze} (\ze+z) \left[ R_{n+2}(\ze^{2};s) - \left( \varkappa^{(s)}_{n+1}+ \rho^{(s-1)}_{n+1}+\rho^{(s)}_{n+2}+\varkappa^{(s)}_{n+2}\right) R_{n+1}(\ze^{2};s) + \left( \rho^{(s)}_{n+2}+\varkappa^{(s)}_{n+2}\right)\left( \varkappa^{(s)}_{n}+ \rho^{(s-1)}_{n}  \right) R_{n}(\ze^{2};s) \right],  \]
vanishes ($\dd\mu(\ze;s)=w(\ze)\ze^{-s}\dd \ze$), while in the "direct" proof to show that \eqref{Rhat} is annihilated by \eqref{L3}, one is required to show that the expression above is zero when we replace the $(\ze+z)$ by $(\ze+z)^2$. 
This is not as easy, since the presence of the $\ze^2$ in the integrand requires us to deal with \textit{bordered} $j-2k$ determinants.
\end{proof} 
One is expected to find three linearly independent solutions to the third order linear difference equations $\mathcal{L}_i[f_n(z;r)]=0$, $i=1,2,3,4$. The following corollary gives a natural fundamental set of solutions for these difference equations, which can be immediately concluded from Theorem \ref{THM AF rec} due to presence of $z^2$, as opposed to odd functions of $z$, in the linear difference operators \eqref{L1} through \eqref{L4}.
\begin{corollary}
	A fundamental set of solutions for the recurrence relations \eqref{L1}, \eqref{L2}, \eqref{L3} and \eqref{L4} is respectively given by $\left\{ P_n(z;r), P_n(-z;r), \widehat{P}_n(z;r)\right\}$, $\left\{Q_n(z^{-2};r), \widehat{Q}_n(z;r), \widehat{Q}_n(-z;r)\right\}$, $\left\{R_n(z^2;s), \widehat{R}_n(z;s), \widehat{R}_n(-z;s)\right\}$ and $\left\{S_n(z^{-1};s), S_n(-z^{-1};s), \widehat{S}_n(z;s)\right\}$, where $\widehat{P}_n(z;r)$, $\widehat{Q}_n(z;r)$, $\widehat{R}_n(z;s)$ and $\widehat{S}_n(z;r)$ are given by \eqref{Phat}, \eqref{Qhat}, \eqref{Rhat}, and \eqref{Shat}.
	
\end{corollary}

In Theorem \ref{Thm Mult Int Polys}, Eqs.\eqref{MultIntPolysP} through \eqref{MultIntPolysS}, we gave representations of the bi-orthogonal polynomials as averages of characteristic polynomials
and in the following result we give the analogous result for the associated functions as averages of reciprocals of characteristic polynomials.
This also corresponds to formulae Eq.(2.34,35) of \cite{FW_2006} or Eq.(3.3,4) of \cite{W_2009} for the Toeplitz case.
\begin{theorem}
We have the following multiple integral representations for associated functions $\widehat{P}_n(z;r)$, $\widehat{Q}_n(z;r)$, $\widehat{R}_n(z;s)$, and $\widehat{S}_n(z;s)$ defined respectively by \eqref{Phat}, \eqref{Qhat}, \eqref{Rhat}, and \eqref{Shat}:
	
\noindent\begin{minipage}{.5\linewidth}
		\begin{alignat}{2}
		&\widehat{P}_n(z;r) &&= \frac{2z^{-2}}{D^{(r)}_n} \mathcal{D}_{n+1}\left[ \frac{w(\ze)\ze^{-r}}{z^{-2}-\ze^{-2}}\right]-w_r \de_{n0}, \label{MultIntAFP} \\
		& \widehat{Q}_n(z;r)  &&= -\frac{2z}{D^{(r)}_n} \mathcal{D}_{n+1}\left[\frac{w(\ze)\ze^{-r}}{z-\ze}\right]+w_r \de_{n0}, \label{MultIntAFQ}
		\end{alignat}	
	\end{minipage}	
	\begin{minipage}{.5\linewidth}
		\begin{alignat}{2}
		&\widehat{R}_n(z;s) &&= \frac{2z^{-1}}{E^{(s)}_n} \mathcal{E}_{n+1}\left[ \frac{w(\ze)\ze^{-s}}{z^{-1}-\ze^{-1}}\right]-w_s \de_{n0}, \label{MultIntAFR} \\
		& \widehat{S}_n(z;s)  &&= -\frac{2z^{2}}{E^{(s)}_n} \mathcal{E}_{n+1}\left[ \frac{w(\ze)\ze^{-s}}{z^{2}-\ze^{2}}\right]+w_s \de_{n0}. \label{MultIntAFS}
		\end{alignat}	
	\end{minipage}
\end{theorem}
\begin{proof}
We show the proof only for $\widehat{Q}_n(z;r)$ as the proofs for the others are similar. Using \eqref{DD} we have
		\[ \mathcal{D}_{n+1}\left[\frac{w(\ze)\ze^{-r}}{z-\ze}\right] = \frac{1}{(n+1)!} \int_{\T} \frac{\dd \ze_1}{2 \pi \ic \ze_1} \cdots \int_{\T} \frac{\dd \ze_{n+1}}{2 \pi \ic \ze_{n+1}} \prod_{j=1}^{n+1} \frac{w(\ze_j)\ze^{-r}_j}{z-\ze_j} \prod_{1\leq j<k\leq n+1} (\ze_k-\ze_j)(\ze^{-2}_k-\ze^{-2}_j). \]
Using the following partial fraction decomposition
		\[\prod_{j=1}^{n+1} \frac{1}{z - \ze_j} = (-1)^n \sum_{m=1}^{n+1} \frac{1}{(z - \ze_m)\underset{\substack{1 \leq \ell \leq n+1 \\ \ell \neq m}}{\prod} (\ze_{\ell}-\ze_m)}, \]
we can write 
		\begin{equation}
			\mathcal{D}_{n+1}\left[\frac{w(\ze)\ze^{-r}}{z-\ze}\right] = \frac{(-1)^n}{(n+1)!} \sum_{m=1}^{n+1} \int_{\T} \frac{\dd \ze_1}{2 \pi \ic \ze_1} \cdots \int_{\T} \frac{\dd \ze_{n+1}}{2 \pi \ic \ze_{n+1}} \prod_{j=1}^{n+1} w(\ze_j)\ze^{-r}_j \frac{ \underset{1\leq j<k\leq n+1}{\prod} (\ze_k-\ze_j)(\ze^{-2}_k-\ze^{-2}_j)}{(z - \ze_m)\underset{\substack{1 \leq \ell \leq n+1 \\ \ell \neq m}}{\prod} (\ze_{\ell}-\ze_m)}.
		\end{equation}
By symmetry, for each choice of $m$ we have the same object, so we set $m = n+1$ and write
		\begin{equation}
		\mathcal{D}_{n+1}\left[\frac{w(\ze)\ze^{-r}}{z-\ze}\right] = \frac{(-1)^n}{n!} \int_{\T} \frac{\dd \ze_1}{2 \pi \ic \ze_1} \cdots \int_{\T} \frac{\dd \ze_{n+1}}{2 \pi \ic \ze_{n+1}} \prod_{j=1}^{n+1} w(\ze_j)\ze^{-r}_j \frac{\underset{1\leq j<k\leq n+1}{\prod} (\ze_k-\ze_j)(\ze^{-2}_k-\ze^{-2}_j)}{(z - \ze_{n+1})\underset{\substack{1 \leq \ell \leq n}}{\prod} (\ze_{\ell}-\ze_{n+1})}.
		\end{equation}
Notice that
		\begin{equation}
			\frac{\underset{1\leq j<k\leq n+1}{\prod} (\ze_k-\ze_j)}{\underset{\substack{1 \leq \ell \leq n}}{\prod} (\ze_{\ell}-\ze_{n+1})} = \frac{\underset{1\leq j<k\leq n}{\prod} (\ze_k-\ze_j)\underset{1\leq j\leq n}{\prod} (\ze_{n+1}-\ze_j)}{\underset{\substack{1 \leq \ell \leq n}}{\prod} (\ze_{\ell}-\ze_{n+1})} = (-1)^n \underset{1\leq j<k\leq n}{\prod} (\ze_k-\ze_j).
		\end{equation}
Therefore
			\begin{equation}
			\begin{split}
			\mathcal{D}_{n+1}\left[\frac{w(\ze)\ze^{-r}}{z-\ze}\right] & = \frac{1}{n!} \int_{\T} \frac{\dd \ze_1}{2 \pi \ic \ze_1} \cdots \int_{\T} \frac{\dd \ze_{n+1}}{2 \pi \ic \ze_{n+1}} \prod_{j=1}^{n+1} w(\ze_j)\ze^{-r}_j \frac{\underset{1\leq j<k\leq n}{\prod} (\ze_k-\ze_j) \underset{1\leq j<k\leq n+1}{\prod} (\ze^{-2}_k-\ze^{-2}_j)}{(z - \ze_{n+1})} \\ & = \frac{1}{n!} \int_{\T} \frac{\dd \ze_1}{2 \pi \ic \ze_1} \cdots \int_{\T} \frac{\dd \ze_{n+1}}{2 \pi \ic \ze_{n+1}} \frac{w(\ze_{n+1})\ze^{-r}_{n+1}}{z - \ze_{n+1}} \prod_{j=1}^{n} w(\ze_j)\ze^{-r}_j (\ze^{-2}_{n+1}-\ze^{-2}_j) \underset{1\leq j<k\leq n}{\prod} (\ze_k-\ze_j) (\ze^{-2}_k-\ze^{-2}_j).
			\end{split}
		\end{equation}
Now, recalling \eqref{DD} and \eqref{MultIntPolysQ} we observe that
		\begin{equation}
		\begin{split}
			\frac{1}{n!} \int_{\T} \frac{\dd \ze_1}{2 \pi \ic \ze_1} \cdots \int_{\T} \frac{\dd \ze_{n}}{2 \pi \ic \ze_{n}}  \prod_{j=1}^{n} w(\ze_j)\ze^{-r}_j (\ze^{-2}_{n+1}-\ze^{-2}_j) \underset{1\leq j<k\leq n}{\prod} (\ze_k-\ze_j) (\ze^{-2}_k-\ze^{-2}_j) & = \mathcal{D}_n[w(\ze)\ze^{-r}(\ze^{-2}_{n+1}-\ze^{-2})] \\ & = D^{(r)}_nQ_n(\ze^{-2}_{n+1};r). 
		\end{split}
		\end{equation}
		Hence
			\begin{equation}
		\mathcal{D}_{n+1}\left[\frac{w(\ze)\ze^{-r}}{z-\ze}\right] = D^{(r)}_n \int_{\T} \frac{\dd \ze}{2 \pi \ic \ze} \frac{w(\ze)\ze^{-r}}{z - \ze}Q_n(\ze^{-2};r).
		\end{equation}
Now \eqref{MultIntAFQ} follows by recalling \eqref{Qhat} and noticing
	  \[ \frac{1}{z-\ze} = \frac{1}{2z} \left( \frac{z+\ze}{z-\ze}+1 \right).\]
\end{proof}

A key step in understanding the recurrence structures in the $2j-k$ and $j-2k$ systems is their formulation as a first order $3\times 3$ matrix difference equation.
The matrix variable so defined would then be come a central object in the integrable system and possibly related to the solution of a rank 3 Riemann-Hilbert problem.
One choice for the $2\times 2$ matrix variable in the Toeplitz case was given by Eq.(2.69) along with its matrix recurrence relation Eq.(2.70) 
(other choices were given in Eqs.(2.73,76)) of \cite{FW_2006}.
\begin{definition} \normalfont
The \textit{Casorati} matrices (see e.g. \cite{Elaydi}) associated with the recurrence relations \eqref{L1} through \eqref{L4} are defined as  

\noindent\begin{minipage}{.5\linewidth}
	\begin{alignat}{2}
	&\boldsymbol{\mathfrak{P}}_n(z;r) &&= 
	\begin{pmatrix}
	P_n(z;r) & P_n(-z;r) & \widehat{P}_n(z;r) \\[5pt]
	P_{n+1}(z;r) & P_{n+1}(-z;r) & \widehat{P}_{n+1}(z;r) \\[5pt]
	P_{n+2}(z;r) & P_{n+2}(-z;r) & \widehat{P}_{n+2}(z;r) \\		
	\end{pmatrix}, \label{Cas P} \\
	&\boldsymbol{\mathfrak{Q}}_n(z;r) &&= 
	\begin{pmatrix}
Q_n(z^{-2};r) & \widehat{Q}_n(z;r) & \widehat{Q}_n(-z;r) \\[5pt]
Q_{n+1}(z^{-2};r) & \widehat{Q}_{n+1}(z;r) & \widehat{Q}_{n+1}(-z;r) \\[5pt]
Q_{n+2}(z^{-2};r) & \widehat{Q}_{n+2}(z;r) & \widehat{Q}_{n+2}(-z;r) \\		
\end{pmatrix}, \label{Cas Q}
	\end{alignat}	
\end{minipage}	
\begin{minipage}{.5\linewidth}
	\begin{alignat}{2}
&\boldsymbol{\mathfrak{R}}_n(z;s) &&=\begin{pmatrix}
R_n(z^2;s) & \widehat{R}_n(z;s) & \widehat{R}_n(-z;s) \\[5pt]
R_{n+1}(z^2;s) & \widehat{R}_{n+1}(z;s) & \widehat{R}_{n+1}(-z;s) \\[5pt]
R_{n+2}(z^2;s) & \widehat{R}_{n+2}(z;s) & \widehat{R}_{n+2}(-z;s) \\		
\end{pmatrix}, \label{Cas R} \\ 	&\boldsymbol{\mathfrak{S}}_n(z;s) &&= 
 \begin{pmatrix}
 S_n(z^{-1};s) & S_n(-z^{-1};s) & \widehat{S}_n(z;s) \\[5pt]
 S_{n+1}(z^{-1};s) & S_{n+1}(-z^{-1};s) & \widehat{S}_{n+1}(z;s) \\[5pt]
 S_{n+2}(z^{-1};s) & S_{n+2}(-z^{-1};s) & \widehat{S}_{n+2}(z;s) \\		
 \end{pmatrix}. \label{Cas S}
	\end{alignat}	
\end{minipage}
\end{definition}

\begin{lemma}
The Casorati matrices satisfy the following first order recurrence relations
	\begin{equation}
			\boldsymbol{\mathfrak{P}}_{n+1}(z;r) = 
		\begin{pmatrix}
		0 & 1 & 0 \\
		0 & 0 & 1 \\
		-(\de_n^{(r)}+\eta^{(r-2)}_n)z^2 &   z^2-\de_{n+1}^{(r-1)}\de_{n+1}^{(r)} & \de_{n+2}^{(r)}+\de_{n+1}^{(r-1)} \\		
		\end{pmatrix} \boldsymbol{\mathfrak{P}}_{n}(z;r),
	\end{equation}
	\begin{equation}
		\boldsymbol{\mathfrak{Q}}_{n+1}(z;r) = 
		\begin{pmatrix}
		0 & 1 & 0 \\
		0 & 0 & 1 \\
		(\be^{(r+1)}_{n+1}+\al^{(r+2)}_{n+1})(\be^{(r)}_{n}+\al^{(r+1)}_{n})z^{-2} & -(\be^{(r)}_{n+1}+\al^{(r+1)}_{n+1}+\be^{(r+1)}_{n+1}+\al^{(r+2)}_{n+1})z^{-2} & z^{-2}+\be_{n+2}^{(r)} \\		
		\end{pmatrix} \boldsymbol{\mathfrak{Q}}_{n}(z;r),
\end{equation}
\begin{equation} 
	\boldsymbol{\mathfrak{R}}_{n+1}(z;r) = 
	\begin{pmatrix}
	0 & 1 & 0 \\
	0 & 0 & 1 \\
	\left( \rho^{(s)}_{n+2}+\varkappa^{(s)}_{n+2} \right)\left( \varkappa^{(s)}_{n}+ \rho^{(s-1)}_{n}  \right)z^2 & -\left( \varkappa^{(s)}_{n+1}+ \rho^{(s-1)}_{n+1}+\rho^{(s)}_{n+2}+\varkappa^{(s)}_{n+2}\right)z^2 & z^2+\varkappa^{(s)}_{n+2} \\		
	\end{pmatrix} \boldsymbol{\mathfrak{R}}_{n}(z;r),
\end{equation}
\begin{equation}
	\boldsymbol{\mathfrak{S}}_{n+1}(z;r) = 
	\begin{pmatrix}
	0 & 1 & 0 \\
	0 & 0 & 1 \\
	-\left(\theta^{(s+1)}_{n+1}+\ga^{(s+1)}_{n+1} \right) z^{-2} & z^{-2}- \ga^{(s+1)}_{n+1}\ga^{(s)}_{n+1}  & \ga^{(s)}_{n+2}+ \ga^{(s+1)}_{n+1}   \\		
	\end{pmatrix} \boldsymbol{\mathfrak{S}}_{n}(z;r).
\end{equation}
\end{lemma}

A necessary step in establishing the existence of a valid $3\times 3$ matrix system is to compute their Casoratians and to demonstrate their boundedness and non-vanishing character for $ z\neq 0,\infty$.

\begin{lemma}\label{Abel}[Abel's Lemma, \cite{Elaydi}]
The Casoratians for the matrix systems \eqref{Cas P} through \eqref{Cas R}  are respectively given by
	
\noindent\begin{minipage}{.5\linewidth}
	\begin{alignat}{2}
	&\frac{\det \boldsymbol{\mathfrak{P}}_{n}(z;r)}{\det \boldsymbol{\mathfrak{P}}_{0}(z;r)}  &&= (-1)^nz^{2n}\prod_{\ell=0}^{n-1} \left(\de_{\ell}^{(r)}+\eta^{(r-2)}_{\ell}\right), \label{Cas Pn / Cas P0} \\
	&\frac{\det \boldsymbol{\mathfrak{Q}}_{n}(z;r)}{\det \boldsymbol{\mathfrak{Q}}_{0}(z;r)}  &&=z^{-2n} \prod_{\ell=0}^{n-1} \left(\be^{(r+1)}_{\ell+1}+\al^{(r+2)}_{\ell+1}\right)\left(\be^{(r)}_{\ell}+\al^{(r+1)}_{\ell}\right), \label{Cas Qn / Cas Q0}
	\end{alignat}	
\end{minipage}	
\begin{minipage}{.5\linewidth}
	\begin{alignat}{2}
	&		\frac{\det \boldsymbol{\mathfrak{R}}_{n}(z;s)}{\det \boldsymbol{\mathfrak{R}}_{0}(z;s)}  &&= z^{2n}\prod_{\ell=0}^{n-1} \left( \rho^{(s)}_{\ell+2}+\varkappa^{(s)}_{\ell+2}\right)\left( \varkappa^{(s)}_{\ell}+ \rho^{(s-1)}_{\ell}  \right), \label{Cas Rn / Cas R0} \\
	&		\frac{\det \boldsymbol{\mathfrak{S}}_{n}(z;s)}{\det \boldsymbol{\mathfrak{S}}_{0}(z;s)}  &&= (-1)^nz^{-2n}\prod_{\ell=0}^{n-1} \left(\theta^{(s+1)}_{\ell+1}+\ga^{(s+1)}_{\ell+1} \right). \label{Cas Sn / Cas S0}
	\end{alignat}	
\end{minipage}	
\end{lemma}
\begin{proof}
Since the Casorati matrix $\boldsymbol{\mathfrak{P}}_n(z;r)$ satisfies the recurrence relation \eqref{L1}, 
its determinant satisfies the recurrence relation 
	\begin{equation}\label{Casoratian recurrence relation}
		\det \boldsymbol{\mathfrak{P}}_{n+1}(z;r)+(\de_n^{(r)}+\eta^{(r-2)}_n)z^2\det\boldsymbol{\mathfrak{P}}_n(z;r)=0,
	\end{equation}
which follows from rewriting the objects with index $n+3$ in the last row of $\det \mathfrak{P}_{n+1}(z;r)$ in terms of the objects with indices $n$, $n+1,$ and $n+2$ using \eqref{L1} and performing row operations. 
Using \eqref{Casoratian recurrence relation} we can write
	\begin{equation}\label{cccc}
		\frac{\det \boldsymbol{\mathfrak{P}}_{n}(z;r)}{\det \boldsymbol{\mathfrak{P}}_{0}(z;r)} = (-1)^nz^{2n}\prod_{\ell=0}^{n-1} \left(\de_{\ell}^{(r)}+\eta^{(r-2)}_{\ell}\right).
	\end{equation}
The formulas \eqref{Cas Qn / Cas Q0}, \eqref{Cas Rn / Cas R0}, and \eqref{Cas Sn / Cas S0} can be established similarly.
\end{proof}
\begin{lemma}\label{Lem Cas 0}
The Casoratians with index zero appearing in Lemma \ref{Abel} can be represented in terms of moments of the weight $w$ and the function $F_2$ as follows 
\begin{equation}\label{det frak P 0}
	\det \boldsymbol{\mathfrak{P}}_{0}(z;r) = -2z\left(w_{r-2}+w_r z^2\right),
\end{equation}
\begin{equation}
	\det \boldsymbol{\mathfrak{Q}}_{0}(z;r) = \det 
	\begin{pmatrix}
	1 & z^{-2} & z^{-4} \\ 
	F_2(z;r) & F_2(z;r+2) & F_2(z;r+4) \\
	F_2(-z;r) & F_2(-z;r+2) & F_2(-z;r+4)
	\end{pmatrix},
\end{equation}
\begin{equation}
	\det \boldsymbol{\mathfrak{R}}_{0}(z;r) = \det 
	\begin{pmatrix}
	1 & z^{2}& z^{4} \\ 
	F_2(z;s) & F_2(z;s-2) & F_2(z;s-4) \\
	F_2(-z;s) & F_2(-z;s-2) & F_2                                                                                                                                                                                                                                                                                                                                                                                                                                                                                                                                                                                                                                                                                                                                                                                                                                                                                                                                                                                                                                                                                                                                                                                                                                                                                                                                                                                                                                                                                                                                                                                                                                                                                                                                                                                                                                                                                                                                                                                                                                                                                                                                                                                                                                                                                                                                                                                                                                                                                                                                                                                                                                                                                                                                                                                                                                                                                                                                                                                                                                                                                                                                                                                                                                                                                                                                                                                                                                                                                                                                                                                                                                                                                                                                                                              (-z;s-4) 
	\end{pmatrix},
\end{equation}
and
\begin{equation}
\det \boldsymbol{\mathfrak{S}}_{0}(z;r) =2z^{-3}\left(w_s+z^2w_{s+2}\right).
\end{equation}
\end{lemma}
\begin{proof}
Due to similarity we only show the proof for \eqref{det frak P 0}. Notice that 
	\begin{equation}
	\det \boldsymbol{\mathfrak{P}}_{0}(z;r) = \det 
	\begin{pmatrix}
	1 & 1 & F_1(z;r) \\
	z+P_1(0;r) & -z+P_1(0;r) & \widehat{P}_{1}(z;r) \\
	z^2+\mathcal{p}^{(r)}_{2,1}z+P_2(0;r) & z^2-\mathcal{p}^{(r)}_{2,1}z+P_2(0;r) & \widehat{P}_{2}(z;r) \\		
	\end{pmatrix}.
	\end{equation}
After performing elementary row operations we obtain
	\begin{equation}
	\det \boldsymbol{\mathfrak{P}}_{0}(z;r) = \det 
	\begin{pmatrix}
	1 & 0 & F_1(z;r) \\
	z & -2z & \widehat{P}_{1}(z;r) - P_1(0;r)F_1(z;r) \\
	z^2 & 0 & \widehat{P}_{2}(z;r) -P_2(0;r)F_1(z;r) - \mathcal{p}^{(r)}_{2,1}\left(\widehat{P}_{1}(z;r) - P_1(0;r)F_1(z;r)\right) \\		
	\end{pmatrix},
	\end{equation}
and thus 
	\begin{equation}\label{aaaa}
	\det \boldsymbol{\mathfrak{P}}_{0}(z;r) = -2z \left( \widehat{P}_{2}(z;r) - \mathcal{p}^{(r)}_{2,1}\widehat{P}_{1}(z;r) + \left( \mathcal{p}^{(r)}_{2,1}P_1(0;r)  -P_2(0;r) -z^2 \right)F_1(z;r) \right).
	\end{equation}
Notice that 
	\begin{equation}\label{bbbb}
	\begin{split}
	& \widehat{P}_{2}(z;r) - \mathcal{p}^{(r)}_{2,1}\widehat{P}_{1}(z;r) + \left( \mathcal{p}^{(r)}_{2,1}P_1(0;r)  -P_2(0;r) -z^2 \right)F_1(z;r)  \\ & =  \int_{\T} \frac{\dd \ze}{2 \pi \ic \ze} w(\ze) \ze^{-r} \frac{\ze^2+z^2}{\ze^2-z^2} \left\{ P_2(\ze;r) - \mathcal{p}^{(r)}_{2,1}P_{1}(\ze;r) + \mathcal{p}^{(r)}_{2,1}P_1(0;r)  -P_2(0;r) -z^2 \right\} \\ & = \int_{\T} \frac{\dd \ze}{2 \pi \ic \ze} w(\ze) \ze^{-r} \frac{\ze^2+z^2}{\ze^2-z^2} \left\{ \ze^2 -z^2 \right\} = w_{r-2}+w_r z^2.
	\end{split}
	\end{equation}
\end{proof}

\section{The symplectic and orthogonal example}\label{Sec Exp weight}
Now we return to the specific case of $2j-k$ determinants which appeared in \cite{ABP+_2014}. 
As we discussed in \S \ref{Sec Intro}, it is shown in \cite{ABP+_2014} that the leading constant $b_{\ell}$ in the large-$N$ asymptotic expansion of the $2\ell$-th moment of $|\La'_A(1)|$ (associated with $\tn{USp}(2N)$,  $\tn{SO}(2N),$ and $\tn{O}^-(2N)$) can be expressed in terms of evaluations of the $\ell \times \ell$, $2j-k$ determinant associated to the symbol  
\begin{equation}
	w(z;u) = \exp\left( z + u z^{-2} \right).
\end{equation}
Analysing this particular determinant in full details will be carried out in our future work, but in this section as a concrete first example, we are going to present explicit formulae for the determinants,  orthogonal polynomials, and recurrence coefficients corresponding to the undeformed weight
\begin{equation}\label{weight exp}
	w(z;0) \equiv w(z) = e^{z}.
\end{equation}
Notice that for this weight
\begin{equation}\label{Det exp}
	D^{(r)}_n = \underset{0\leq j,k \leq n-1}{\det}\left( w_{r+2j-k} \right) = \underset{0\leq j,k \leq n-1}{\det}\left( \frac{1}{\Ga(1+r+2j-k)} \right).
\end{equation}
The entries of the determinant exist if $r \in \Z_{\geq 0}$. Let us now recall a result in  \cite{Normand} which is all we need to evaluate $D^{(r)}_n$ explicitly. 
The equation $(4.13)$ in \cite{Normand} reads
\begin{equation}\label{Normand 1}
	 \underset{0\leq j,k \leq n-1}{\det}\left( \frac{1}{\Ga(z_j-k)} \right)
	 = \frac{\underset{0\leq j<k\leq n-1}{\prod} (z_k-z_j)}{\prod_{j=0}^{n-1} \Ga(z_j)}.
\end{equation}
Let $z_j:= 1+r+2j$. From \eqref{Normand 1} we have 
\begin{equation}\label{D Gamma}
	D^{(r)}_n = \frac{\underset{0\leq j<k\leq n-1}{\prod} 2(k-j)}{\prod_{j=0}^{n-1} \Ga(1+r+2j)}
	 = \frac{2^{\frac{n(n-1)}{2}}\prod_{j=1}^{n-1} \Ga(1+j)}{\prod_{j=0}^{n-1} \Ga(1+r+2j)}.
\end{equation}
Using the functional equation for the Barnes G-function $G(z+1)=\Ga(z)G(z)$,
we can write $\prod_{j=1}^{n-1} \Ga(1+j) = G(n+1)$. Also using the identity $\Ga(2z)=\pi^{-\frac{1}{2}}2^{2z-1}\Ga(z)\Ga(z+\tfrac{1}{2})$, we write 
\begin{equation}
	\prod_{j=0}^{n-1} \Ga(1+r+2j) = \pi^{-\frac{n}{2}} 2^{n(n+r-1)} \frac{G\left( n + \frac{r+1}{2} \right)G\left( n + \frac{r+2}{2} \right)}{G\left(\frac{r+1}{2} \right)G\left( \frac{r+2}{2} \right)}.
\end{equation}
Therefore after simplifications \eqref{D Gamma} can be expressed as
\begin{equation}\label{D BarnesG}
	D^{(r)}_n = \frac{(2\pi)^{\frac{n}{2}}G(\frac{r+1}{2})G(\frac{r+2}{2})}{2^{\frac{n}{2}(2r+n)}}\frac{G(n+1)}{G(n+\frac{r+1}{2})G(n+\frac{r+2}{2})}
	= \frac{2^{-(\frac{1}{2}n+r)(n-1)}}{\Ga(1+r)} \prod_{j=1}^{n-1}\frac{\Ga(\frac{1}{2})\Ga(1+j)}{\Ga(\frac{r+1}{2}+j)\Ga(\frac{r}{2}+1+j)} .
\end{equation}
Thus $D^{(r)}_n$ exists and is nonzero for all $n \in \N$ and $r \in \Z_{\geq 0}$ (see \eqref{Det exp} and below) , which ensures that the polynomials $P_n(z;r)$ and $Q_n(z;r)$ exist and are unique for all $n \in \N$ and $r \in \Z_{\geq 0}$ (see the discussion in the beginning of \S \ref{Sec Bordered}). 

Now, we can immediately find the recurrence coefficients and the norm $h^{(r)}_n$ of the bi-orthogonal polynomials. Indeed, from \eqref{h} and \eqref{D BarnesG} we find
\begin{equation}\label{h exp wt}
	h^{(r)}_n = \frac{2^n n!}{(2n+r)!}.
\end{equation}
Recalling \eqref{delta}, \eqref{eta}, \eqref{beta}, \eqref{alpha} and using \eqref{D BarnesG} and \eqref{h exp wt} we find

\noindent\begin{minipage}{.5\linewidth}
	\begin{alignat}{2}
	&\de^{(r)}_{n} &&= -2n-r, \label{delta exp wt} \\
	& \eta^{(r)}_n &&=r, \label{eta exp wt}
	\end{alignat}	
\end{minipage}	
\begin{minipage}{.5\linewidth}
	\begin{alignat}{2}
	&\be^{(r)}_{n} &&= -\frac{1}{(2n+r+1)(2n+r+2)}, \label{beta exp wt} \\
	& \al^{(r)}_n&&= \frac{r}{(2n+r)(2n+r+1)(2n+r+2)}. \label{alpha exp wt}
	\end{alignat}	
\end{minipage}
Notice that, as expected, \eqref{delta exp wt}, \eqref{eta exp wt}, \eqref{beta exp wt}, and \eqref{alpha exp wt} are in agreement with \eqref{2j-k interrels}. Using these expressions for the recurrence coefficients we can write the pure-degree and pure-offset recurrence relations \eqref{P pure n rec}, \eqref{Q* pure n rec}, \eqref{P pure r rec}, and \eqref{Q pure r rec} as

\begin{equation}\label{P pure n rec exp wt}
	P_{n+3}(z;r) +(4n+2r+5)P_{n+2}(z;r) + \left((2n+r+1)(2n+r+2) - z^2\right)P_{n+1}(z;r)- (2n+2)z^2P_{n}(z;r)=0,
\end{equation}
\begin{equation}\label{Q* pure n rec exp wt}
\begin{split}
	 Q^*_{n+3}(z;r) & -  (1-\frac{z}{(2n+r+5)(2n+r+6)} ) Q^*_{n+2}(z;r) -  \frac{(2n+4)(4n+2r+9)}{(2n+r+3)(2n+r+4)(2n+r+5)(2n+r+6)}z Q^*_{n+1}(z;r) \\ & -\frac{4(n+1)(n+2)z^2}{(2n+r+1)(2n+r+2)(2n+r+3)(2n+r+4)(2n+r+5)(2n+r+6)} Q^*_{n}(z;r) = 0,
\end{split}
\end{equation}
\begin{equation}\label{P pure r rec exp wt}
	(r+1) P_{n}(z;r+3) - zP_n(z;r+2) - (2n+r+1)P_n(z;r+1) + zP_n(z;r) =0,
\end{equation}
and
\begin{equation}\label{Q pure r rec exp wt}
Q^*_n(z;r+3) - Q^*_{n}(z;r+2)  -\frac{z}{(2n+r+2)(2n+r+3)}  Q^*_{n}(z;r+1)  +\frac{(r+1)z}{(2n+r+1)(2n+r+2)(2n+r+3)} Q^*_{n}(z;r)=0.
\end{equation}

A straight-forward residue calculation yields the following result about the Carath\'{e}odory function $F_2(z;r)$. Notice that due to \eqref{F1 F2} this gives an evaluation of $F_1(z;r)$ as well.
\begin{lemma}
The Carath\'{e}odory function $F_2(z;r)$ associated to the weight $w(z)=e^z$ evaluates to
\begin{equation}
	F_2(z;r) = 
	\begin{cases}
	2z^{-r}\left[ e^z - \di \sum_{m=0}^{r-1} \frac{z^m}{m!} \right] - \di \frac{1}{r!}, & 0<|z|<1, \\
	-2z^{-r} \di \sum_{m=0}^{r-1} \frac{z^m}{m!} - \di \frac{1}{r!}, & |z|>1.
	\end{cases}
\end{equation}
\end{lemma}
\begin{lemma}
The $2j-k$ polynomials of the first and the second kind associated to the exponential weight \eqref{weight exp} are given by
\begin{equation}
		P_n(z;r) = 2^n \sum_{\ell=0}^n \frac{1}{\ell !} \sum_{m=0}^{\ell} (-1)^{m+n+\ell} \binom{\ell}{m} \left(\frac{r-m}{2}\right)_{n} z^{\ell},
\end{equation}and
\begin{equation}
	Q_n(z;r) = \frac{(-1)^n}{4^n \left(\frac{r+1}{2}\right)_{n}\left(\frac{r+2}{2}\right)_{n}}{}_3F_0(-n, \frac{r+1}{2}, \frac{r+2}{2};;4z),
\end{equation}	
\end{lemma}
\begin{proof}
Let us first prove the formula for $Q_n(z;r)$. Recalling the notations in \eqref{polys}, from \eqref{OP22} we have
\begin{equation}\label{qnl}
	D_{n}^{(r)} \mathcal{q}^{(r)}_{n,\ell}
	= (-1)^{n+\ell}  \det 
	\begin{pmatrix}
	w_{r} & w_{r-1}  & \cdots & w_{r-n+1}  \\
	\vdots & \vdots  & \vdots & \vdots  \\
	w_{r+2\ell-2} & w_{r+2\ell-3} & \cdots & w_{r+2\ell-n-1}  \\
	w_{r+2\ell+2} & w_{r+2\ell+1} & \cdots & w_{r+2\ell-n+3}  \\
	\vdots & \vdots  & \vdots & \vdots  \\
	w_{r+2n} & w_{r+2n-1}  & \cdots & w_{r+n+1}
	\end{pmatrix}
	 = (-1)^{n+\ell}\underset{\substack{0\leq k \leq n-1 \\ 0\leq j \leq n \\ j \neq \ell}}{\det}\left( \frac{1}{\Ga(1+r+2j-k)} \right).
\end{equation}
Again, let $z_j \equiv 1+r+2j$ and use \eqref{Normand 1} to obtain 
\begin{equation}\label{qnl1}
	D_{n}^{(r)} \mathcal{q}^{(r)}_{n,\ell} =  (-1)^{n+\ell} \frac{\underset{\substack{0\leq j<k\leq n \\ j,k\neq \ell}}{\prod} (z_k-z_j)}{\underset{\substack{0\leq j \leq n \\ j\neq \ell}}{\prod} \Ga(z_j)}.
\end{equation}
Notice that 
\begin{equation}
	\underset{\substack{0\leq j<k\leq n \\ j,k\neq \ell}}{\prod} (z_k-z_j) = \left( \prod_{j=0}^{\ell-2}\prod_{k=j+1}^{\ell-1} 2(k-j) \right) \left( \prod_{j=0}^{\ell-1}\prod_{k=\ell+1}^{n} 2(k-j) \right)\left( \prod_{j=\ell+1}^{n-1}\prod_{k=j+1}^{n} 2(k-j) \right).
\end{equation}
Straightforward calculation of each of the double products on the right hand side and simplifications yield
\begin{equation}\label{incomplete vandermonde}
	\underset{\substack{0\leq j<k\leq n \\ j,k\neq \ell}}{\prod} (z_k-z_j)
	 = 2^{n(n-1)/2}\frac{\prod_{j=1}^{n}j!}{\ell!(n-\ell)!}
	 = \frac{2^{\frac{n(n-1)}{2}}G(n+2)}{\Ga(\ell+1)\Ga(n-\ell+1)}.
\end{equation}
By combining  \eqref{D BarnesG}, \eqref{qnl1}, and \eqref{incomplete vandermonde} and performing simplifications we obtain
\begin{equation}\label{qnl2}
	 \mathcal{q}^{(r)}_{n,\ell} =  (-1)^{n+\ell}  \frac{4^{(\ell-n)}n!}{\ell! (n-\ell)!}\frac{\Ga(\ell+\frac{r+1}{2})\Ga(\ell+\frac{r+2}{2})}{\Ga(n+\frac{r+1}{2})\Ga(n+\frac{r+2}{2})}.
\end{equation}
If we write it in terms of rising factorials we get
\begin{equation}\label{qnl3}
	\mathcal{q}^{(r)}_{n,\ell} =  (-1)^{n}4^{(\ell-n)}
		\frac{(-n)_{\ell}\left(\frac{r+1}{2}\right)_{\ell}\left(\frac{r+2}{2}\right)_{\ell}}
		     {\ell!\left(\frac{r+1}{2}\right)_{n}\left(\frac{r+2}{2}\right)_{n}}.
\end{equation}
Combining this with \eqref{polys} yields 
\begin{equation}
	Q_{n}(z;r) = \frac{(-1)^n}{4^n \left(\frac{r+1}{2}\right)_{n}\left(\frac{r+2}{2}\right)_{n}}{}_3F_0(-n, \frac{r+1}{2}, \frac{r+2}{2};;4z).
\end{equation}
For $P_n(z;r)$ it is possible to provide a constructive proof, however it is simpler to justify that 
\begin{equation}\label{Pn exp weight}
\mathcal{p}^{(r)}_{n,\ell}= \frac{2^n}{\ell !} \sum_{m=0}^{\ell} (-1)^{n+\ell+m} \binom{\ell}{m} \left(\frac{r-m}{2}\right)_{n},	
\end{equation}
are the coefficients of the $P$-polynomial (see \eqref{polys}) associated to \eqref{weight exp}. Using the uniqueness of the orthogonal polynomials $P_n(z;r)$ (see \eqref{D BarnesG} and below), if the polynomial $P_n(z;r)$ defined by \eqref{polys} and \eqref{Pn exp weight} satisfies the recurrence relation \eqref{P pure n rec exp wt} it must be the only one, and this is exactly what we want to justify now. To this end, let us first extract a recurrence relation satisfied by the polynomial coefficients $\mathcal{p}^{(r)}_{n,\ell}$ by matching the coefficients of $z^{\ell+2}$ in \eqref{P pure n rec exp wt}:
\begin{equation}\label{Recurrence Poly Coefs}
	\mathcal{p}^{(r)}_{n+1,\ell}+(2n+2)\mathcal{p}^{(r)}_{n,\ell}=\mathcal{p}^{(r)}_{n+3,\ell+2}+(4n+2r+5)\mathcal{p}^{(r)}_{n+2,\ell+2}+(2n+r+1)(2n+r+2)\mathcal{p}^{(r)}_{n+1,\ell+2}.
\end{equation}
Combining \eqref{Pn exp weight} and \eqref{Recurrence Poly Coefs} we obtain
\begin{equation}\label{LHS}
	l.h.s.\eqref{Recurrence Poly Coefs} = (-1)^{n+\ell} \frac{2^n}{\ell !} \sum_{m=0}^{\ell} (-1)^m \binom{\ell}{m} \left[-2\left(\frac{r-m}{2}\right)_{n+1} + (2n+2)\left(\frac{r-m}{2}\right)_{n} \right],
\end{equation} 
and
\begin{equation}\label{RHS}
\begin{split}
r.h.s.\eqref{Recurrence Poly Coefs} & = (-1)^{n+\ell+1} \frac{2^{n+1}}{(\ell+2)!} \\ & \times \sum_{m=0}^{\ell+2} (-1)^m \binom{\ell+2}{m} \left[4\left(\frac{r-m}{2}\right)_{n+3}-2(4n+2r+5)\left(\frac{r-m}{2}\right)_{n+2}+(2n+r+1)(2n+r+2)\left(\frac{r-m}{2}\right)_{n+1}\right],
\end{split}
\end{equation} 
respectively for the left-hand side of \eqref{Recurrence Poly Coefs}, and for its right-hand side. Note that
\begin{equation}
	4\left(\frac{r-m}{2}\right)_{n+3} -2(4n+2r+5)\left(\frac{r-m}{2}\right)_{n+2}+(2n+r+1)(2n+r+2)\left(\frac{r-m}{2}\right)_{n+1} = m(m-1)\left(\frac{r-m}{2}\right)_{n+1},
\end{equation}
by straight-forward simplifications. Therefore \eqref{RHS} can be written as 
\begin{equation}\label{RHS1}
(-1)^{n+\ell+1} \frac{2^{n+1}}{(\ell+2)!} \sum_{m=2}^{\ell+2} (-1)^m \binom{\ell+2}{m}m(m-1)\left(\frac{r-m}{2}\right)_{n+1}.
\end{equation} 
Further simplification shows that \eqref{RHS1}, and thus \eqref{RHS}, can be written as
\begin{equation}\label{LHS RHS reduction}
	(-1)^{n+\ell+1} \frac{2^{n+1}}{\ell !} \sum_{m=0}^{\ell} (-1)^m \binom{\ell}{m} \left(\frac{r-m-2}{2}\right)_{n+1} \equiv \mathcal{p}^{(r-2)}_{n+1,\ell}.
\end{equation}
In a similar way, one can see that \eqref{LHS} also reduces to \eqref{LHS RHS reduction}. This shows that $\mathcal{p}^{(r)}_{n,\ell}$ given by \eqref{Pn exp weight} is the solution of \eqref{Recurrence Poly Coefs}, and thus finishes the proof. 
\end{proof}

\begin{remark} \normalfont
	Unlike $Q_n(z;r)$, the polynomials $P_n(z;r)$ associated with the weight $e^z$ can not be written as hypergeometric functions, however the coefficients $\mathcal{p}^{(r)}_{n,\ell}$ can be written as a difference of two hypergeometric functions. Indeed, by straight forward calculation after splitting the sum \eqref{Pn exp weight} into two parts over even and odd indices $m$, one arrives at
	\begin{equation}
\mathcal{p}^{(r)}_{n,\ell} = (-1)^{n+\ell}2^{n} \left\{ \frac{\left(\frac{r}{2}\right)_n}{\ell!} \pFq[10]{3}{2}{-\frac{\ell}{2},\frac{1-\ell}{2},1-\frac{r}{2}}{\frac{1}{2},1-\frac{r}{2}-n}{1} - \frac{\left(\frac{r-1}{2}\right)_n}{(\ell-1)!} \pFq[10]{3}{2}{\frac{1-\ell}{2},1-\frac{\ell}{2},\frac{3-r}{2}}{\frac{3}{2},\frac{3-r}{2}-n}{1} \right\}.
	\end{equation}
\end{remark}

\begin{remark} \normalfont
	There is yet another representation for the coefficients $\mathcal{p}^{(r)}_{n,\ell}$. Let us consider the backward difference operator $\nabla_t$ defined as
	\begin{equation}
		\nabla_t f(t) := f(t)-f(t-1).
	\end{equation}
	If we compose $\nabla_t$ with itself $n$ times, then it acts on $f$ as
	\begin{equation}
	\nabla^n_t f(t) = \sum_{m=0}^{n}(-1)^m \binom{n}{m}f(t-m), \qquad n=0,1,2,\cdots,
	\end{equation}
	with the convention that $\nabla^0_t$ is the identity operator.  Therefore 
	\begin{equation}
		\nabla^{\ell}_r \left(\frac{r}{2}\right)_n = \sum_{m=0}^{\ell}(-1)^m \binom{\ell}{m} \left(\frac{r-m}{2}\right)_n, \qquad n=0,1,2,\cdots,
	\end{equation}
	and thus
	\begin{equation}\label{pnl diff operator}
		\mathcal{p}^{(r)}_{n,\ell}= \frac{2^n}{\ell !} (-1)^{n+\ell} \nabla^{\ell}_r \left(\frac{r}{2}\right)_n, \qquad n=0,1,2,\cdots.
	\end{equation}
\end{remark}
\begin{remark} \normalfont
	From \eqref{Pn0r} and \eqref{delta exp wt} we can directly compute $\mathcal{p}^{(r)}_{n,0}$ associated to $w(z)=e^z$. Indeed
	\begin{equation}\label{pn0 1}
		\mathcal{p}^{(r)}_{n,0} = (-1)^n \prod^{n-1}_{\ell=0} (2\ell + r) = (-2)^n \prod^{n-1}_{\ell=0} (\frac{r}{2}+\ell) = (-2)^n  \left(\frac{r}{2}\right)_n,
	\end{equation}
    Moreover, from \eqref{delta exp wt} and \eqref{1st rec P}  one can write a non-homogeneous linear recurrence relation of order one for $\mathcal{p}^{(r)}_{n,\ell+1}$, and in particular for $\mathcal{p}^{(r)}_{n,1}$. Solving this recurrence relation (using, e.g., the variation of parameters method) one finds 
    \begin{equation}\label{pn1 1}
		\mathcal{p}^{(r)}_{n,1} = (-2)^n \left\{ \left(\frac{r-1}{2}\right)_n - \left(\frac{r}{2}\right)_n \right\}. 
	\end{equation} Notice that	both \eqref{pn0 1} and\eqref{pn1 1} are in agreement with the general formula \eqref{Pn exp weight} (or equivalently with \eqref{pnl diff operator}).
\end{remark}


\section{Prospects of future work}\label{Sec open Qs}

In our study of $2j-k$ and $j-2k$ bi-orthogonal polynomial systems on the unit circle 
we have concentrated on laying out the theory for generic class of weights and only considered the essential orthogonality structures. Even within this circumscribed area there are other important aspects that are yet to be investigated. Our purpose in this section is to bring the $2j-k$ and $j-2k$ systems to the attention of a wider audience of mathematicians by providing a (non-exhaustive) list of open problems which we believe are significant for the development of the theory of $2j-k$ and $j-2k$ systems, and consequently for the theory of $pj-qk$ systems, with relatively prime $p$ and $q$. One example of this appears to be the novel form of the joint density function for the $pj-qk$ systems which can be rewritten as the product of differences
\begin{equation*}
	\prod_{1\leq j<k \leq n}\left( \zeta^{q}_{k}-\zeta^{q}_{j} \right)\left( \zeta^{-p}_{k}-\zeta^{-p}_{j} \right)
	= 4^{n(n-1)/2} \prod_{j=1}^{n} \zeta^{\frac{1}{2}(q-p)(n-1)}_{j} \prod_{1\leq j<k \leq n} \sin\left( \tfrac{1}{2}q(\theta_k-\theta_j) \right) \sin\left( \tfrac{1}{2}p(\theta_k-\theta_j) \right) ,
\end{equation*}
and features a primary repulsion of $ \beta=2 $ as $ \theta_k-\theta_j \to 0 $ along with weaker repulsions of $ \beta=1 $ at as $ \theta_k-\theta_j \to \pm\frac{2}{p}\pi,\pm\frac{2}{q}\pi,\ldots $.
\color{black}

There has been a growing interest in recent years about the asymptotic aspects of structured moment determinants other than the well-known cases of Toeplitz and Hankel. Among those studies are asymptotics of Toeplitz+Hankel determinants \cite{DIK,GI}, and recently the bordered Toeplitz determinants\cite{BEGIL}. A successful Riemann-Hilbert formulation for $2j-k$ and $j-2k$ systems places them among the collection of structured moment determinants for which the large size asymptotics of the determinant and the large degree asymptotics of the corresponding systems of bi-orthogonal polynomials can be investigated. In this regard, the formulation of a ($3\times3$) Riemann-Hilbert problem corresponding to these bi-orthogonal polynomial systems, both as a means of founding the whole theory upon this and deriving key results but also to pave the way for a suitable Deift-Zhou analysis will be the topic of a future publication. At a later stage, the asymptotic description of these determinants can be investigated when the symbol $w$ is of Fisher-Hartwig type, similar to what has been done for Toeplitz \cite{DIK} ,  Hankel \cite{CHankelFH,CG}, and Muttalib-Borodin \cite{C MB FH} determinants. 
  
On a different but possibly related perspective, it would be interesting to search for possible Fredholm determinant representations for the $2j-k$ and $j-2k$ determinants in the same spirit as the Fredholm determinant representations for Toeplitz determinants which could also unveil a Riemann-Hilbert formulation (see \cite{Deift} where the connection was first found for Toeplitz determinants, and also \cite{ItsTracyWidom}). Also looking for the possible $2j-k/j-2k$ analogue of the Fredholm sine-kernel determinant representations for gap probabilities in random matrix theory seems to be another important front to be investigated (see e.g. \cite{Krasovsky} and references therein). Recalling  the last paragraph of \S \ref{Sec Intro}, it would be interesting to ask if from the viewpoint of Operator Theory, the analogue of the (strong) Szeg{\H o} limit theorem can be proven for the $2j-k/j-2k$ determinants (See. e.g. the monographs \cite{BS1,BS}). If this is plausible, we could hope for arriving at yet another example of close interaction between Operator Theory and  Riemann-Hilbert techniques.
  
What is mentioned above is only a short list and obviously the full list of $2j-k/j-2k$ analogues of Toeplitz theory can not be enumerated in full here, 
given the voluminous literature on the Toeplitz side. 
Nevertheless we point out some other obvious candidates:    
Rational approximations to the Carath{\'e}odory functions with a two-point ($z=0$ and $z=\infty$) Hermite-Pad\'e approximation of the associated functions;  
Quadrature problems on the unit circle and Christoffel weights;
The elucidation of an analogue of the CMV matrix\cite{CMV_2003} in the Toeplitz case, i.e. the minimum banded representation of the spectral multiplication operator;
and the analogue of the discrete Fredholm determinant arising from the scattering theory approach to the matrix difference system for the Toeplitz system,
commonly known as the Geronimo-Case-Borodin-Okounkov identity \cite{GC_1979, BO_2000}.

Tasks which we envisage completing progressively include the representation of the derivatives of the polynomials and associated functions in terms
of the fore-mentioned bases. In the context of semi-classical weights this will provide the pair of Lax operators of the integrable system.
We anticipate making progress on the question raised in \cite{ABP+_2014} 
\begin{quote}{"It would be very interesting to determine whether or not there exists a differential equation arising from our formula (5) which plays a role for symplectic and orthogonal types that Painlev\'e III plays for unitary symmetry".}
\end{quote}

\section*{Acknowledgements}
We gratefully acknowledge the American Institute of Mathematics, San Jose, California for the invitation and funding to attend the workshop
{\it Painlev\'e equations and their applications}, 6-10 February, 2017 where this project was initiated. Also we would like to thank Christophe Charlier and Torsten Ehrhardt for their useful suggestions and comments.

\newpage
\section*{List of Symbols}\label{Sec list of symbols}
 
 \footnotesize
\begin{tabular}{lll}
	\textbf{Symbol} & \textbf{Description} & \textbf{Definition} \\
$w_k$ & the $k$-th Fourier coefficient of $w(z)$, $k \in \Z$ & \eqref{Fourier Coeff}\\  
	$\boldsymbol{\mathscr{D}}_r(z,\mathcal{z})$ & The $(n+3)\times(n+3)$ master matrix of $2j-k$ structure with offset $r\in \Z$ &  \eqref{DDD}, \eqref{DDEE}\\
	 $\mathscr{D}_r(z,\mathcal{z})$ & The determinant of $\boldsymbol{\mathscr{D}}_r(z,\mathcal{z})$ & \\
	$\boldsymbol{\mathscr{E}}_s(z,\mathcal{z})$ & The $(n+3)\times(n+3)$  master matrix of $j-2k$ structure with offset $s\in \Z$ &  \eqref{EEE}, \eqref{DDEE}\\
	 $\mathscr{E}_s(z,\mathcal{z})$ & The determinant of $\boldsymbol{\mathscr{E}}_s(z,\mathcal{z})$ & \\
	 $\boldsymbol{D}^{(r)}_{n}$ & The $n\times n$ matrix of $2j-k$ structure with offset $r \in \Z$ & \\
	 	 $D^{(r)}_{n}$ & The determinant of $\boldsymbol{D}^{(r)}_{n}$ &  \eqref{Det}\\
	  $\boldsymbol{E}^{(s)}_{n}$ & The $n\times n$ matrix of $j-2k$ structure with offset $s \in \Z$ & \\
	 $E^{(s)}_{n}$ & The determinant of $\boldsymbol{E}^{(s)}_{n}$ & \eqref{Det E}\\
	 $\mathcal{D}_n[f(\ze)]$ & The $2j-k$ multiple integral with weight $f$ & \eqref{DD} \\
 	 $\mathcal{E}_n[f(\ze)]$ & The $j-2k$ multiple integral with weight $f$ & \eqref{EE} \\
 	 $P_{n}(z;r)$ & The monic $2j-k$ polynomial of the first kind with offset $r \in \Z$ of degree $n$ & \eqref{OP11} \\
   	 $Q_{n}(z;r)$ & The monic $2j-k$ polynomial of the second kind with offset $r \in \Z$ of degree $n$ & \eqref{OP22} \\
 	 $R_{n}(z;s)$ & The monic $j-2k$ polynomial of the first kind with offset $s \in \Z$ of degree $n$ & \eqref{OP11 R} \\
     $S_{n}(z;s)$ & The monic $j-2k$ polynomial of the second kind with offset $s \in \Z$ of degree $n$ & \eqref{OP22 S} \\
   	 $h^{(r)}_{n} $ & The norm of $2j-k$ bi-orthogonal polynomials & \eqref{h}, \eqref{PQorth}-\eqref{OP2}\\
   	 $g^{(s)}_{n} $ & The norm of $j-2k$ bi-orthogonal polynomials& \eqref{H}, \eqref{RSorth}-\eqref{OP2 S}\\
   	 $\boldsymbol{Z}_n(z)$ & The $(n+1)$-vector of monomials  of degrees zero to $n$ & \eqref{Vectors}\\
   	 $\boldsymbol{P}_{n}(z;r)$  & 	The $(n+1)$-vector of $P$-polynomials of degrees zero to $n$ & \eqref{Vectors}\\
   	 $\boldsymbol{Q}_{n}(z;r)$  & 	The $(n+1)$-vector of $Q$-polynomials of degrees zero to $n$ & \eqref{Vectors}\\ 
   	 $\boldsymbol{R}_{n}(z;s)$  & 	The $(n+1)$-vector of $R$-polynomials of degrees zero to $n$ & \eqref{Vectors}\\ 
   	 $\boldsymbol{S}_{n}(z;s)$  & 	The $(n+1)$-vector of $S$-polynomials of degrees zero to $n$ & \eqref{Vectors}\\
   	 $\boldsymbol{h}^{(r)}_{n}$ & The $(n+1)\times(n+1)$ diagonal matrix of $2j-k$ norms & \eqref{h&H diag} \\ 
   	 $\boldsymbol{g}^{(s)}_{n}$ & The $(n+1)\times(n+1)$ diagonal matrix of $j-2k$ norms & \eqref{h&H diag} \\
   	 $\boldsymbol{\mathcal{P}}_{n}^{(r)}$ and $\boldsymbol{\mathcal{Q}}_{n}^{(r)}$ & the $(n+1)\times(n+1)$ lower triangular matrices in the LDU decomposition of $\boldsymbol{D}^{(r)}_{n+1}$ & \eqref{A B}  	\\
 	 $\boldsymbol{\mathcal{R}}_{n}^{(s)}$ and $\boldsymbol{\mathcal{S}}_{n}^{(s)}$ & the $(n+1)\times(n+1)$ lower triangular matrices in the LDU decomposition of $\boldsymbol{E}^{(s)}_{n+1}$ & \eqref{C G}    \\
 	 $p^*(z)$ 	 & The reciprocal polynomial associated to the polynomial $p(z)$ & \eqref{star} \\ 
 	 $\mathcal{p}^{(r)}_{n,\ell}$ & The coefficient of $z^{\ell}$ in $P_{n}(z;r)$ & \eqref{polys} \\
 	  	 $\mathcal{q}^{(r)}_{n,\ell}$ & The coefficient of $z^{\ell}$ in $Q_{n}(z;r)$ & \eqref{polys} \\
 	  	  	 $\mathcal{r}^{(s)}_{n,\ell}$ & The coefficient of $z^{\ell}$ in $R_{n}(z;s)$ & \eqref{polys} \\
 	  	  	  	 $\mathcal{s}^{(s)}_{n,\ell}$ & The coefficient of $z^{\ell}$ in $S_{n}(z;s)$ & \eqref{polys} \\
 	 $K_{n}(z,\mathcal{z};r) $ &  The reproducing Kernel of the $2j-k$ bi-orthogonal polynomials & \eqref{RepKer3} 	\\
 	 $L_{n}(z,\mathcal{z};s) $ &  The reproducing Kernel of the $j-2k$ bi-orthogonal polynomials & \eqref{RepKer3 j-2k} 	     \\
 	 $\de^{(r)}_{n}$ and $\eta^{(r)}_{n}$ & The recurrence coefficients associated to $P_n(z;r)$ & \eqref{delta}, \eqref{eta} \\	 	  	 
 	 $\be^{(r)}_{n}$ and $\al^{(r)}_{n}$ & The recurrence coefficients associated to $Q_n(z;r)$ & \eqref{beta}, \eqref{alpha} \\
 	  	 $\varkappa^{(s)}_{n}$ and $\rho^{(s)}_{n}$ & The recurrence coefficients associated to $R_n(z;s)$ & \eqref{kappa}, \eqref{rho} \\	 	  	 
 	 $\ga^{(s)}_{n}$ and $\theta^{(s)}_{n}$ & The recurrence coefficients associated to $S_n(z;s)$ & \eqref{gamma}, \eqref{theta} \\	
 	 $\mathscr{P}^{(r)}_n$ & The basis for polynomials of degree at most $n$ consisting of $P_m(z;r)$, $m=0,1,\cdots,n$.  & \eqref{PolyCollections} \\
 	  	 $\mathscr{Q}^{(r)}_n$ & The basis for polynomials of degree at most $n$ consisting of $Q_m(z;r)$, $m=0,1,\cdots,n$.  & \eqref{PolyCollections} \\
 	  	  	 $\mathscr{R}^{(s)}_n$ & The basis for polynomials of degree at most $n$ consisting of $R_m(z;s)$, $m=0,1,\cdots,n$.  & \eqref{PolyCollections} \\
 	  	  	  	 $\mathscr{S}^{(s)}_n$ & The basis for polynomials of degree at most $n$ consisting of $S_m(z;s)$, $m=0,1,\cdots,n$.  & \eqref{PolyCollections} \\
 	  	  	  	 $\widecheck{P}_n(z;r)$ and $\widehat{P}_n(z;r)$ & Associated functions corresponding to $P_n(z;r)$ & \eqref{Pcirc}, \eqref{Phat}\\
 	  	  	  	 $\widecheck{Q}_n(z;r)$ and $\widehat{Q}_n(z;r)$ & Associated functions corresponding to $Q_n(z;r)$ & \eqref{Qcirc}, \eqref{Qhat}\\
 	  	         $\widecheck{R}_n(z;s)$ and $\widehat{R}_n(z;s)$ & Associated functions corresponding to $R_n(z;s)$ & \eqref{Rcirc}, \eqref{Rhat}\\
                 $\widecheck{S}_n(z;s)$ and $\widehat{S}_n(z;s)$ & Associated functions corresponding to $S_n(z;s)$ & \eqref{Scirc}, \eqref{Shat}\\  
                 $F_1$ and $F_2$ & 	  Carath{\'e}odory functions & \eqref{F1}, \eqref{F2}  	\\
                 $\mathcal{L}_1$ & Linear difference operator annihilating $P_n(z;r)$ & \eqref{L1}\\
                 $\mathcal{L}_2$ & Linear difference operator annihilating $Q_n(z^{-2};r)$ & \eqref{L2}\\
                 $\mathcal{L}_3$ & Linear difference operator annihilating $R_n(z^2;s)$ & \eqref{L3}\\
                 $\mathcal{L}_4$ & Linear difference operator annihilating $S_n(z^{-1};r)$ & \eqref{L4}\\
                 $\boldsymbol{\mathfrak{P}}_n(z;r)$ & Casorati matrix annihilated by $\mathcal{L}_1$ & \eqref{Cas P}  	 \\
                 $\boldsymbol{\mathfrak{Q}}_n(z;r)$ & Casorati matrix annihilated by $\mathcal{L}_2$ & \eqref{Cas Q}  	 \\
                 $\boldsymbol{\mathfrak{R}}_n(z;s)$ & Casorati matrix annihilated by $\mathcal{L}_3$ & \eqref{Cas R}  	 \\
                 $\boldsymbol{\mathfrak{S}}_n(z;s)$ & Casorati matrix annihilated by $\mathcal{L}_4$ & \eqref{Cas S}  	 \\
\end{tabular}

\color{black}
\normalsize

\bibliographystyle{plain}
\def\cprime{$'$} \def\cprime{$'$} \def\cprime{$'$} \def\cprime{$'$}

\end{document}